\documentclass[12pt]{amsart}

\usepackage{setspace}   

\usepackage[usenames]{color}
\usepackage{fullpage,url,mathrsfs,stmaryrd}
\usepackage{amsmath,amsthm,amssymb,mathrsfs,stmaryrd,color}
\usepackage[utf8]{inputenc}
\usepackage[T1]{fontenc}

\usepackage{relsize}
\usepackage[bbgreekl]{mathbbol}
\usepackage{amsfonts}

\DeclareSymbolFontAlphabet{\mathbb}{AMSb}
\DeclareSymbolFontAlphabet{\mathbbl}{bbold}
\newcommand{\prism}{{\mathlarger{\mathbbl{\Delta}}}}

\newcommand{\hatsharp}{{\hat\sharp}}

\newcommand{\SW}{{{}^s W}}
\newcommand{\SR}{{{}^s \sR}}

\newcommand{\BBT}{\overline\BTst}
\newcommand{\BTst}{\mathscr{B}\!\mathscr{T}}

\newcommand{\BA}{{\mathbb{A}}}

\newcommand{\BG}{{\mathbb{G}}}

\usepackage{amssymb}
\usepackage[all]{xy}
\usepackage{mathrsfs}
\usepackage{enumerate}
\usepackage{amscd}

\usepackage{bbm}

\DeclareMathOperator{\bbig}{{big}}

\DeclareMathOperator{\all}{{all}}
\DeclareMathOperator{\ec}{{ec}}
\DeclareMathOperator{\strong}{{strong}}
\DeclareMathOperator{\weak}{{weak}}

\DeclareMathOperator{\Sm}{{Sm}}

\DeclareMathOperator{\Ob}{{Ob}}

\DeclareMathOperator{\Fil}{{Fil}}

\DeclareMathOperator{\EEnd}{\underline{\on{End}}}
\DeclareMathOperator{\AAut}{\underline{\on{Aut}}}
\DeclareMathOperator{\HHom}{\underline{\on{Hom}}}

\DeclareMathOperator{\PPoly-add}{\underline{\on{Poly-add}}}

\DeclareMathOperator{\MMor}{\underline{\on{Mor}}}

\DeclareMathOperator{\Ab}{{Ab}}

\DeclareMathOperator{\rank}{{rank}}

\newcommand{\bA}{{\mathbb A}}

\newcommand{\cA}{{\mathcal A}}
\newcommand{\cB}{{\mathcal B}}
\newcommand{\cC}{{\mathcal C}}

\newcommand{\cF}{{\mathcal F}}

\newcommand{\cI}{{\mathcal I}}

\newcommand{\cN}{{\mathcal N}}
\newcommand{\cO}{{\mathcal O}}
\newcommand{\cP}{{\mathcal P}}
\newcommand{\cQ}{{\mathcal Q}}

\newcommand{\sG}{{\mathscr G}}

\newcommand{\sP}{{\mathscr P}}

\newcommand{\sR}{{\mathscr R}}

\newcommand{\sX}{{\mathscr X}}
\newcommand{\sY}{{\mathscr Y}}

\newcommand{\fD}{{\mathfrak D}}

\newcommand{\fL}{{\mathfrak L}}

\newcommand{\fZ}{{\mathfrak Z}}

\newcommand{\fg}{{\mathfrak g}}

\newcommand{\nc}{\newcommand}

\nc\wh{\widehat}

\nc\on{\operatorname}

\nc\Gr{\on{Gr}}

\nc\Fl{\on{Fl}}

\newtheorem{cor}[subsubsection]{Corollary}
\newtheorem{lem}[subsubsection]{Lemma}
\newtheorem{prop}[subsubsection]{Proposition}

\newtheorem{conj}[subsubsection]{Conjecture}
\newtheorem{thm}[subsubsection]{Theorem}

\theoremstyle{remark}

\newcommand{\BF}{{\mathbb{F}}}
\newcommand{\BN}{{\mathbb{N}}}
\newcommand{\BQ}{{\mathbb{Q}}}

\newcommand{\BV}{{\mathbb{V}}}
\newcommand{\BZ}{{\mathbb{Z}}}

\DeclareMathOperator{\Isom}{{Isom}}
\DeclareMathOperator{\Lie}{{Lie}}
\DeclareMathOperator{\coLie}{{coLie}}

 \DeclareMathOperator{\Spf}{{Spf}}  

\DeclareMathOperator{\Sym}{{Sym}}
\DeclareMathOperator{\univ}{{univ}}

\DeclareMathOperator{\Cone}{{Cone}}

\DeclareMathOperator{\Ad}{{Ad}}
\DeclareMathOperator{\BP}{{BP}}
\DeclareMathOperator{\BT}{{BT}}
\DeclareMathOperator{\Disp}{{Disp}}
\DeclareMathOperator{\sDisp}{{sDisp}}
\DeclareMathOperator{\DISP}{{DISP}}
\DeclareMathOperator{\preDISP}{{preDISP}}
\DeclareMathOperator{\PREDISP}{{\bf preDISP}}
\DeclareMathOperator{\prr}{{pr}}
\DeclareMathOperator{\Syn}{{Syn}}
\DeclareMathOperator{\Lau}{{Lau}}
\DeclareMathOperator{\diag}{{diag}}

\newcommand{\limto}{{\displaystyle\lim_{\longrightarrow}}}
\newcommand{\rightlim}{\mathop{\limto}}

\newcommand{\leftlim}{\mathop{\displaystyle\lim_{\longleftarrow}}}
\newcommand{\limfromn}{\leftlim\limits_{\raise3pt\hbox{$n$}}}
\newcommand{\limton}{\rightlim\limits_{\raise3pt\hbox{$n$}}}

\newcommand{\rightlimit}[1]{\mathop{\lim\limits_{\longrightarrow}}\limits%
                    _{\raise3pt\hbox{$\scriptstyle #1$}}}

\newcommand{\leftlimit}[1]{\mathop{\lim\limits_{\longleftarrow}}\limits%
                    _{\raise3pt\hbox{$\scriptstyle #1$}}}

\newcommand{\epi}{\twoheadrightarrow}
\newcommand{\iso}{\buildrel{\sim}\over{\longrightarrow}}
\newcommand{\mono}{\hookrightarrow}

\newcommand{\bcdot}{{\textstyle\cdot}}

\DeclareMathOperator{\rig}{{rig}}

\DeclareMathOperator{\Aut}{{Aut}}
\DeclareMathOperator{\Coker}{{Coker}}

\DeclareMathOperator{\End}{{End}}

\DeclareMathOperator{\gr}{{gr}} 
\DeclareMathOperator{\Hom}{{Hom}} 
 
\DeclareMathOperator{\Poly-add}{{Poly-add}}

\DeclareMathOperator{\Ker}{{Ker}} \DeclareMathOperator{\id}{{id}}
\DeclareMathOperator{\im}{{Im}}

\DeclareMathOperator{\Mat}{{Mat}}
\DeclareMathOperator{\Mor}{{Mor}}
 \DeclareMathOperator{\op}{{op}}

\DeclareMathOperator{\Spec}{{Spec}}

\DeclareMathOperator{\Vect}{{Vec}}

\theoremstyle{definition}

\numberwithin{equation}{section}

\newcommand{\Fr}{\operatorname{Fr}}

\newcommand{\hV}{\tilde V}

\newcommand{\bone}{{\boldsymbol{1}}}

\begin{document}
\title[On the Lau group scheme]{On the Lau group scheme}
\author{Vladimir Drinfeld}
\address{University of Chicago, Department of Mathematics, Chicago, IL 60637}

\dedicatory{To Maxim Kontsevich with deepest admiration}

\begin{abstract}
In a 2013 article, Eike Lau constructed a canonical morphism from the stack of $n$-truncated Barsotti-Tate groups over $\BF_p$ to the stack of $n$-truncated displays. He also proved that this morphism is a gerbe banded by a commutative group scheme.
In this paper we describe the group scheme explicitly.

The stack of $n$-truncated Barsotti-Tate groups over $\BF_p$ has a generalization related to any pair $(G,\mu)$, where $G$ is a smooth group scheme over $\BZ/p^n\BZ$ and $\mu$ is a 1-bounded cocharacter of $G$. The same is true for the stack of $n$-truncated displays. We conjecture that in this more general situation the first stack is a gerbe over the second one banded by a commutative group scheme, and we give a conjectural description of this group scheme.

We also give a conjectural description of the stack of $n$-truncated Barsotti-Tate groups over $\Spf\BZ_p$ and of its $(G,\mu)$-generalization.
\end{abstract}




\maketitle

\tableofcontents

\section{Introduction}
Throughout this article, we fix a prime $p$.

Let $d,d'$ be integers such that $0\le d'\le d$. Let $n\in\BN$. 

\subsection{A theorem of E.~Lau}
The notion of $n$-truncated Barsotti-Tate group was introduced by Grothen\-dieck \cite{Gr}. We recall it in \S\ref{s:BT_n groups}.

$n$-truncated Barsotti-Tate groups of height $d$ and dimension $d'\le d$ form an algebraic stack over $\BZ$, denoted by $\BTst_n^{d,d'}$ (see \S\ref{sss:definition of BBT_n} for more details).
This stack is rather mysterious. 

Let $\BBT_n^{d,d'}:=\BTst_n^{d,d'}\otimes\BF_p$. Using  (a covariant version of) crystalline Dieudonn\'e theory,  E.~Lau defined in \cite{Lau13} a canonical morphism $\phi_n:\BBT_n^{d,d'}\to\Disp_n^{d,d'}$, where $\Disp_n^{d,d'}$ is a certain explicit algebraic stack\footnote{$\Disp_n^{d,d'}$ was defined in \cite{Lau13} using Th.~Zink's ideas.} over $\BF_p$ (which is called the stack of $n$-truncated displays of height $d$ and dimension $d'$). 
Moreover, he proved the following

\begin{thm}  \label{t:Lau's theorem}
(i) The morphism $\phi_n:\BBT_n^{d,d'}\to\Disp_n^{d,d'}$ is a gerbe banded by a commutative locally free finite group scheme over $\Disp_n^{d,d'}$, which we denote by $\Lau_n^{d,d'}$. 

(ii) The group scheme $\Lau_n^{d,d'}$ has order $p^{nd'(d-d')}$ and is killed by $F^n$, where $F$ is the geometric Frobenius.
\end{thm}

This theorem of E.~Lau is a combination of \cite[Thm.~B]{Lau13} and \cite[Rem.~4.8]{Lau13}.

\subsection{Main results}
In this article we describe $\Lau_n^{d,d'}$ explictly. We also prove a property of $\Lau_n^{d,d'}$, which we call \emph{$n$-smoothness\footnote{In \cite[\S VI.2]{Gr} Grothendieck used a different name for this property.}.} A group scheme $G$ over an $\BF_p$-scheme $S$ is said to be 
\emph{$n$-smooth} if locally on $S$, there exists an isomorphism of pointed $S$-schemes
\[
G\iso S\times\Spec A_{n,r}, \quad A_{n,r}:=\BF_p[x_1,\ldots, x_r]/(x_1^{p^n}, \ldots , x_r^{p^n})
\]
for some $r\in\BZ_+$ (here $\Spec A_{n,r}$ is viewed as a pointed $\BF_p$-scheme). If $n>1$ then $n$-smoothness is stronger than being a finite locally free group scheme killed by $F^n$.

\subsection{The description of $\Lau_n^{d,d'}$}
The stack $\Disp_n^{d,d'}$ is very explicit: it is the quotient of the $\BF_p$-scheme $GL(d,W_n)$ by an explicit action of a certain group scheme, 
see \S\ref{sss:Disp_n & sDisp_n as quotients} of Appendix~\ref{s:Explicit presentations} (which goes back to \cite{LZ}). In this article we give an explicit description of the group scheme $\Lau_n^{d,d'}$, see Theorem~\ref{t:Lau_n explicitly}.
Namely, $\Lau_n^{d,d'}$ is obtained by applying the \emph{Zink functor} (see \S\ref{s:Zink's functor}) to a certain $n$-truncated \emph{semidisplay} on $\Disp_n^{d,d'}$ (in the sense of \S\ref{s:n-truncated semidisplays}).

Moreover, in \S\ref{ss:the Cartier dual explicitly} of Appendix~\ref{s:Disp_n and Lau_n} we describe the Cartier dual $(\Lau_n^{d,d'})^*$ as an explicit closed subgroup of a very simple smooth group scheme over $\Disp_n^{d,d'}$.
In fact, we describe there the Cartier duals of more general group schemes $\Lau_n^{G,\mu}$ discussed in the next subsection.

\subsection{A ``Shimurian'' generalization of $\Lau_n^{d,d'}$ and two conjectures}    \label{ss:Shimurization} 
For any smooth affine group scheme $G$ over $\BZ/p^n\BZ$ and any 1-bounded\footnote{$1$-boundedness means that all weights of the action of $\BG_m$ on $\Lie (G)$ are $\le 1$.} homomorphism $\mu :\BG_m\to G$, one has certain stacks\footnote{$\Disp_n^{G,\mu}$ was defined in \cite{BP,Lau21}; we recall the definition in Appendix~\ref{s:Disp_n and Lau_n}. For $\BT_n^{G,\mu}$ see \cite{GMM}.}  $\Disp_n^{G,\mu}$ and $\BT_n^{G,\mu}$ (hopefully, they are related to Shimura varieties); in the case $G=GL(d)$ these are the stacks $\Disp_n^{d,d'}$ and $\BTst_n^{d,d'}$, where $d'$ depends on $\mu$.
In Appendix~\ref{s:Disp_n and Lau_n} we formulate Conjecture~\ref{conjecture}, which says that $\BT_n^{G,\mu}\otimes\BF_p$ is a gerbe over $\Disp_n^{G,\mu}$
banded by a certain commutative group scheme $\Lau_n^{G,\mu}$. The definition of $\Lau_n^{G,\mu}$ given in Appendix~\ref{s:Disp_n and Lau_n} is inspired by our description of 
$\Lau_n^{d,d'}$; in fact, $\Lau_n^{d,d'}$ is a particular case of $\Lau_n^{G,\mu}$. Our description of $\Lau_n^{d,d'}$ ``almost'' implies that Conjecture~\ref{conjecture} is true for $G=GL(d)$; for a precise statement, see \S\ref{sss:conjecture}(ii).

We also give a conjectural description of the stack $\BT_n^{G,\mu}$ for any 1-bounded $\mu$ and any $n$, see Conjecture~\ref{conjecture in characteristic p}.

\subsection{Organization}
In \S\ref{s:n-smoothness} we discuss the notion of $n$-smoothness and the Cartier-dual notion of $n$-cosmoothness. In \S\ref{sss:tensor product setting}-\ref{sss:the case m=2} we study the natural tensor structure on the category of 1-cosmooth group schemes; Corollary~\ref{c:Cartier dual of tensor product} is used in the proof of Theorems~\ref{t:main} and \ref{t:Lie of Lau}. 

In \S\ref{ss:BT_n recollections} we recall the notion of $n$-truncated Barsotti-Tate group. In \S\ref{ss:A subgroup of  Aut G} we discuss automorphisms of such groups.

In \S\ref{s:Formulation of main results} we formulate Theorems~\ref{t:main} and \ref{t:Lie of Lau}, which provide some information about $\Lau_n^{d,d'}$ (including $n$-smoothness). To transform Theorem~\ref{t:Lie of Lau} into an explicit description of $\Lau_n^{d,d'}$, we need further steps, which are briefly discussed in \S\ref{sss:Improving}.

In \S\ref{s:proof of main theorem} we prove Theorems~\ref{t:main} and \ref{t:Lie of Lau}.

In \S\ref{s:n-truncated semidisplays} we introduce the stack of $n$-truncated semidisplays and two related stacks (weak and strong $n$-truncated semidisplays).

In \S\ref{s:Zink's functor}  we discuss the Zink functor $\fZ$.

In \S\ref{s:n-cosmooth ones as tensor category} we prove Proposition~\ref{p:Zink functor of a tensor product}, which plays a key role in the proof of Theorem~\ref{t:Lau_n explicitly}.

In \S\ref{s:Explicit description} we formulate and prove Theorem~\ref{t:Lau_n explicitly}, which gives an explicit description of $\Lau_n^{d,d'}$.
The proof uses \cite[Lemma~3.12]{LZ}.

In \S\ref{s:higher displays} we discuss the category of $n$-truncated higher displays from \cite{Lau21}.
We use it to reformulate a certain construction from \S\ref{sss:the semidisplay on Disp} in a way convenient for Appendix~\ref{s:Disp_n and Lau_n}.

In Appendix~\ref{s:hat W} we recall the description of the Cartier dual of the group of Witt vectors ($p$-typical or ``big'' ones).

In Appendix~\ref{s:Explicit presentations} we describe the stacks from \S\ref{s:n-truncated semidisplays} in very explicit terms.

In Appendix~\ref{s:Disp_n and Lau_n} we define the group scheme $\Lau_n^{G,\mu}$ mentioned in \S\ref{ss:Shimurization} and formulate Conjecture~\ref{conjecture}. Using \S\ref{s:higher displays}, we show that
$\Lau_n^{d,d'}$ is a particular case of $\Lau_n^{G,\mu}$. We also describe $(\Lau_n^{G,\mu})^*$ very explicitly.

In Appendix~\ref{s:conjectural description of BT_n}  we formulate Conjecture~\ref{conjecture in characteristic p} describing the stack $\BT_n^{G,\mu}$ for any 1-bounded $\mu$ and any $n$.

\subsection{Simplifications if $n=1$}
A complete description of $\Lau_1^{d,d'}$ is given already by Theorem~\ref{t:Lie of Lau}(ii) combined with \S\ref{sss:shortcut for n=1}.
Moreover, in the $n=1$ case \S\ref{s:n-cosmooth ones as tensor category} and \S\ref{s:higher displays} are unnecessary and \S\ref{s:Zink's functor} is almost unnecessary
(because if $n=1$ then the Zink functor $\fZ$ from \S\ref{s:Zink's functor}  is just the classical functor from restricted Lie algebras to group schemes of height $1$).

\subsection{Acknowledgements}
This paper is based on the theory of displays developed by Th.~Zink and E.~Lau. I also benefited from discussions with them.

On the other hand, Conjecture~\ref{conjecture in characteristic p} is strongly influenced by the theory of sheared prismatization~\cite{BMVZ, Sheared} and by my discussions with D.~Arinkin and N.~Rozenblyum.

The author's work on this project was partially supported by NSF grant DMS-2001425.

\section{$n$-smoothness and $n$-cosmoothness}  \label{s:n-smoothness}
\subsection{Recollections on exact sequences of group schemes}
\subsubsection{Definition of exactness}
Let $S$ be a scheme. A sequence 
\begin{equation}   \label{e:G' to G to G'' bis}
G'\overset{f}\longrightarrow G\overset{h}\longrightarrow G''
\end{equation}
 of affine group $S$-schemes is said to be exact if the corresponding sequence of fpqc-sheaves on the category of $S$-schemes is exact. If $S$ is the spectrum of a field, exactness just means that $\Ker h=\im f$.

\subsubsection{How to check exactness}  \label{sss:2How to check exactness}
Exactness of \eqref{e:G' to G to G'' bis} clearly implies fiberwise exactness. Now suppose we have a diagram \eqref{e:G' to G to G'' bis} such that $h\circ f=0$ and $G',G,G''$ are flat schemes of finite presentation over $S$.  In this situation it is well known (e.g., see \cite[Prop.~1.1]{dJ} and its proof) that if the sequence \eqref{e:G' to G to G'' bis} is fiberwise exact then it is exact, $\Ker h$ is flat (and finitely presented) over $S$ and the morphism $G'\to\Ker h$ is faithfully flat. The latter implies that $\Ker f$ is flat (and finitely presented) over $S$.

From now on, let us assume that the group schemes $G',G,G''$ from the exact sequence~\eqref{e:G' to G to G'' bis} are
finite and locally free. Then $\Ker h$ and $\Ker f$ are also finite and locally free. 
Moreover, the morphism $G/\Ker h\to G''$ is a closed immersion, so if $G''$ is commutative we also have the group scheme $\Coker h$, which is finite and locally free.

Note that exactness of the fiber of \eqref{e:G' to G to G'' bis} over $s\in S$ is equivalent to the condition 
$$|\Ker h_s|=|\im f_s|;$$
this condition is open because $|\Ker h_s|$ is upper-semicontinuous and $|\im f_s|$ is lower-semi\-con\-tinuous (since $|\im f_s|\cdot |\Ker f_s|=|G'_s|$).

If the groups $G,G',G''$ in a complex \eqref{e:G' to G to G'' bis} are commutative then exactness of \eqref{e:G' to G to G'' bis} is equivalent to exactness of the Cartier-dual complex.

\subsection{$n$-smooth group schemes}     \label{ss:n-smooth group schemes bis}
From now on,  \emph{we assume that $S$ is an $\BF_p$-scheme.}

\subsubsection{Pointed $S$-schemes}
By a pointed $S$-scheme we mean an $S$-scheme equipped with a section. We have a forgetful functor from the category of group $S$-schemes to that of pointed $S$-schemes (forget multiplication but remember the unit).

\subsubsection{Definition}   \label{sss:n-smooth group scheme bis}
Let $n\in\BN$. A group $S$-scheme $G$ is said to be \emph{$n$-smooth\footnote{In \cite[\S VI.2]{Gr} Grothendieck used a different name for the class of $n$-smooth  group schemes.}} if Zariski-locally on $S$, there exists an isomorphism of pointed $S$-schemes
\begin{equation}   \label{e:pointed isomorphism}
G\iso S\times\Spec A_{n,r}, \quad A_{n,r}:=\BF_p[x_1,\ldots, x_r]/(x_1^{p^n}, \ldots , x_r^{p^n})
\end{equation}
for some $r\in\BZ_+$ (here $\Spec A_{n,r}$ is viewed as a pointed $\BF_p$-scheme).

If $G$ is an $n$-smooth group $S$-scheme then the $\cO_S$-module $\coLie (G)$ is a vector bundle. Its rank is a locally constant function on $S$, which is called the \emph{rank} of $G$. (This is the number $r$ from~\eqref{e:pointed isomorphism}). It is clear that an $n$-smooth group $S$-scheme is finite and locally free; moreover, its order equals $p^{nr}$, where $r$ is the rank.

The next lemma implies that the property of $n$-smoothness is fpqc-local with respect to~$S$.

\begin{lem}  \label{l:reformulation of n-smoothness}
Let $r\in\BZ_+$ and $G$ a group scheme over an $\BF_p$-scheme $S$. The following are equivalent:

(a) $G$ is $n$-smooth of rank $r$;

(b) $G$ is a finite locally free group scheme of order $p^{nr}$ killed by $F^n$, and $\dim\Lie (G_s)=r$ for all $s\in S$. \qed
\end{lem}

The following variant of Lemma~\ref{l:reformulation of n-smoothness} will be used in \S\ref{sss:deducing main theorem}.

\begin{lem}  \label{l:technical lemma}
Let $r\in\BZ_+$. Let $H$ be a finite group scheme over an $\BF_p$-scheme $S$. Let $H'\subset H$ be a closed subgroup. Suppose that

(i) $H$ is killed by $F^n$;

(ii) $H'$ is a finite locally free group $S$-scheme of order $p^{nr}$;

(iii) $\dim\Lie (H_s)\le r$ for all $s\in S$. 

Then $H'=H$ and $H$ is $n$-smooth of rank $r$. \qed
\end{lem}

\begin{prop}  \label{p:Messing}
A group $S$-scheme $G$ is $n$-smooth if and only if it is finite, locally free, killed by $F^n:G\to (\Fr^n_S)^*G$, and satisfies the following condition: the complex
\begin{equation}   \label{e:F^m followed by F^n-m}
G\overset{F^m}\longrightarrow (\Fr^m_S)^*G\overset{F^{n-m}}\longrightarrow (\Fr^n_S)^*G
\end{equation}
is an exact sequence for every $m\in\{1,\ldots, n-1\}$.
\end{prop}

This is due to W.~Messing \cite[Prop.~II.2.1.2]{Me72} and Grothendieck (see Proposition~2.1 of \cite[Ch.~VI]{Gr}). Here is a slight improvement of the above proposition.

\begin{prop}  \label{p:improved Messing}
Let $G$ be a finite locally free $S$-scheme killed by $F^n:G\to (\Fr^n_S)^*G$. If the complex \eqref{e:F^m followed by F^n-m}
is exact for at least one $m\in\{1,\ldots, n-1\}$ then $G$ is $n$-smooth.
\end{prop}

\begin{proof}
This was proved by Grothendieck (see Proposition~2.1 of \cite[Ch.~VI]{Gr}). On the other hand, the argument from \cite{Me72} works with the following modification: in line 10 of p.29 of \cite{Me72} replace $1\otimes T_1$ by $1\otimes T_1^{p^i}$, where $i=\max (0,n_1-m)$.
\end{proof}

\subsubsection{The case $n=1$}   \label{sss:n=1 case bis}
A group scheme $G$ over $S$ is 1-smooth if and only if it is finite, locally free, and killed by Frobenius. This is a part of Theorem~7.4 of \cite[Exp.VIIA]{SGA3} (more precisely, the equivalence (ii)$\Leftrightarrow$(iv) of the theorem). On the other hand, this follows from Proposition~\ref{p:Messing}.

\subsubsection{The kernel of $F^n$} For a group scheme (or ind-scheme) $G$ over $S$, we set 
\[
G^{(F^n)}:=\Ker (G\overset{F^n}\longrightarrow (\Fr_S^n)^*G).
\]
It is easy to see that $G^{(F^n)}$ is $n$-smooth if $G$ is either smooth or $N$-smooth for some $N\ge n$.

\subsubsection{Notation}   
Let $\Sm_n^{\all} (S)$ be the category of $n$-smooth group schemes over $S$. Let $\Sm_n (S)\subset\Sm_n^{\all} (S)$ be the full subcategory of commutative group schemes. The categories $\Sm_n^{\all} (S)$ form a projective system: for $N\ge n$ the functor $\Sm_N^{\all}  (S)\to \Sm_n^{\all}  (S)$ is $G\mapsto G^{(F^n)}$.
The same is true for the categories $\Sm_n (S)$.

\subsubsection{Relation to formal Lie groups}  \label{sss:2formal groups&n-smooth ones}
By a \emph{formal Lie group} over $S$ we mean a group ind-scheme $G$ over $S$ such that Zarsiski-locally on $S$, there exists an isomorphism of pointed $S$-ind-schemes
$G\iso (\hat\BA^r_S, 0)$ (here $\hat\BA^r_S$ is the formal completion of $\BA^r_S$ along the zero section).

Let $\Sm_\infty^{\all} (S)$ (resp.~$\Sm_\infty^{\all} (S)$) be the category of all (resp.~commutative) formal Lie groups over $S$. If $G\in\Sm_\infty^{\all} (S)$ then $G^{(F^n)}\in\Sm_n^{\all} (S)$.
 As noted in \cite[Ch.~II]{Me72}, this construction defines equivalences
\begin{equation}    \label{e:2formal groups&n-smooth ones}
\Sm_\infty^{\all} (S)\iso  \underset{n}{\underset{\longleftarrow}{\lim}}\,\Sm_n^{\all} (S),\quad \Sm_\infty (S)\iso  \underset{n}{\underset{\longleftarrow}{\lim}}\,\Sm_n (S).
\end{equation}

\subsection{$n$-cosmooth group schemes}     \label{ss:n-cosmooth group schemes}
Recall that every $n$-smooth group scheme over $S$ is finite and locally free.

\subsubsection{Definitions}   \label{sss:n-cosmooth group scheme}
Let $n\in\BN$. A group $S$-scheme $G$ is said to be \emph{$n$-cosmooth} if it is Cartier dual to a commutative $n$-smooth group $S$-scheme. The \emph{rank} of $G$ is defined to be the rank of $G^*$.

\subsubsection{Notation}
The category of $n$-cosmooth group schemes over $S$ is denoted by $\Sm_n^*(S)$. By definition, it is anti-equivalent to $\Sm_n (S)$.

If $S=\Spec R$ we write $\Sm_n(R), \Sm_n^*(R)$ instead of $\Sm_n(S), \Sm_n^*(S)$.

\subsubsection{Remarks}   \label{sss:remarks on n-cosmooth groups}
By definition, any $n$-cosmooth group $S$-scheme $G$ is commutative, finite, and locally free; moreover, its order equals $p^{nr}$, where $r$ is the rank of $G$ (this follows from a similar statement about $n$-smooth group schemes).

\subsubsection{Example} 
Let $m,n\in\BN$. The group $S$-scheme $W_{n,S}^{(F^m)}:=\Ker (W_{n,S}\overset{F^m}\longrightarrow W_{n,S})$ is clearly $m$-smooth. It is also $n$-cosmooth because its Cartier dual is isomorphic to $W_{m,S}^{(F^n)}$.

\subsection{The category $\Sm_1^*(R)$}  \label{ss:1-cosmooth group schemes}
Let $R$ be an $\BF_p$-algebra.

\subsubsection{}  \label{sss:A_Q,phi}
Let $\cB (R)$ be the category of pairs $(P,\varphi )$, where $P$ is a finitely generated projective $R$-module and $\varphi :P\to P$ is a $p$-linear map. 
If $(P,\varphi )\in\cB (R)$ and $\tilde R$ is an $R$-algebra, let
$A_{P,\varphi}(\tilde R)$ be the group of $R$-linear maps $\xi:P\to\tilde R$ such that $\xi(\varphi (x))=\xi (x)^p$ for all $x\in P$. For any $(P,\varphi )\in\cB (R)$ the functor
$A_{P,\varphi}$ is an affine group $R$-scheme. The following theorem is well known (see \S 2 of \cite{dJ}, which refers to \cite[Exp.VIIA]{SGA3}).

\begin{thm}  \label{t:well known}
(i) For any $(P,\varphi )\in\cB (R)$, the group scheme $A_{P,\varphi}$ is 1-cosmooth. Its rank equals the rank of $P$.

(ii) The functor
\[
\cB (R)^{\op}\to\Sm_1^*(R), \quad (P,\varphi )\mapsto A_{P,\varphi}
\]
is an equivalence.

(iii) The inverse functor takes $A\in\Sm_1^*(R)$ to $(P,\varphi)$, where $P=\Hom (A, (\BG_a)_R)$ and $\varphi:P\to P$ is given by composition with 
$F:(\BG_a)_R\to (\BG_a)_R\,$. Equivalently, $(P,\varphi )$ is the restricted Lie algebra of $A^*\in\Sm_1(R)$.
\qed
\end{thm}

\subsubsection{$\Sm_1^*(R)$ as a tensor category}   \label{sss:tensor structure on 1-cosmooth ones}
$\cB (R)$ is clearly a tensor category (i.e., a symmetric monoidal additive category).
So Theorem~\ref{t:well known}(ii) provides a structure of tensor category on $\Sm_1^*(R)$.
We are going to describe this structure directly (without using $\cB (R)$), see Propositions~\ref{p:tensor structure}(ii) and \ref{p:Cartier dual of tensor product}.
A similar tensor structure on $\Sm_n^*(R)$ is briefly mentioned in \S\ref{sss:tensor remarks}(ii-iii) below.

\subsubsection{}   \label{sss:tensor product setting}
Let $(P_i,\varphi_i)\in\cB (R)$, where $1\le i\le m$. Let $(P,\varphi):=\bigotimes\limits_i (P_i,\varphi_i)$. Then we have a poly-additive morphism
\begin{equation}  \label{e:obvious m-additive map}
A_{P_1,\varphi_1}\times\ldots \times A_{P_m,\varphi_m}\to A_{P,\varphi}\, ,
\end{equation}
where the Cartesian product is over $R$: namely, $(\xi_1,\ldots ,\xi_m)\in A_{P_1,\varphi_1}(\tilde R)\times\ldots \times A_{P_m,\varphi_m}(\tilde R)$ goes to the $R$-linear map
\[
P_1\otimes\ldots\otimes P_m\to\tilde R, \quad x_1\otimes\ldots\otimes x_m\mapsto\prod_{i=1}^m \xi_i(x_i).
\]

\begin{prop}   \label{p:tensor structure}
(i) The map
\begin{equation}   \label{e:Lie (f)}
\Hom (A_{P,\varphi}, (\BG_a)_R)\to\Poly-add (A_{P_1,\varphi_1}\times\ldots \times A_{P_m,\varphi_m}, (\BG_a)_R))
\end{equation}
induced by \eqref{e:obvious m-additive map} is an isomorphism (here $\Poly-add$ stands for the group of poly-additive maps).

(ii) For any $A\in\Sm_1^*(R)$, the map
\[
\Hom (A_{P,\varphi}, A)\to\Poly-add (A_{P_1,\varphi_1}\times\ldots \times A_{P_m,\varphi_m}, A)
\]
induced by \eqref{e:obvious m-additive map} is an isomorphism 
\end{prop}

Note that (i) would become false if $\BG_a$ is replaced by $\BG_m$, see \S\ref{sss:the warning} below.

\begin{proof}
To simplify the notation, we assume that $m=2$.

(i) We can assume that the $R$-modules $P_1$ and $P_2$ are free. Let $x_1,\ldots ,x_r$ (resp.~$y_1,\ldots ,y_s$) be a basis in $P_1$ (resp.~in $P_2$). Then the monomials
$x_1^{a_1}\ldots x_l^{a_r}$ with $a_1,\ldots a_r\le p-1$ form a basis in the $R$-module of regular functions on $A_{P_1,\varphi_1}$. Using this fact and a similar fact for $(P_2,\varphi_2)$, 
we see that a bi-additive morphism $A_{P_1,\varphi_1}\times A_{P_2,\varphi_2}\to (\BG_a)_R$ is the same as a polynomial $f\in R[x_1,\ldots, x_r,y_1,\ldots ,y_s]$ which has degree $<p$ with respect to each variable and is bi-additive in the usual sense. Such $f$ is bilinear.

(ii) If $A=A_{P',\varphi'}$ then
\[
\Poly-add (A_{P_1,\varphi_1}\times A_{P_2,\varphi_2}, A_{P',\varphi'})=
\Hom_{R[F]} (P', \Poly-add (A_{P_1,\varphi_1}\times A_{P_2,\varphi_2}, (\BG_a)_R)),
\] 
where  $R[F]:=\End (\BG_a)_R$ acts on $P'$ via $\varphi'$. Similarly,
\[
\Hom (A_{P,\varphi}, A_{P',\varphi'})=
\Hom_{R[F]} (P', \Hom (A_{P,\varphi}, (\BG_a)_R)).
\] 
So statement (ii) follows from (i).
\end{proof}

\subsubsection{}    \label{sss:the warning}
In the situation of \S\ref{sss:tensor product setting}, the map \eqref{e:obvious m-additive map} induces a homomorphism
\begin{equation}  \label{e:not an isomorphism}
A_{P,\varphi}^*=\HHom (A_{P,\varphi}, (\BG_m)_R)\to\PPoly-add (A_{P_1,\varphi_1}\times\ldots \times A_{P_m,\varphi_m}, (\BG_m)_R)),
\end{equation}
where $\PPoly-add$ stands for the group scheme of poly-additive maps. Note that \eqref{e:not an isomorphism} is \emph{not an isomorphism}, in general.
E.g., let $(P_1,\varphi_1)=(P_2,\varphi_2)=(P,\varphi)=(R,0)$. Then
\[
A_{P_1,\varphi_1}=A_{P_2,\varphi_2}=A_{P,\varphi}=(\alpha_p)_R \, , \quad
\HHom ((\alpha_p)_R, (\BG_m)_R)=(\alpha_p)_R, 
\]
\[
\quad \PPoly-add ((\alpha_p)_R\times (\alpha_p)_R, (\BG_m)_R))=\HHom ( (\alpha_p)_R, ( (\alpha_p)_R)^*)=\HHom ( (\alpha_p)_R, (\alpha_p)_R)=(\BG_a)_R\, .
\]

\begin{prop}    \label{p:Cartier dual of tensor product}
The homomorphism \eqref{e:not an isomorphism} induces an isomorphism
\begin{equation}
A_{P,\varphi}^*\iso \PPoly-add (A_{P_1,\varphi_1}\times\ldots \times A_{P_m,\varphi_m}, (\BG_m)_R)^{(F)}.
\end{equation}
As usual, $\PPoly-add (A_{P_1,\varphi_1}\times\ldots \times A_{P_m,\varphi_m}, (\BG_m)_R)^{(F)}$ stands for the kernel of Frobenius in
the group scheme $\PPoly-add (A_{P_1,\varphi_1}\times\ldots \times A_{P_m,\varphi_m}, (\BG_m)_R)$.
\end{prop}

\begin{proof}
Let $H:=\PPoly-add (A_{P_1,\varphi_1}\times\ldots \times A_{P_m,\varphi_m}, (\BG_m)_R)^{(F)}$. Since $A_{P,\varphi}^*$ is killed by $F$, the map \eqref{e:not an isomorphism} induces a homomorphism $f:A_{P,\varphi}^*\to H$. The problem is to show that $f$ is an isomorphism.

(i) By definition, $H$ is killed by $F$. Moreover, $H$ has finite type over $R$ because the scheme parametrizing \emph{all} invertible functions on $A_{P_1,\varphi_1}\times\ldots \times A_{P_m,\varphi_m}$ has finite type over $R$. So $H$ is finite over $R$.

(ii) Assume that $R$ is a field. Then the theory of height $1$ group schemes tells us that to prove that $f:A_{P,\varphi}^*\to H$ is an isomorphism it suffices to show that $\Lie (f)$ is an isomorphism. But $\Lie (f)$ is the map \eqref{e:Lie (f)}, which is an isomorphism by Proposition~\ref{p:tensor structure}(i).

(iii) Let $R$ be any ring. Let $C$ and $C'$ be the coordinate rings of the finite schemes $A_{P,\varphi}^*$ and $H$. Then $f^*:C'\to C$ is a homomorphism of finitely generated $R$-modules. By (i), $f^*$ is a fiberwise isomorphism. But $C$ is projective, so $f^*$ is an isomorphism.
\end{proof}

\begin{cor}   \label{c:Cartier dual of tensor product}
For any $A_1,\ldots ,A_m\in\Sm_1^*(R)$ the group scheme 
$$\PPoly-add (A_1\times\ldots \times A_m, (\BG_m)_R)^{(F)}$$ 
is $1$-smooth. Its rank equals the product of the ranks of $A_1,\ldots ,A_m$.
One has a canonical isomorphism of restricted Lie $R$-algebras\footnote{The commutator in these restricted Lie algebras is, of course, zero.}
$\Lie\PPoly-add (A_1\times\ldots \times A_m, (\BG_m)_R)^{(F)}\iso\bigotimes\limits_i\Lie (A_i^*)$.
\end{cor}

\subsubsection{Remarks}    \label{sss:the case m=2}
(i) In the case $m=2$ one has $$\PPoly-add (A_1\times A_2, (\BG_m)_R)^{(F)}=\HHom (A_1,A_2^*)^{(F)}.$$

(ii) If $A_1, A_2\in\Sm_1^*(R)$ then the group scheme $\HHom (A_1,A_2^*)$ does not have to be either finite or flat, see \S\ref{sss:neither finite nor flat nor commutative}-\ref{sss:neither finite nor flat} below.

\subsection{A general principle}   \label{ss:n-cosmooth easier than n-smooth}
Objects of $\Sm_n^*(R)$ (i.e., $n$-cosmooth group schemes) are easier to handle than objects of the dual category $\Sm_n(R)$.

E.g., the 1-cosmooth group scheme $A_{P,\varphi}$ from \S\ref{sss:A_Q,phi} is a subgroup of the very simple group scheme $\BV (P):=\Spec\Sym (P)$, and this subgroup is defined by very simple equations.
On the other hand, $A_{P,\varphi}^*$ is a \emph{quotient} of the group \emph{ind}-scheme $\BV (P)^*$; this is less elementary. Note that $\BV (P)^*$ is the PD-nilpotent PD neighborhood of zero in $\BV (P^*)$.

\section{$n$-truncated Barsotti-Tate groups}     \label{s:BT_n groups}
\subsection{Recollections}   \label{ss:BT_n recollections}

\subsubsection{Definition of $n$-truncated Barsotti-Tate group}
The notion of $n$-truncated Barsotti-Tate group was introduced by Grothendieck \cite{Gr}. It is reviewed in \cite{Me72,Il,dJ}.

Let $n\in\BN$. By definition, an $n$-truncated Barsotti-Tate group (in short, a $\BT_n$ group) over a scheme $S$ is a finite locally free commutative group scheme $G$ over $S$ which is killed by $p^n$, is $(\BZ/p^n\BZ)$-flat (as an fpqc sheaf) and satisfies an extra condition in the case $n=1$. If  $n=1$ and $S$ is an $\BF_p$-scheme we have complexes
\begin{equation}   \label{e:the complexes for n=1}
G\overset{F}\longrightarrow\Fr_S^*G\overset{V}\longrightarrow G, \quad \Fr_S^*G\overset{V}\longrightarrow G\overset{F}\longrightarrow \Fr_S^*G,
\end{equation}
and these complexes are required to be exact sequences; as noted by Grothendieck, exactness of one of them implies exactness of the other.\footnote{See \cite[\S 1.3 (b)]{Il}. The idea is to use \S\ref{sss:2How to check exactness} and look at the orders of the kernels and images.} If $n=1$ and $S$ is any scheme then the above condition is required for the restriction of $G$ to $S\otimes\BF_p\,$.

The above-mentioned $(\BZ/p^n\BZ)$-flatness condition is equivalent to exactness of the sequences
\[
G\overset{p^{n-m}}\longrightarrow G\overset{p^m}\longrightarrow G
\]
for all $m<n$. By \S\ref{sss:2How to check exactness}, the group scheme $\Ker (G\overset{p^m}\longrightarrow G)$ is finite and locally free over $S$. Moreover, it is a $\BT_m$ group (even if $m=1$), see \cite[III.3.3]{Gr} or \cite[II.3.3.11]{Me72}.

If $G$ is a $\BT_n$ group then so is its Cartier dual $G^*$.

The groupoid of $\BT_n$ groups over $S$ will be denoted by $\BT_n(S)$. If $S=\Spec R$ we write $\BT_n(R)$ instead of $\BT_n(S)$.

\subsubsection{Height and dimension}   \label{sss:Height and dimension}
From now on, \emph{we assume that $S$ is an $\BF_p$-scheme.} If $G$ is a $\BT_n$ group over $S$ then the group schemes $\Ker (G\overset{p}\longrightarrow G)$ and 
$\Ker (G\overset{F}\longrightarrow \Fr_S^*G)$ are finite and locally free by \S\ref{sss:2How to check exactness}, and they are killed by $p$. So there exist  locally constant functions $d:S\to\BZ_+$ and $d':S\to\BZ_+$ such that the orders of these groups equal $p^d$ and $p^{d'}$, respectively; clearly $d'\le d$. Names: $d$ is the \emph{height} of $G$, and $d'$ is the \emph{dimension} of $G$.

The group scheme $G^{(F)}:=\Ker (G\overset{F}\longrightarrow \Fr_S^*G)$ is 1-smooth (see \S\ref{sss:n=1 case bis}), so its Lie algebra is  a vector bundle of rank $d'$. The Lie algebra of $G$ is the same.

\subsubsection{Some results of Messing}  \label{sss:results of Messing}
Let $G\in\BT_n(S)$. Similarly to~\eqref{e:the complexes for n=1}, we have a complex
\begin{equation}   \label{e:the complexes for any n}
G\overset{F^n}\longrightarrow (\Fr^n_S)^*G\overset{V^n}\longrightarrow G\overset{F^n}\longrightarrow (\Fr^n_S)^*G.
\end{equation}
By \cite[II.3.3.11 (b)]{Me72}, this complex is exact. So it gives rise to exact sequences of finite locally free group $S$-schemes
\begin{equation}  \label{e:extension of n-cosmooth by n-smooth}
0\to G^{(F^n)}\to G\to G/G^{(F^n)}\to 0,
\end{equation}
\begin{equation}   \label{e:extension of n-smooth by n-cosmooth}
0\to G/G^{(F^n)}\to (\Fr^n_S)^*G\to G^{(F^n)}\to 0.
\end{equation}
Moreover, $G^{(F^n)}$ is $n$-smooth (see \cite[II.3.3.11 (a)]{Me72} and \cite[II.2.1.2]{Me72}). On the other hand, $(G/G^{(F^n)})^*=(G^*)^{(F^n)}$, so $G/G^{(F^n)}$ is $n$-cosmooth.

Thus \eqref{e:extension of n-cosmooth by n-smooth} exhibits $G$ as an extension of an $n$-cosmooth group scheme by an $n$-smooth one, and \eqref{e:extension of n-smooth by n-cosmooth}
exhibits $(\Fr^n_S)^*G$ as an extension of an $n$-smooth group scheme by an $n$-cosmooth one.

\subsubsection{The stacks $\BTst_n^{d,d'}$ and $\BBT_n^{d,d'}$}   \label{sss:definition of BBT_n}
For any scheme $S$, let $\BTst_n^{d,d'}(S)$ be the groupoid of $\BT_n$ groups over $S$ whose restriction to $S\otimes\BF_p$ has height $d$ and dimension $d'$. The assignment $S\mapsto\BTst_n^{d,d'}(S)$ is an algebraic stack of finite type over $\BZ$ with affine diagonal, see \cite[Prop.~1.8]{Wedhorn}.
Moreover, the stack $\BTst_n^{d,d'}$ is \emph{smooth} over $\BZ$ by a deep theorem of Grothendieck, whose proof is given in Illusie's article~\cite{Il}.

\subsection{A subgroup of  $\Aut G$, where $G\in \BT_n (S)$ and $S$ is over $\BF_p$}   \label{ss:A subgroup of  Aut G}
\subsubsection{}    \label{sss:A subgroup of  Aut G}
As before, let $S$ be an $\BF_p$-scheme and $G\in\BT_n (S)$. Then one has a monomorphism $\Hom (G/G^{(F^n)}, G^{(F^n)})\mono\Aut G$ defined by $f\mapsto 1+\tilde f$, where $\tilde f$ is the composite morphism
\[
G\epi G/G^{(F^n)}\overset{f}\longrightarrow G^{(F^n)}\mono G.
\]
Thus we get a homomorphism 
\[
\HHom (G/G^{(F^n)}, G^{(F^n)})\mono\AAut G,
\]
which is a closed immersion.

\subsubsection{Example}  \label{sss:neither finite nor flat nor commutative}
Suppose that $S=\Spec \bar\BF_p$ and $G=\Ker (E\overset{p}\longrightarrow E)$, where $E$ is an elliptic curve over $S$. If $E$ is not supersingular then 
\[
\HHom (G/G^{(F)}, G^{(F)})\simeq\mu_p, \quad \AAut G=(\BF_p^\times\times\BF_p^\times)\ltimes\HHom (G/G^{(F)}, G^{(F)}),
\]
where the action of $\BF_p^\times\times\BF_p^\times$ on $\HHom (G/G^{(F)}, G^{(F)})$ is nontrivial for $p>2$.

Now suppose that $E$ is supersingular. Then 
\[
\HHom (G/G^{(F)}, G^{(F)})\simeq \HHom (\alpha_p, \alpha_p)=\BG_a ,
\]
\begin{equation}   \label{e:supersingular}
\quad \AAut G=\BF_{p^2}^\times\ltimes \HHom (G/G^{(F)}, G^{(F)}),
\end{equation}
where the action of $\BF_{p^2}^\times$ on $\HHom (G/G^{(F)}, G^{(F)})$ is nontrivial. Here is a way to prove \eqref{e:supersingular}. First, 
$\AAut G\supset \HHom (G/G^{(F)}, G^{(F)})$. Second, classical Dieudonn\'e theory describes the $\bar\BF_p$-points of $\AAut G$. So to prove \eqref{e:supersingular} it remains to show that
$\dim\Lie(\AAut G)=1$. Indeed, it is well known and easy to prove that for any commutative finite locally free group scheme $G$ over any ring one has
\[
\Lie(\AAut G)=\Hom (G, \Lie (G))=\Hom (G, \BG_a)\otimes \Lie (G)= \Lie (G^*)\otimes\Lie (G).
\]

\subsubsection{Remark}  \label{sss:neither finite nor flat}
The above example implies that $\HHom (G/G^{(F^n)}, G^{(F^n)})$ is neither finite nor flat, in general.

\section{Formulating Theorems~\ref{t:main} and \ref{t:Lie of Lau}}  \label{s:Formulation of main results}
\subsection{The group scheme $\cI_n^{d,d'}$}
Let $\cI_n^{d,d'}$ be the inertia stack of $\BBT_n^{d,d'}$. Thus an $S$-point of $\cI_n^{d,d'}$ is a pair consisting of an $S$-point of $\BBT_n^{d,d'}$ and an automorphism of this point.

It is clear that $\cI_n^{d,d'}$ is an affine group scheme of finite type over the stack $\BBT_n^{d,d'}$.

\subsection{The first main theorem}
\subsubsection{The group scheme $\Lau_n^{d,d'}$} \label{sss:Lau_n}
E.~Lau defined in \cite[\S 4]{Lau13} a canonical morphism 
\[
\phi_n:\BBT_n^{d,d'}\to\Disp_n^{d,d'}.
\]
He proved that $\phi_n$ is a gerbe banded by a commutative locally free finite group scheme over $\Disp_n^{d,d'}$, which we denote by $\Lau_n^{d,d'}$. He also proved some properties of $\Lau_n^{d,d'}$; we formulated them in Theorem~\ref{t:Lau's theorem}.

Note that $\phi_n^*\Lau_n$ is a subgroup of $\cI_n^{d,d'}$. This subgroup is closed because $\Lau_n$ is finite over $\Disp_n^{d,d'}$ and the morphism 
$\cI_n^{d,d'}\to\BBT_n^{d,d'}$ is separated.

\begin{thm} \label{t:main}
(i) $\phi_n^*\Lau_n^{d,d'}=(\cI_n^{d,d'})^{(F^n)}$, where 
$$(\cI_n^{d,d'})^{(F^n)}:=\Ker (\cI_n^{d,d'}\overset{F^n}\longrightarrow (\Fr^n)^*\cI_n^{d,d'}).$$

(ii) The group scheme $\Lau_n^{d,d'}$ is $n$-smooth of rank $d'(d-d')$.
\end{thm}

The proof will be given in \S\ref{s:proof of main theorem}. 

The above theorem and Theorem~\ref{t:Lie of Lau} below are steps towards an explicit description of $\Lau_n^{d,d'}$.

\subsubsection{Remarks}
(i) Theorem~\ref{t:main}(i) improves Proposition 4.2 of \cite{LZ}.

(ii) Theorem~\ref{t:main} implies that the group scheme $(\cI_n^{d,d'})^{(F^n)}$ is finite, locally free, and commutative.
On the other hand, by \S\ref{sss:neither finite nor flat nor commutative}, the group scheme $\cI_1^{2,1}$ is neither finite nor flat, nor commutative.

\subsection{Reconstructing $\Lau_n^{d,d'}$ from $\phi_n^*\Lau_n^{d,d'}$} \label{ss:Reconstructing}
\begin{lem}   \label{l:fully faithful}
Let $\phi :\sX\to\sY$ be a morphism of algebraic stacks. If $\phi$ is an fppf gerbe then the functor
\begin{equation} \label{e:fully faithful}
\phi^*:\{\mbox{Group schemes over }\sY\}\to\{\mbox{Group schemes over }\sX\}
\end{equation}
 is fully faithful.
\end{lem}

(In fact, the lemma and its proof given below remain valid if the words ``group scheme" are replaced by ``scheme'', ``vector bundle'', etc.)

\begin{proof}
We can assume that the gerbe is trivial, so $\sX$ is the classifying stack of a flat group scheme $H$ over $\sY$. Then a group scheme over $\sX$ is the same as a group
scheme over $\sY$ equipped with an action of $H$, and the functor $\phi^*$ takes a group scheme over $\sY$ to the same group scheme equipped with the trivial action of $H$.
\end{proof}

By Lemma~\ref{l:fully faithful}, $\Lau_n^{d,d'}$ can be reconstructed from $\phi_n^*\Lau_n^{d,d'}$, in principle. In practice, this requires some work.

\subsubsection{The essential image of the functor \eqref{e:fully faithful}}   \label{sss:essential image}
In the situation of Lemma~\ref{l:fully faithful},
let $\cI_{\sX/\sY}$ be the relative inertia stack; this is the group scheme over $\sX$ defined by
$$\cI_{\sX/\sY}:=\Ker (\cI_{\sX}\to\phi^*\cI_{\sY}),$$
 where $\cI_{\sX}$ and $\cI_{\sY}$ are the inertia stacks of $\sX$ and $\sY$.  The proof of Lemma~\ref{l:fully faithful} implies the following description of the essential image of the functor \eqref{e:fully faithful} (assuming that $\phi :\sX\to\sY$ is a gerbe):
a group scheme $H$ over $\sX$ belongs to the essential image of \eqref{e:fully faithful} if and only if the canonical action of $\cI_{\sX/\sY}$ on $H$ is trivial.

\subsection{A complement to Theorem~\ref{t:main}}  \label{ss:Lie of Lau}
\subsubsection{The group schemes $\tilde\cA _n^{d,d'}, \tilde\cB_n^{d,d'}$}
Let $S$ be an $\BF_p$-scheme and $G$ a $\BT_n$ group scheme over $S$ of height $d$ and dimension $d'$.
Then $G^{(F^n)}$ and $(G/G^{(F^n)})^*=(G^*)^{(F^n)}$ are $n$-smooth group schemes over $S$ (see \S\ref{sss:results of Messing}); their ranks equal $d'$ and $d-d'$, respectively. As $S$ and $G$ vary, $G^{(F^n)}$ and $(G^*)^{(F^n)}$ determine commutative $n$-smooth group schemes $\tilde\cA _n^{d,d'}$ and $\tilde\cB_n^{d,d'}$ over 
$\BBT_n^{d,d'}$, whose ranks equal $d'$ and $d-d'$, respectively.

\begin{thm}   \label{t:Lie of Lau}
(i) $\tilde\cA _n^{d,d'}$ and $\tilde\cB_n^{d,d'}$ descend\,\footnote{By Lemma~\ref{l:fully faithful}, the descent is unique up to unique isomorphism.} to commutative $n$-smooth group schemes $\cA _n^{d,d'}$ and $\cB_n^{d,d'}$ over $\Disp_n^{d,d'}$.

(ii) One has a canonical isomorphism of (commutative) restricted Lie algebras
\[
\Lie (\Lau_n^{d,d'})\iso \Lie (\cA _n^{d,d'})\otimes \Lie (\cB_n^{d,d'}).
\]

(iii) One has a canonical isomorphism
\begin{equation}    \label{e:preliminary description of Lau_n}
\Lau_n^{d,d'}\iso \HHom((\cB _n^{d,d'})^*, \cA _n^{d,d'})^{(F^n)},
\end{equation}
where as usual, 
$$\HHom(((\cB _n^{d,d'})^*, \cA _n^{d,d'})^{(F^n)}:=\Ker (F^n:\HHom(((\cB _n^{d,d'})^*, \cA _n^{d,d'})\longrightarrow (\Fr^n)^*\HHom(\cB _n^*, \cA _n^{d,d'})).$$ 
\end{thm}

The proof will be given in \S\ref{ss:proof of Lie of Lau}. Note that Theorem~\ref{t:Lie of Lau}(ii) is really informative if $n=1$ (indeed, the group scheme $\Lau_n^{d,d'}$ is $1$-smooth, so it is uniquely determined by its restricted Lie algebra).

\subsection{Improving formula~\eqref{e:preliminary description of Lau_n}}  \label{sss:Improving}
Eventually, we will transform formula~\eqref{e:preliminary description of Lau_n} into an \emph{explicit} description of $\Lau_n^{d,d'}$ (see Theorem~\ref{t:Lau_n explicitly}) using two ingredients. The first one is an explicit description (due to E.~Lau and T.~Zink) of the group schemes $\cA _n^{d,d'}$ and $\cB_n^{d,d'}$ defined in Theorem~\ref{t:Lie of Lau}(i). The other ingredient is Proposition~\ref{p:Cartier dual of tensor product} for $n=1$ and Proposition~\ref{p:Zink functor of a tensor product} for an arbitrary $n$. In the case $n=1$ Proposition~\ref{p:Cartier dual of tensor product} allows one to rewrite the r.h.s of \eqref{e:preliminary description of Lau_n} as a tensor product of $\cA _1^{d,d'}$ and $\cB_1^{d,d'}$ in the sense of the tensor structure on $\Sm_1^*(R)$ defined in \S\ref{sss:tensor structure on 1-cosmooth ones}. For $n>1$ we will do something similar.

\section{Proofs of Theorems~\ref{t:main} and \ref{t:Lie of Lau}}    \label{s:proof of main theorem}
\subsection{Deducing Theorem~\ref{t:main} from the key lemma}  \label{ss:formulating key lemma}
Here is the key lemma, which will be proved in \S\ref{ss:proof of key lemma}.

\begin{lem}   \label{l:key lemma}
The Lie algebra of the fiber of $(\cI_n^{d,d'})^{(F^n)}$ over any geometric point of $\BBT_n^{d,d'}$ has dimension $d'(d-d')$.
\end{lem}

\subsubsection{Deducing Theorem~\ref{t:main} from the key lemma}  \label{sss:deducing main theorem}
$\phi_n^*\Lau_n^{d,d'}$ is a closed subgroup of the group scheme $\cI_n^{d,d'}$. Moreover, $\phi_n^*\Lau_n^{d,d'}\subset (\cI_n^{d,d'})^{(F^n)}$ by Theorem~\ref{t:Lau's theorem}(ii). 
$(\cI_n^{d,d'})^{(F^n)}$ is finite over $\BBT_n^{d,d'}$ because $\cI_n^{d,d'}$ has finite type. By Theorem~\ref{t:Lau's theorem}(ii), $\phi_n^*\Lau_n^{d,d'}$ is finite and locally free 
over~$\BBT_n^{d,d'}$ of order $p^{nd'(d-d')}$. So Theorem~\ref{t:main} follows from Lemma~\ref{l:key lemma} combined with Lemma~\ref{l:technical lemma} (the latter has to be applied to the pullbacks of $\phi_n^*\Lau_n^{d,d'}$ and $(\cI_n^{d,d'})^{(F^n)}$ to a scheme $S$ equipped with a faithfully flat morphism to $\BBT_n^{d,d'}$).

\subsection{A description of $(\cI_n^{d,d'})^{(F^n)}$}  \label{ss:a more understandable description}
The goal of this subsection is to prove Corollary~\ref{c:a more understandable description}, which can be regarded as a description of $(\cI_n^{d,d'})^{(F^n)}$.

Let $S$ be an $\BF_p$-scheme and $G\in\BT_n(S)$. 

\begin{lem}  \label{l:F^n-killed endomorphisms}
If $f\in\End G$ and $(\Fr^n)^*(f)=0$ then $f:G\to G$ factors as
\[
G\epi G/G^{(F^n)}\to G^{(F^n)}\mono G.
\]
In particular, $f^2=0$.
\end{lem}

\begin{proof} 
One has $F^n\circ f=(\Fr^n)^*(f)\circ F^n=0$ and  $f\circ V^n=V^n\circ (\Fr^n)^*(f)=0$.
So $\im f\subset G^{(F^n)}$ and $\Ker (f)\supset \im ((\Fr^n)^*G\overset{V^n}\longrightarrow G)=G^{(F^n)}$.
\end{proof}

\begin{cor}  \label{c:a more understandable description}
One has group isomorphisms
\[
(\AAut G)^{(F^n)}\iso (\EEnd G)^{(F^n)}\iso\HHom(G/G^{(F^n)}, G^{(F^n)})^{(F^n)},
\] 
where the first map is $h\mapsto h-1$. 
As usual, $\HHom(G/G^{(F^n)}, G^{(F^n)})^{(F^n)}$ denotes the kernel of 
$F^n:\HHom(G/G^{(F^n)}, G^{(F^n)})\longrightarrow (\Fr^n)^*\HHom(G/G^{(F^n)}, G^{(F^n)})$. \qed
\end{cor}
Note that by \S\ref{sss:neither finite nor flat nor commutative}, the group scheme $\HHom(G/G^{(F^n)}, G^{(F^n)})$ is neither finite nor flat, in general.

\subsection{The case $n=1$}   \label{ss:The case n=1}
In this case, Lemma~\ref{l:key lemma} follows from the next one.

\begin{lem}  \label{l:The case n=1}
Let $G\in\BT_1(S)$. Then the group scheme $(\AAut G)^{(F)}$ is $1$-smooth (in particular, finite and locally free), and one has a canonical isomorphism
of (commutative) restricted Lie $\cO_S$-algebras
$$\Lie ((\AAut G)^{(F)})\iso \Lie (G^{(F)})\otimes\Lie ((G/G^{(F)})^*)=\Lie (G^{(F)})\otimes\Lie ((G^*)^{(F)}) .$$
\end{lem}

\begin{proof}
Use Corollary~\ref{c:a more understandable description}, Corollary~\ref{c:Cartier dual of tensor product}, and \S\ref{sss:the case m=2}(i).
\end{proof}

\subsection{On $\End G_m$, where $m\le n$ and $G_m:=\Ker (p^m:G\to G)$} 
As before, let $G\in\BT_n(S)$, where $S$ is an $\BF_p$-scheme.
Let $m\le n$ and 
\[
G_m:=\Ker (p^m:G\to G).
\]
The subgroup $G_m\subset G$ can also be regarded as a quotient of $G$ via $p^{n-m}:G\epi G_m$. We have an \emph{additive} homomorphism
\begin{equation}  \label{e:End G_m mono End G}
\End G_m\to\End G, \quad f\mapsto \tilde f,
\end{equation}
where $\tilde f:G\to G$ is the composition $G\,\overset{p^{n-m}}\epi \!\! G_m\overset{f}\longrightarrow G_m\mono G$. 
The map \eqref{e:End G_m mono End G} induces additive isomorphisms
\begin{equation}  \label{e:End G_m as subgroup of End G}
\End G_m\iso\Ker (\End G \overset{p^m}\longrightarrow\End G),
\end{equation}
\begin{equation}  \label{e:End G_m as subgroup of EEnd G}
\EEnd G_m\iso\Ker (\EEnd G \overset{p^m}\longrightarrow\EEnd G).
\end{equation}

Moreover, \eqref{e:End G_m as subgroup of EEnd G} induces an additive isomorphism
\begin{equation}  \label{e:EEnd G_m as subgroup of EEnd G}
(\EEnd G_m)^{(F^m)}\iso (\EEnd G)^{(F^m)}
\end{equation}
because $(\EEnd G)^{(F^m)}\subset\Ker (\EEnd G \overset{p^m}\longrightarrow\EEnd G)$.

Using the group isomorphism
\[
(\EEnd G)^{(F^n)}\iso (\AAut G)^{(F^n)}, \quad f\mapsto 1+f
\]
and a similar group isomorphism $(\EEnd G_m)^{(F^m)}\iso (\AAut G_m)^{(F^m)}$, we get from \eqref{e:EEnd G_m as subgroup of EEnd G}
a group isomorphism
\begin{equation}  \label{e:AAut G_m as subgroup of AAut G}
(\AAut G_m)^{(F^m)}\iso (\AAut G)^{(F^m)}.
\end{equation}

\subsection{Proof of Lemma~\ref{l:key lemma}} \label{ss:proof of key lemma}
The problem is to show that for any field $k$ of characteristic $p$ and any $G\in\BT_n(k)$, one has
\begin{equation}  \label{e:dim=d'(d-d')}
\dim\Lie ((\AAut G)^{(F^n)})=d'(d-d'),
\end{equation}
where $d,d'$ are the height and the dimension of $G$. Note that
$(\AAut G)^{(F^n)}$ and $(\AAut G)^{(F)}$ have the same Lie algebra. So applying \eqref{e:AAut G_m as subgroup of AAut G} for $m=1$, we reduce the proof of \eqref{e:dim=d'(d-d')} to the case $n=1$, which was treated in \S\ref{ss:The case n=1}.

\subsection{Proof of Theorem~\ref{t:Lie of Lau}} \label{ss:proof of Lie of Lau}
\subsubsection{Proof of Theorem~\ref{t:Lie of Lau}(i)}
By \S\ref{sss:essential image} and Theorem~\ref{t:main}(i), it suffices to show that the action of 
$(\cI_n^{d,d'})^{(F^n)}$ on $\tilde\cA_n^{d,d'}$ and $\tilde\cB_n^{d,d'}$ is trivial. This follows from Corollary~\ref{c:a more understandable description}. \qed

\subsubsection{Proof of Theorem~\ref{t:Lie of Lau}(ii)}
By Theorem~\ref{t:main}(i), $\phi_n^*\Lau_n^{d,d'}=(\cI_n^{d,d'})^{(F^n)}$.
So Lemma~\ref{l:fully faithful} implies that proving Theorem~\ref{t:Lie of Lau}(ii) amounts to constructing an isomorphism
$$\Lie ((\cI_n^{d,d'})^{(F^n)}) \iso\Lie (\tilde\cA_n^{d,d'})\otimes \Lie (\tilde\cB_n^{d,d'}).$$
For $n=1$, such an isomorphism is provided by Lemma~\ref{l:The case n=1}. It remains to note that the pullbacks of
$(\cI_1^{d,d'})^{(F)}$, $\tilde\cA_1^{d,d'}$, $\tilde\cB_1^{d,d'}$ via the morphism $\BBT_n^{d,d'}\to\BBT_1^{d,d'}$ are canonically isomorphic to
$ (\cI_n^{d,d'})^{(F)}$, $(\tilde\cA_n^{d,d'})^{(F)}$, $(\tilde\cB_n^{d,d'})^{(F)}$ (in the case of $(\cI_1^{d,d'})^{(F)}$ this is the iso\-mor\-phism~\eqref{e:AAut G_m as subgroup of AAut G}). \qed

\subsubsection{Proof of Theorem~\ref{t:Lie of Lau}(iii)}
Combining Theorem~\ref{t:main}(i) and Corollary~\ref{c:a more understandable description}, we get a canonical isomorphism
$\phi_n^*\Lau_n^{d,d'}\iso \HHom((\tilde\cB _n^{d,d'})^*, \tilde\cA _n^{d,d'})^{(F^n)}$.
It remains to use Lemma~\ref{l:fully faithful}.
\qed

\section{$n$-truncated semidisplays}   \label{s:n-truncated semidisplays}
Let $R$ be an $\BF_p$-algebra and $n\in\BN$. Let $\Disp_n(R)$ be the additive category  of {$n$-truncated displays in the sense of \cite[Def.~3.4]{Lau13}. 
In \S\ref{ss:n-truncated semidisplays}-\ref{ss:n-truncated displays} we will construct a diagram of additive categories
\[
\Disp_n(R)\mono\sDisp_n^{\strong}(R)\to\sDisp_n(R)\to\sDisp_n^{\weak}(R)
\]
in which the first functor is fully faithful and the other functors are essentially surjective;
$\sDisp_n(R)$ and $\sDisp_n^{\weak}(R)$ are defined in \S\ref{ss:n-truncated semidisplays},
$\Disp_n(R)$ and $\sDisp_n^{\strong}(R)$ are discussed in \S\ref{ss:n-truncated displays}.
Objects of $\sDisp_n(R)$ are called $n$-truncated \emph{semidisplays}.
In this paper $\sDisp_n(R)$ is the main player. Introducing $\sDisp_n^{\weak}(R)$ is motivated by \S\ref{sss:1-truncated semidisplays}.

Unlike $\Disp_n(R)$, the categories $\sDisp_n(R)$, $\sDisp_n^{\weak}(R)$, and  $\sDisp_n^{\strong}(R)$ are tensor categories, see~\S\ref{ss:tensoring truncated semidisplays} and \S\ref{sss:how to deal}.

The category $\Disp_n(R)$ is equipped with a duality functor due to E.~Lau (see \S\ref{sss:duality on DISP_n and Disp_n} below).
There is no such functor on $\sDisp_n(R)$, $\sDisp_n^{\weak}(R)$, or $\sDisp_n^{\strong}(R)$.

In \S\ref{ss:truncated Zink's functor}-\ref{ss:Cartier dual of Z_P} we discuss a functor $$\sDisp_n(R)\to\Sm_n(R),$$
where $\Sm_n(R)$ is the category of commutative $n$-smooth group schemes over $R$.
The corresponding functor $\Disp_n(R)\to\Sm_n(R)$ was defined in \cite{LZ}.

\subsection{$n$-truncated semidisplays}   \label{ss:n-truncated semidisplays}
\subsubsection{Notation}
Let $R$ be an $\BF_p$-algebra. Fix $n\in\BN$. We have natural epimorphisms 
$$W_n(R)\epi W_{n-1}(R)\epi\ldots \epi W_1(R)=R\epi W_0(R)=0.$$
Let $I_{n,R}:=\Ker (W_n (R)\epi W_1(R))$, $J_{n,R}:=\Ker (W_n (R)\epi W_{n-1}(R))$.

\subsubsection{Definition}  \label{sss:definition of sDisp_n}
An \emph{$n$-truncated semidisplay} over $R$ is a quadruple $(P,Q,F, F_1)$, where $P$ is a finitely generated projective $W_n (R)$-module, $Q\subset P$ a submodule,
$F:P\to P$ and $F_1:Q\to P/J_{n,R}\cdot P$ are semilinear with respect to $F:W(R)\to W(R)$, and the following conditions hold:

(i) $P/Q$ is a projective module over $W_1(R)=R$ (in particular, $Q\supset I_{n,R}\cdot P$);

(ii) for $a\in W(R)$, $x\in P$, one has $F_1(V(a)\cdot x)=a\cdot \bar F(x)$, where $\bar F(x)$ is the image of $F(x)$ in $P/J_{n,R}\cdot P$;

(iii) $F(x)=pF_1(x)$ for $x\in Q$ (the expression  $pF_1(x)$ makes sense because $pJ_{n,R}=0$).

\medskip

Let $\sDisp_n (R)$ be the additive category of $n$-truncated semidisplays over $R$. 

\subsubsection{Remarks}  \label{sss:2Zink's observations}
(a) By \S\ref{sss:definition of sDisp_n}(ii), one has
\begin{equation} \label{e:bar F via F_1}
\bar F(x)=F_1(V(1)\cdot x)=F_1(px). 
\end{equation}

(b) If $x\in Q$ then \eqref{e:bar F via F_1} implies that $F(x)-pF_1 (x)\in J_{n,R}\cdot P$, which is weaker than \S\ref{sss:definition of sDisp_n}(iii).

(c) We have 
\begin{equation}  \label{e:J_{n,R}Q is killed}
F_1(J_{n,R}\cdot Q)=0, \quad F(J_{n,R}\cdot Q)=0.
\end{equation}
Indeed, the first equality is clear because $P/J_{n,R}\cdot P$ is killed by $J_{n,R}$, and the second equality follows from the first one by \S\ref{sss:definition of sDisp_n}(iii).

(d) By \S\ref{sss:definition of sDisp_n}(iii), $F:P\to P$ induces a $p$-linear map $P/Q\to P/Q$ (so $P/Q$ is a commutative restricted Lie $R$-algebra).

\subsubsection{Normal decompositions}   \label{sss:Normal decompositions} 
By \S\ref{sss:definition of sDisp_n}(i), there exists a decomposition $P=T\oplus L$ such that $Q=I_{n,R}\cdot T\oplus L$. Following Th.~Zink and E.~Lau \cite{Zink, Lau13, LZ}, we call this a \emph{normal decomposition.}

Let us note that the notation $T,L$ for the terms of a normal decomposition is standard. Mnemonic rule: $T$ stands for ``tangent'' (in fact, the $R$-module $T/I_{n,R}\cdot T=P/Q$ is the Lie algebra of the $n$-smooth group scheme discussed in
\S\ref{ss:truncated Zink's functor} below, see Lemma~\ref{l:Lie of Zink}).

\subsubsection{Weak $n$-truncated semidisplays}   \label{sss:sDisp_weak}
Given $(P,Q,F, F_1)\in\sDisp_n (R)$, set 
$$M:=P/J_{n,R}\cdot Q, \quad \cQ:=Q/J_{n,R}\cdot Q.$$
 By \eqref{e:J_{n,R}Q is killed}, the maps $F:P\to P$ and $F_1:Q\to P/J_{n,R}\cdot P$ induce maps
$$F:M\to M, \quad F_1:\cQ\to M/J_{n,R}\cdot M.$$

The quadruple $(M,\cQ, F:M\to M, F_1:\cQ\to M/J_{n,R}\cdot M)$ has the following properties:

(a) $F$ and $F_1$ are semilinear with respect to $F:W_n(R)\to W_n(R)$ and satisfy the relations from \S\ref{sss:definition of sDisp_n}(ii)-(iii).

(b) The pair $(M,\cQ )$ is isomorphic to $(T,I_{n,R}\cdot T)\oplus
(\bar L, \bar L)$ for some finitely generated projective $W_n(R)$-module $T$ and some finitely generated projective $W_{n-1}(R)$-module $\bar L$.

Quadruples $(M,\cQ, F, F_1)$ satisfying (a)-(b) are called \emph{weak $n$-truncated semidisplays.} They form an additive category, denoted by $\sDisp_n^{\weak} (R)$. We have defined a functor 
$$\sDisp_n (R)\to\sDisp_n^{\weak} (R).$$
Using normal decompositions, one checks that it is essentially surjective.

\subsubsection{Example:  $n=1$}  \label{sss:1-truncated semidisplays}
$\sDisp_1 (R)$ is the category of triples $(P,Q,F)$, where $P$ is a finitely generated projective $R$-module, $Q\subset P$ is a direct summand of $P$, and $F:P/Q\to P$ is a $p$-linear map. On the other hand, $\sDisp_1^{\weak} (R)$ is the category of finitely generated projective $R$-modules equipped with a $p$-linear endomorphism. The functor $\sDisp_n (R)\to\sDisp_n^{\weak} (R)$ takes $(P,Q,F)$ to $P/Q$ equipped with the composite map $P/Q\overset{F}\longrightarrow P\epi P/Q$.

\subsection{$n$-truncated displays}   \label{ss:n-truncated displays}
As before, let $R$ be an $\BF_p$-algebra and $n\in\BN$. 
Let $\Disp_n(R)$ be the category of \emph{$n$-truncated displays} in the sense of \cite[Def.~3.4]{Lau13}. This category is studied in \cite[\S 3]{Lau13} and \cite[\S 1]{LZ}. (Note that the setting of \cite{LZ} is more general than we need: the ring $R$ is required there to be a $(\BZ/p^m\BZ)$-algebra for some $m$.)

\subsubsection{The functor $\Disp_n(R)\to\sDisp_n(R)$}
According to  \cite[Def.~3.4]{Lau13}, an object of $\Disp_n(R)$ is a collection
\begin{equation}  \label{e:P,Q,iota,epsilon,F,F_1}
(P,Q,\iota :Q\to P, \varepsilon :I_{n+1,R}\otimes_{W_n(R)}P\to Q,\, F:P\to P, \, F_1:Q\to P)
\end{equation}
with certain properties. We consider the functor 
\begin{equation}  \label{e:Disp_n to sDisp_n}
\Disp_n(R)\to\sDisp_n(R)
\end{equation}
that takes a collection \eqref{e:P,Q,iota,epsilon,F,F_1} to $(P,Q',F,F'_1)$, where $Q':=\iota (Q)$ and $F'_1:Q'\to P/J_{n,R}\cdot P$ is induced by $F_1:Q\to P$.

\subsubsection{Strong $n$-truncated semidisplays}
One of the properties of a collection \eqref{e:P,Q,iota,epsilon,F,F_1} required in \cite[Def.~3.4]{Lau13} is as follows: $P$ has to be generated by $F_1(Q)$ as a module. Skipping this property, one gets a generalization of the notion of $n$-truncated display, which we call \emph{strong $n$-truncated semidisplay.}  The category of strong $n$-truncated semidisplays over $R$ is denoted by $\sDisp_n^{\strong}(R)$.
The functor \eqref{e:Disp_n to sDisp_n} extends from 
$\Disp_n(R)$ to the bigger category $\sDisp_n^{\strong}(R)$.

In the last paragraph on p.141 of \cite{Lau13} it is explained that the pair $(F,F_1)$ from \eqref{e:P,Q,iota,epsilon,F,F_1} is described by a semilinear map $\Psi :L\oplus T\to P$, where
$(L,T)$ is a normal decomposition\footnote{In this context, the notion of normal decomposition is defined in \cite{Lau13} before Lemma~3.3.} of $(P,Q,\iota , \varepsilon )$. In the case of an $n$-truncated display, the linear map corresponding to $\Psi$ has to be invertible; in the case of a strong $n$-truncated semidisplay, this is not required. Roughly speaking, the difference between $n$-truncated displays and strong $n$-truncated semidisplays amounts to the difference between the group of invertible matrices and the semigroup of all matrices.

Using normal decompositions, one checks that the functor $\sDisp_n^{\strong}(R)\to\sDisp_n(R)$ is essentially surjective.

\subsubsection{Remark}  \label{sss:how to deal}
According to E.~Lau, a good way of dealing with $\Disp_n(R)$ is to replace it by the equivalent category $\DISP^{[0,1]}_n(R)$ (see \S\ref{sss:preDISP} and  \S\ref{sss:preDISP=sDisp-strong} below).
Similarly, one can work with $\sDisp_n^{\strong}(R)$ using the equivalence $\preDISP_n^{[0,1]}(R)\iso\sDisp_n^{\strong}(R)$ from \S\ref{sss:preDISP=sDisp-strong}; e.g., one can use it to define the structure of tensor category on $\sDisp_n^{\strong}(R)$ (see \S\ref{sss:tensor structure on the strong one} below).

\subsection{Tensor product in $\sDisp_n(R)$ and $\sDisp_n^{\weak}(R)$}    \label{ss:tensoring truncated semidisplays} 
\begin{lem}  \label{l:tensoring truncated semidisplays}
Let $(P,Q,F, F_1)$ and $(P',Q',F', F'_1)$ be objects of $\sDisp_n(R)$. Let 
\begin{equation}
P'':=P\otimes P', \quad Q'':=P\otimes Q'+Q\otimes P'=\Ker (P\otimes P'\epi (P/Q)\otimes (P'/Q')).
\end{equation}
Define $F'':P''\to P''$ by
\begin{equation}
F'':=F\otimes F'.
\end{equation}
Then there exists a (unique) additive map $F''_1:Q''\to P''/J_{n,R}\cdot P''$ such that
\begin{equation}
F''_1|_{P\otimes Q'}=F\otimes F'_1, \quad F''_1|_{Q\otimes P'}=F_1\otimes F'.
\end{equation}
Moreover, $(P'',Q'',F'', F''_1)\in \sDisp_n(R)$. 
\end{lem}

\begin{proof}
The composite maps
\[
Q\otimes Q'\to Q\otimes P'\overset{F_1\otimes F'}\longrightarrow (P\otimes P')/J_{n,R}(P\otimes P'), \;
Q\otimes Q'\to P\otimes Q'
\overset{F\otimes F'_1}\longrightarrow P\otimes P'/J_{n,R}(P\otimes P') 
\]
are both equal to $pF_1\otimes F'_1$ by \S\ref{sss:definition of sDisp_n}(iii). If $x\in P$, $x'\in P'$, $a\in W_n(R)$ then
\begin{equation}   \label{e:restricting to IP''}
(F\otimes F'_1)(V(a)(x\otimes x'))=a\overline{F''} (x\otimes x')=
(F_1\otimes F')(V(a)(x\otimes x')).
\end{equation}
These facts imply the existence of $F''_1$ because
\[
(P\otimes Q')\cap (Q\otimes P')=\im (Q\otimes Q'\to P\otimes P')+I_R\cdot (P\otimes P').
\]
By \eqref{e:restricting to IP''}, $F''$ and $F''_1$ satisfy the relation from \S\ref{sss:definition of sDisp_n}(ii). Finally, it is easy to check that $F''|_{Q''}=pF_1''$.
\end{proof}

\subsubsection{$\sDisp_n(R)$ as a tensor category}  \label{sss:sDisp_n(R) as a tensor category} 
In the situation of Lemma~\ref{l:tensoring truncated semidisplays}, $(P'',Q'',F'', F''_1)$ is called the \emph{tensor product} of $(P,Q,F, F_1)$ and $(P',Q',F', F'_1)$; it is denoted by 
$$(P,Q,F, F_1)\otimes (P',Q',F', F'_1).$$ This tensor product makes $\sDisp_n(R)$ into a tensor category (by which we mean a symmetric monoidal additive category). The object
\begin{equation}   \label{e:unit object}
\bone_{n,R}:=(W_n(R), I_{n,R}, \, F, \: V^{-1}\!:\!I_{n,R}\to W_{n-1}(R))
\end{equation}
is the unit in $\sDisp_n(R)$.

Let us note that in the case $n=\infty$ the object \eqref{e:unit object} appears in \cite[Example 16]{Zink} under the name of ``multiplicative display'' (because it corresponds to the multiplicative formal group).

\subsubsection{$\sDisp_n^{\weak}(R)$ as a tensor category}  \label{sss:sDisp_n weak(R) as a tensor category}
If $(M,Q,F, F_1)$ and $(M',Q',F', F'_1)$ are objects of the category $\sDisp_n^{\weak}(R)$ defined in \S\ref{sss:sDisp_weak}, we set
$$(M,Q,F, F_1)\otimes (M',Q',F', F'_1):=(M'',Q'',F'', F''_1),$$
where
\[
M'':=M\otimes M', \quad Q'':=\Ker (M\otimes M'\epi (M/Q')\otimes (M'/Q')), \quad F'':=F\otimes F',
\]
and $F''_1:Q''\to M''/J_{n,R}\cdot M''$ is the unique additive map such that
\[
F''_1(x\otimes y)=F(x)\otimes F'_1 (y) \mbox{ if } x\in M, y\in Q', 
\]
\[
F''_1(x\otimes y)=F_1(x)\otimes F' (y) \mbox{ if } x\in Q, y\in M'.
\]
The existence of $F''_1$ follows from Lemma~\ref{l:tensoring truncated semidisplays} and essential surjectivity of the functor $\sDisp_n(R)\to\sDisp_n^{\weak}(R)$.

Thus $\sDisp_n^{\weak}(R)$ becomes a tensor category equipped with a tensor functor $$\sDisp_n(R)\to\sDisp_n^{\weak}(R).$$ 
The quadruple $\bone_{n,R}$ given by \eqref{e:unit object}  is the unit object of $\sDisp_n^{\weak}(R)$.

\subsection{$\sDisp_n$ and $\sDisp_n^{\weak}$ as stacks of tensor categories}
In \cite{Lau13} it is proved that the assignment $R\mapsto\Disp_n (R)$ is an fpqc stack of categories.\footnote{Given a morphism of $\BF_p$-algebras $f:R\to\tilde R$, there is an obvious notion of $f$-morphism from an object of $\Disp_n (R)$ to an object of $\Disp_n (\tilde R)$.
Existence of a (unique) base change functor is proved in \cite[Lemma~3.6]{Lau13}. Descent for $n$-truncated displays is proved in \cite[\S 3.3]{Lau13}.} The same is true for $\sDisp_n^{\strong}$, $\sDisp_n$, $\sDisp_n^{\weak}$ and proved in the same way.

For a morphism of $\BF_p$-algebras $f:R\to\tilde R$, the corresponding base change functor
$
\sDisp_n (R)\to\sDisp_n (\tilde R)
$
takes $(P,Q,F,F_1)$ to $(\tilde P,\tilde Q,\tilde F,\tilde F_1)$, where 
$$\tilde P=W_n(\tilde R)\otimes_{W_n(R)}P, \quad \tilde P/\tilde Q=W_n(\tilde R)\otimes_{W_n(R)}(P/Q),$$
$\tilde F:\tilde P\to\tilde P$ is the base change of $F$, and $\tilde F_1:\tilde Q\to\tilde P/J_{n,R}\cdot\tilde P$ is the unqiue $F$-linear map such that 
the diagram
\[
\xymatrix{
Q\ar[r]^-{F_1}\ar[d]&  \; P/J_{n,R}\cdot P\ar[d]\\
\tilde Q\ar[r]^-{\tilde F_1}&  \tilde P/J_{n,\tilde R}\cdot \tilde P
}
\]
commutes and $\tilde F_1(V(a)\otimes x)=a\otimes\bar F(x)$ for all $a\in W(\tilde R)$ and $x\in P$. The existence of $\tilde F_1$ can be proved using a normal decomposition of $(P,Q,F,F_1)$.

The above description of the base change functor $\sDisp_n (R)\to\sDisp_n (\tilde R)$ and the quite similar description of the functor $\sDisp_n^{\weak} (R)\to\sDisp_n^{\weak} (\tilde R)$ shows that these are \emph{tensor} functors.
Thus $\sDisp_n$ and $\sDisp_n^{\weak}$ as fpqc stacks of tensor categories

\subsection{Algebraicity of the stacks}  \label{ss:Algebraicity of sDisp_n}
It is easy to show that the stacks of categories $\Disp_n$, $\sDisp_n$, $\sDisp_n^{\weak}$, $\sDisp_n^{\strong}$ are algebraic c-stacks in the sense of \cite[\S 2.3]{Prismatization}.
For the purposes of this paper, it is enough to know that the corresponding stacks of groupoids are algebraic (in the usual sense). 
We will formulate a more precise statement, see Proposition~\ref{p:smoothness of sDisp_n} below.

Given integers  $d$ and $d'$ such that $0\le d'\le d$, let $\sDisp_n^{d,d'}(R)$ be the full subgroupoid of the underlying groupoid of $\sDisp_n(R)$ whose objects are quadruples $(P,Q,F,F_1)\in\sDisp_n(R)$ such that
$\rank P=d$ and $\rank (P/Q)=d'$.

Define $\sDisp_n^{d,d',\weak}(R)$, $\sDisp_n^{d,d',\strong}(R)$, and $\Disp_n^{d,d'}(R)$ similarly but with the following changes:

(i) in the case of  $\Disp_n^{d,d'}(R)$ and $\sDisp_n^{d,d',\strong}(R)$ replace $P/Q$ by $\Coker (Q\to P)$;

(ii) in the case of $\sDisp_n^{d,d',\weak}(R)$ the condition for $P$ is that $\rank (P/I_{n,R}P)=d$.

Note that $\sDisp_1^{d,d'}\ne\emptyset$ only if $d'=d$; this follows from condition (b) from \S\ref{sss:sDisp_weak}.

\begin{prop}  \label{p:smoothness of sDisp_n}
Each of the stacks $\sDisp_n^{d,d'}$, $\sDisp_n^{d,d',\weak}$, $\sDisp_n^{d,d',\strong}$, and $\Disp_n^{d,d'}$ is a smooth algebraic stack of finite type over $\BF_p$.
Moreover, $\sDisp_n^{d,d'}$ has pure dimension $-(d-d')^2$, and the other three stacks have pure dimension $0$.
\end{prop}

In the case of $\Disp_n^{d,d'}$, this is \cite[Prop.~3.15]{Lau13}.

\begin{proof}
Follows from the explicit presentation of the four stacks given in Appendix~\ref{s:Explicit presentations}.
\end{proof}

\section{Zink's functor}  \label{s:Zink's functor}
\subsection{Zink's functor $\sDisp_n(R)\to\Sm_n(R)$}  \label{ss:truncated Zink's functor}
As before, $\Sm_n(R)$ stands for the category of commutative $n$-smooth group schemes over $R$, see \S\ref{ss:n-smooth group schemes bis}.
In this subsection we recall the functor $\sDisp_n(R)\to\Sm_n(R)$, which was essentially\footnote{The caveat is due to the fact that the authors of \cite{LZ} worked with displays rather than semidisplays.} constructed in \cite[\S 3.4]{LZ}.
In \S\ref{sss:Zink's functor on the weak category} we will decompose this functor as $\sDisp_n(R)\to\sDisp_n^{\weak}(R)\to\Sm_n(R)$.

\subsubsection{Format of the construction}  \label{sss:format of LZ}
To an object $\sP=(P,Q,F.F_1)\in\sDisp_n(R)$ one functorially associates a diagram of commutative group ind-schemes
\begin{equation}  \label{e:Zink's diagram}
C_\sP^{-1}\overset{\Phi}\longrightarrow C_\sP^0\, ,
\end{equation}
in which $C_\sP^{-1}$ is a closed subgroup of $C_\sP^0$. 
It is proved in \cite[Prop.~3.11]{LZ} that $\Ker (1-\Phi)=0$ and the functor 
$\tilde R\mapsto \Coker (C_\sP^{-1}(\tilde R )\overset{1-\Phi}\longrightarrow C_\sP^0 (\tilde R ))$
(where $\tilde R$ is an $R$-algebra) is representable by an $n$-smooth group scheme. We denote this group scheme by $\fZ_{\sP}$. Thus
\begin{equation}  \label{e:Zink functor}
\fZ_{\sP}:=\Coker (C_\sP^{-1}\overset{1-\Phi}\longrightarrow C_\sP^0).
\end{equation}
The functor $\sP\mapsto\fZ_{\sP}$ will be called the \emph{Zink functor}; in the case $n=\infty$ it was defined by Th.~Zink in \cite[\S 3]{Zink} (under the name of $BT_{\sP}$). The complex of group ind-schemes
\begin{equation} \label{e:Zink complex}
0\to C_\sP^{-1}\overset{1-\Phi}\longrightarrow C_\sP^0\to 0
\end{equation}
will be called the \emph{Zink complex} of $\sP$.

\subsubsection{Defining \eqref{e:Zink's diagram}}  \label{sss:Zink's diagram}
Let $\hat W^{(F^n)}:=\Ker (F^n:\hat W\to\hat W)$, where $\hat W$ is the formal Witt group (see Appendix A).
For any $\BF_p$-algebra $\tilde R$, the subgroup  $\hat W^{(F^n)}(\tilde R)\subset W(\tilde R)$ consists of Witt vectors $x\in W(\tilde R)$ such that $F^n(x)=0$ and almost all components of $x$ are zero. Note that $\hat W^{(F^n)}(\tilde R)$ is a $W(\tilde R)$-submodule of $W(\tilde R)$ and moreover, a $W_n(\tilde R)$-submodule.

The group ind-schemes $C_\sP^0,C_\sP^{-1}$ from \eqref{e:Zink's diagram} are as follows: for any $R$-algebra $\tilde R$,
\[
C_\sP^0 (\tilde R):=\hat W^{(F^n)}(\tilde R)\otimes_{W_n(R)}P,
\]
\[
C_\sP^{-1} (\tilde R):=\Ker (C_\sP^0 (\tilde R)\overset{\pi}\epi \BG_a^{(F^n)}(\tilde R)\otimes_R(P/Q)),
\]
where $\pi$ comes from the map $\hat W^{(F^n)}\to \BG_a^{(F^n)}$ that takes a Witt vector to its $0$-th component.

The additive homomorphism $\Phi :C_\sP^{-1}(\tilde R)\to C_\sP^0(\tilde R)$ is uniquely determined by the following properties:
\begin{equation}   \label{e:condition 1 for Phi}
\Phi (V(a)\otimes x)=a\otimes F(x) \quad\mbox{ for }\, a\in\hat W^{(F^n)}(\tilde R), \, x\in P,
\end{equation}
\begin{equation}  \label{e:condition 2 for Phi}
\Phi (a\otimes y)=F(a)\otimes F_1(y) \quad\mbox{ for }\, a\in\hat W^{(F^n)}(\tilde R), \, y\in Q.
\end{equation}
The r.h.s. of \eqref{e:condition 2 for Phi} makes sense (despite the fact that $F_1(y)$ is defined only modulo $J_{n,\tilde R}\cdot P$)
because for any $a\in\hat W^{(F^n)}(\tilde R)$ and $b\in W(\tilde R)$ one has $V^{n-1}(b)\cdot F(a)=V^{n-1}(b\cdot F^n(a))=0$.

\subsubsection{The ind-schemes $C_\sP^i$ are ``safe''}   \label{sss:safe}
Say that an ind-scheme over $R$ is \emph{safe} if it has the form $\Spf A^*$, where $A$ is a cocommutative $R$-coalgebra which is a projective $R$-module. For commutative group ind-schemes whose underlying ind-scheme is safe there is a reasonable notion of Cartier dual (see \S\ref{ss:generalities on Cartier duality} of Appendix~\ref{s:hat W}).

The ind-scheme underlying $C_\sP^i$ is safe for each $i$. In fact, it is easy to see that the group ind-scheme $C_\sP^i$ is isomorphic to $\hat W_R^{(F^n)}\otimes_{W_n(R)}P_i$ for some finitely generated projective $W_n(R)$-module~$P_i$ (to see this for $i=-1$, choose a normal decomposition of $\sP$ in the sense of \S\ref{sss:Normal decompositions}).

\subsubsection{On the proof of Proposition~3.11 of \cite{LZ}}
In \S\ref{sss:format of LZ} we formulated a result from \cite{LZ}. In \cite{LZ} it is deduced from Theorem~81 of \cite{Zink}, whose proof (given on p.80-81 of \cite{Zink}) does not use the surjectivity assumption\footnote{Part (ii) of \cite[\S 1, Def.~1]{Zink}.} from Zink's definition of display. So Theorem~81 of \cite{Zink} and Proposition~3.11 of \cite{LZ} are valid for 
\emph{semi}displays\footnote{More details about this can be found in the proof of  \cite[Prop.~11.13]{Lau25}.}.

The idea behind the proof of the result mentioned in \S\ref{sss:format of LZ} is roughly as follows: the group ind-schemes $C_\sP^0,C_\sP^{-1}$ are $n$-smooth in a certain sense\footnote{We defined $n$-smoothness only for group \emph{schemes.}}, and one checks that the map
\begin{equation}   \label{e:f=Lie(1-Phi)}
f:\Lie (C_\sP^{-1})\to\Lie (C_\sP^0)
\end{equation}
induced by $1-\Phi$ is a monomorphism whose cokernel is a finitely generated projective $R$-module. For completeness, let us prove these properties of $f$; we will also give an explicit description of $\Coker f=\Lie (\fZ_\sP)$.

\begin{lem}  \label{l:Lie of Zink}
(i) The map \eqref{e:f=Lie(1-Phi)} is injective.

(ii) One has a canonical isomorphism of restricted Lie $R$-algebras
\begin{equation}   \label{e:Lie of Zink}
\Lie (\fZ_{\sP})\iso P/Q,
\end{equation}
where the structure of restricted Lie algebra on $P/Q$ is as in \S\ref{sss:2Zink's observations}(d). In particular, $\rank (\fZ_{\sP})=\rank (P/Q)$.
\end{lem}

\begin{proof}
A commutative restricted Lie $R$-algebra $\fg$ is the same as a left $R[F]$-module, where $Fa=a^pF$ for all $a\in R$.
If $G$ is a commutative group ind-scheme over~$R$ which is safe in the sense of \S\ref{sss:safe} and $\fg =\Lie (G)$, then $F:\fg\to\fg$ comes from $V:\Fr^*G\to G$.

Let us describe \eqref{e:f=Lie(1-Phi)} as a homomorphism of $R[F]$-modules. One has a canonical isomorphism of $R[F]$-modules 
$$R[F]\iso\Lie (\hat W_R)=\Lie (\hat W^{(F^n)}_R)$$ 
taking $1\in R[F]$ to $\alpha\in \Lie (\hat W_R)$, where $\alpha$ is the derivative of the Teichm\"uller map $\hat\BA^1_R\to\hat W_R$ at $0\in\hat\BA^1_R$.
Let $\bar P=P/I_{n,R}P$, $\bar Q=Q/I_{n,R}P$. Then $\Lie (C_\sP^0)=R[F]\otimes_R\bar P$, and $\Lie (C_\sP^1)$ is the $R[F]$-submodule of $R[F]\otimes_R\bar P$ generated by $1\otimes\bar Q$ and the elements $F\otimes x$, where $x\in\bar P$.
Our map $f:\Lie (C_\sP^{-1})\to\Lie (C_\sP^0)$ equals $\Lie (1-\Phi )$, where $\Phi$ is given by \eqref{e:condition 1 for Phi}-\eqref{e:condition 2 for Phi}. So
\[
f(1\otimes x)=1\otimes x \;\mbox{ for }x\in\bar Q, \quad   f(F\otimes x)=F\otimes x-1\otimes F(x) \;\mbox{ for }x\in\bar P.
\]
This description of $f$ implies the lemma.
\end{proof}

\subsection{The Cartier dual of $\fZ_\sP$}   \label{ss:Cartier dual of Z_P}
\subsubsection{Goal of this subsection}   \label{sss:Cartier dual of Z_P}
Let $\sP=(P,Q,F,F_1)\in\sDisp_n(R)$. Let $\fZ_\sP^*:=\HHom (\fZ_\sP, \BG_m)$, where $\fZ_\sP\in\Sm_n(R)$ is as in \S\ref{sss:format of LZ}-\ref{sss:Zink's diagram}. 
According to the general principle from \S\ref{ss:n-cosmooth easier than n-smooth}, the group scheme $\fZ_\sP^*\in\Sm_n^*(R)$ is easier to understand than $\fZ_\sP$ itself.
Namely, we are going to prove Proposition~\ref{p:equality of subgroups}, which establishes a canonical isomorphism
\begin{equation}   \label{e:Cartier dual of Z_P}
\fZ_\sP^*\iso\HHom (\sP,\bone_{n,R})
\end{equation}
where $\bone_{n,R}\in\sDisp_n(R)$ is given by \eqref{e:unit object}. Here $\HHom (\sP,\bone_{n,R})$ is the group $R$-scheme whose group of points over any $R$-algebra $R'$ is 
$\Hom (\sP',\bone_{n,R'})$, where $\sP'$ is the base change of $\sP$ to $R'$. Explicitly, $\Hom ( \sP',\bone_{n,R'})$ is the group of $W_n(R)$-linear maps $\eta :P\to W_n (R')$ such that
\begin{equation}   \label{e:Cond_x}
\eta (F(x))=F(\eta (x)) \quad \mbox{ for all }x\in P,
\end{equation}
\begin{equation}   \label{e:Cond_y}
\eta (y)=V(\eta (F_1(y)))\quad \mbox{ for all }y\in Q.
\end{equation}

Note that 
\begin{equation}  \label{e:first subgroup}
\HHom (\sP,\bone_{n,R})\subset\Hom_{W_n(R)}(P, W_{n,R}).
\end{equation}

\subsubsection{$\fZ_\sP^*$ as a subgroup of $\Hom_{W_n(R)}(P, W_{n,R})$}
By \S\ref{sss:format of LZ}-\ref{sss:Zink's diagram}, $\fZ_\sP$ is a quotient of the group ind-scheme 
$C_\sP^0:=\hat W_R^{(F^n)}\otimes_{W_n(R)}P$, so $\fZ_\sP^*$ is a subgroup of $(C_\sP^0)^*$.
 
We will be using the canonical nondegenerate pairing
\begin{equation}  \label{e:our pairing}
\hat W^{(F^n)}_R\times W_{n,R}\to (\BG_m)_R \, ,
\end{equation}
which comes from the usual Cartier duality between $W_R$ and $\hat W_R$, see Appendix~\ref{s:hat W}. This pairing induces an isomorphism
\[
(C_\sP^0)^*\iso\Hom_{W_n(R)}(P, W_{n,R}).
\]
Thus
\begin{equation}  \label{e:second subgroup}
\fZ_\sP^*\subset\Hom_{W_n(R)}(P, W_{n,R}).
\end{equation}

\begin{prop}   \label{p:equality of subgroups}
The subgroups of $\Hom_{W_n(R)}(P, W_{n,R})$ given by \eqref{e:first subgroup} and \eqref{e:second subgroup} are equal to each other.
\end{prop}

\begin{proof} 
Let $R'$ be an $R$-algebra and $\eta$ an $R'$-point of $\Hom_{W_n(R)}(P, W_{n,R})$, i.e., $\eta :P\to W_n (R')$ is a $W_n(R)$-linear map.
The problem is to show that $\eta\in\fZ_\sP^* (R')$ if and only if \eqref{e:Cond_x} and \eqref{e:Cond_y} hold.

Looking at formulas \eqref{e:Zink's diagram}-\eqref{e:condition 2 for Phi}, we see that $\eta\in\fZ_\sP^* (R')$ if and only if for every $R'$-algebra $R''$ and every $a\in\hat W^{(F^n)}(R'')$ the following conditions hold:
\begin{equation}   \label{e:2Cond_x}
\langle V(a), \eta (x)\rangle=\langle a, \eta (F(x))\rangle
\quad \mbox{ for all }x\in P,
\end{equation}
\begin{equation}   \label{e:2Cond_y}
\langle a, \eta (y)\rangle=\langle F(a), \eta (F_1(y))\rangle \quad \mbox{ for all }y\in Q.
\end{equation}
Here $\langle\, , \rangle$ stands for the pairing \eqref{e:our pairing}. 

Finally, conditions \eqref{e:2Cond_x}-\eqref{e:2Cond_y} are equivalent to \eqref{e:Cond_x}-\eqref{e:Cond_y}. This follows from the equalities
\[
\langle V(a), \eta (x)\rangle =\langle a, F(\eta (x))\rangle ,\quad \langle F(a), \eta (F_1(y))\rangle = \langle a, V(\eta (F_1(y))\rangle 
\]
(see formula~\eqref{e:F dual to V} from Appendix~\ref{s:hat W}) and the nondegeneracy of the pairing \eqref{e:our pairing}. 
\end{proof}

\begin{lem}   \label{l:nasty}
We have a commutative diagram
\[
\xymatrix{
&P\ar[r]\ar[ld]&  \Hom (\fZ_{\sP}^*,W_{n,R})\ar[d]\\ 
P/Q\ar[r]^-\sim &\Lie (\fZ_{\sP})\ar[r]^-\sim&  \Hom (\fZ_{\sP}^*,\BG_{a,R})\
}
\]
Its upper row comes from the embedding $\fZ_\sP^*\mono\Hom_{W_n(R)}(P, W_{n,R})$ (see \eqref{e:second subgroup}), the first lower horizontal arrow is 
\eqref{e:Lie of Zink}, and the second one is obtained by applying $\Lie$ to the isomorphism $\fZ_\sP\iso\Hom (\fZ_\sP^*,\BG_{m,R})$.
\end{lem}

\begin{proof}
By \S\ref{ss:truncated Zink's functor}, $\fZ_{\sP}$ is a quotient of $C_\sP^0:=P\otimes_{W_n(R)}\hat W^{(F^n)}_R$.
The map $P\to\Lie (\fZ_{\sP})$ from the proof of Lemma~\ref{l:Lie of Zink} comes from a map $P\to\Lie (P\otimes_{W_n(R)}\hat W^{(F^n)}_R)$. So it suffices to show that the composition
\[
P\to\Lie  (P\otimes_{W_n(R)}\hat W^{(F^n)}_R)\iso\Hom ((P\otimes_{W_n(R)}\hat W^{(F^n)}_R)^*,\BG_{a,R})=\Hom (P^*\otimes_{W_n(R)}W_{n,R},\BG_{a,R})
\]
is the tautological map coming from the pairing $P\times P^*\to W_n(R)$ and the usual projection $w_0:W_{n,R}\to\BG_{a,R}\,$. Indeed, the map $P\to\Lie (P\otimes_{W_n(R)}\hat W^{(F^n)}_R)$ was constructed using a canonical element
$\alpha\in\Lie (\hat W^{(F^n)}_R)$ (see the proof of Lemma~\ref{l:Lie of Zink}), and the isomorphism
\[
\Lie (\hat W^{(F^n)}_R)\iso \Hom ((\hat W^{(F^n)}_R)^*,\BG_{a,R})= \Hom (W_{n,R},\BG_{a,R})
\]
takes $\alpha$ to $w_0\in \Hom (W_{n,R},\BG_{a,R})$.
\end{proof}

\subsection{The functor $\sDisp_n^{\weak}(R)\to\Sm_n(R)$}  \label{ss:Zink's functor on the weak category} 
Recall that $\bone_{n,R}$ is the unit object in both $\sDisp_n(R)$ and $\sDisp_n^{\weak}(R)$ (see \S\ref{sss:sDisp_n(R) as a tensor category}-\ref{sss:sDisp_n weak(R) as a tensor category}).
For $\sP\in\sDisp_n^{\weak}(R)$ we define the group $R$-scheme $\HHom (\sP,\bone_{n,R})$ just as in \S\ref{sss:Cartier dual of Z_P}.

\begin{lem}   \label{l:comparing the two HHoms}
Let $\sP'\in\sDisp_n(R)$. Let $\sP\in\sDisp_n^{\weak}(R)$ be the image of $\sP'$ (see \S\ref{sss:sDisp_weak}). Then the natural map
$\HHom (\sP,\bone_{n,R})\to\HHom (\sP',\bone_{n,R})$ is an isomorphism.
\end{lem}

\begin{proof} 
Write $\sP'=(P,Q,F,F_1)$. The problem is to show that for every $f\in\Hom (\sP',\bone_{n,R})$ one has $f(J_{n,R}\cdot Q)=0$. 
By the definition of $\bone_{n,R}$ (see formula \eqref{e:unit object}), we have $f(Q)\subset I_{n,R}$, so $f(J_{n,R}\cdot Q)\subset J_{n,R}\cdot I_{n,R}=0$. 
\end{proof}

\begin{lem}   \label{l:n-cosmoothness}
For every $\sP\in\sDisp_n^{\weak}(R)$, the group scheme $\HHom (\sP,\bone_{n,R})$ is $n$-cosmooth.
\end{lem}

\begin{proof} 
As noted in \S\ref{sss:sDisp_weak}, the functor $\sDisp_n^{\weak}(R)\to\sDisp_n(R)$ is essentially surjective.
So by Lemma~\ref{l:comparing the two HHoms}, it suffices to prove $n$-cosmoothness of 
$\HHom (\sP',\bone_{n,R})$ for $\sP'\in\sDisp_n(R)$. By \eqref{e:Cartier dual of Z_P}, $\HHom (\sP',\bone_{n,R})=\fZ_{\sP'}^*$.
Finally, $\fZ_{\sP'}$ is $n$-smooth by the result of \cite{LZ} mentioned in \S\ref{sss:format of LZ}.
\end{proof}

\subsubsection{The functor $\sDisp_n^{\weak}(R)\to\Sm_n(R)$}  \label{sss:Zink's functor on the weak category} 
Let $\sP\in\sDisp_n^{\weak}(R)$. By Lemma~\ref{l:n-cosmoothness}, $\HHom (\sP,\bone_{n,R})=\fZ_\sP^*$ for some $\fZ_\sP\in\Sm_n(R)$. 
The assignment $\sP\mapsto\fZ_\sP$ is a functor $\sDisp_n^{\weak}(R)\to\Sm_n(R)$. By Lemma~\ref{l:comparing the two HHoms}, precomposing this functor with the functor $\sDisp_n(R)\to\sDisp_n^{\weak}(R)$, we get the functor
$\fZ$ from~\S\ref{sss:format of LZ}.

\subsubsection{Example: $n=1$}   \label{sss:elementary case of Zink's functor}
An object of $\sDisp_n^{\weak}(R)$ is a quadruple $(M,\cQ, F, F_1)$ see \S\ref{sss:sDisp_weak}. Now suppose that $n=1$. Then $\cQ=0$, $F_1=0$, and $M$ is a projective $R$-module.
So an object $\sP\in\sDisp_1^{\weak}(R)$ is just a pair $(M,F)$, where $M$ is a projective $R$-module and $F:M\to M$ is a $p$-linear map. To such a pair we associated in \S\ref{ss:1-cosmooth group schemes} a 1-cosmooth group scheme $A_{M,F}$. It is easy to see that
\begin{equation}  \label{e:Zink's functor for n=1}
\fZ_\sP^*=A_{M,F}.
\end{equation}

\subsection{$\fZ_{\sP}^*$ via Dieudonn\'e modules}   \label{ss:Dual Witt via Dieudonne}
Let $R$ be an $\BF_p$-algebra.

\subsubsection{The ring $\fD_{n,R}$}
Let $\fD_{n,R}$ be the $n$-truncated Dieudonn\'e-Cartier ring of $R$.
It is generated by the ring $W_n(R)$ and elements $F,V$; the defining relations in $\fD_{n,R}$ are
\[
V^n=0, \quad FV=p,
\]
\[
F\cdot a=F(a)\cdot F, \quad a\cdot V=V\cdot  F(a), \quad V\cdot a\cdot F=V(a) \quad \mbox{ for all } a\in W_n(R).
\]
Note that $VF=V\cdot 1\cdot F=V(1)=p$. For every $a\in W_n(R)$ and every $i\in\{1,\ldots, n-1\}$ we have
$V^i\cdot (V^{n-i}(a))=V^n\cdot a\cdot F^{n-i}=0$; in particular,
\begin{equation} \label{e:VJ=0}
V\cdot J_{n,R}=0.
\end{equation}

\subsubsection{The goal}
Let $\fD_{n,R}\mbox{-mod}$ be the category of left $\fD_{n,R}$-modules. 
We will define a functor\footnote{Without $n$-truncation, the functor $D$ is well known: see \cite[Prop.~90]{Zink} and
formula (4) on p.141 of \cite{Zink-short} (this formula goes back to the article \cite{N}, which inspired Zink's notion of display).}
 $D:\sDisp_n^{\weak}(R)\to \fD_{n,R}\mbox{-mod}$ such that for every $\sP\in\sDisp_n^{\weak}(R)$ one has
\begin{equation}   \label{e:Dual Witt via Dieudonne}
\fZ_{\sP}^*=\Hom_{\fD_{n,R}}(D(\sP), W_{n,R}).
\end{equation}
The r.h.s. of \eqref{e:Dual Witt via Dieudonne} makes sense because $\fD_{n,R}$ acts on the group scheme $W_{n,R}$.

\subsubsection{Definition of $D (\sP )$}  \label{sss:Definition of D (sP)}
By \S\ref{sss:sDisp_weak}, an object $\sP\in\sDisp_n^{\weak}(R)$ is a quadruple 
$$(M,\cQ, F:M\to M, F_1:\cQ\to M/J_{n,R}\cdot M).$$
$D(\sP)\in \fD_{n,R}\mbox{-mod}$ is defined as follows: if $N\in \fD_{n,R}\mbox{-mod}$ then $\Hom_{\fD_{n,R}}(D(\sP),N)$ is the group of
$W_n(R)[F]$-homomorphisms $f:M\to N$ such that
\begin{equation}   \label{e:VF_1}
V(f(F_1(y))=f(y) \quad \mbox{ for all } y\in\cQ.
\end{equation}
Although $F_1(y)$ lives in $M/J_{n,R}\cdot M$ rather than in $M$, the l.h.s. of \eqref{e:VF_1} makes sense by virtue of \eqref{e:VJ=0}.

Formula~\eqref{e:Dual Witt via Dieudonne} immediately follows from the definition of $\fZ_{\sP}^*$ (see \S\ref{sss:Zink's functor on the weak category}) and the definition of $\bone_{n,R}$ (see formula~\eqref{e:unit object}).

\begin{prop}  \label{p:n-cosmoothness of Dieudonne}
The $\fD_{n,R}$-module $D(\sP)$ is $n$-cosmooth in the sense of \cite[Def.~1.0.2]{KM}.
\end{prop}

The proposition will be proved in \S\ref{sss:Proof of $n$-cosmoothness of Dieudonne}.

\begin{cor}
$D(\sP)$ is the $n$-cosmooth $\fD_{n,R}$-module corresponding to $\fZ_\sP$ via the ``$n$-truncated Cartier theory'' developed in \cite{KM}.
\end{cor}

\begin{proof}
Combine Proposition~\ref{p:n-cosmoothness of Dieudonne} with formula~\eqref{e:Dual Witt via Dieudonne}.
\end{proof}

\subsubsection{An economic presentation of $D(\sP)$}
Let $\sP,M,\cQ, F, F_1$ be as in \S\ref{sss:Definition of D (sP)}. Choose a normal decomposition
\[
M=T\oplus\bar L, \quad \cQ =I_{n,R}\cdot T\oplus\bar L,
\]
where $T$ is a finitely generated projective $W_n(R)$-module and $\bar L$ is a finitely generated projective $W_{n-1}(R)$-module. Then $F:T\to M$ is a pair $(\varphi_{TT}:T\to T, \varphi_{\bar L T}:T\to\bar L)$, and
$F_1:\bar L\to M$ is a pair $(\varphi_{T\bar L}:\bar L\to T/J_{n,R}\cdot T, \varphi_{\bar L\bar L}:\bar L\to\bar L)$. The following proposition represents $D(\sP)$ as a quotient of
$\fD_{n,R}\otimes_{W_n(R)}T$ by an explicit submodule.

\begin{prop}  \label{p:D(sP) economically}
If $N\in \fD_{n,R}\mbox{-mod}$ then $\Hom_{\fD_{n,R}}(D(\sP),N)$ is the group of $W_n(R)$-linear maps $g:T\to N$ such that
\begin{equation}   \label{e:Fg}
F(g(x))=g(\varphi_{TT}(x))+\sum_{i=1}^{n-1}V^i g(\varphi_{T\bar L}\varphi_{\bar L\bar L}^{i-1}\varphi_{\bar L T}(x)) \quad \mbox{ for all } x\in T.
\end{equation}
\end{prop}

Note that by \eqref{e:VJ=0}, the r.h.s. of \eqref{e:Fg} is well-defined even though $\varphi_{T\bar L}\varphi_{\bar L\bar L}^{i-1}\varphi_{\bar L T}$ is a map $T\to T/J_{n,R}\cdot T$.

\begin{proof}
Write the map $f:M\to N$ from \S\ref{sss:Definition of D (sP)} as a pair $(g,h)$, where $g\in\Hom_{W_n(R)}(T,N)$, $h\in\Hom_{W_n(R)}(\bar L,N)$. The conditions for $g,h$ are as follows\footnote{Here we use that a $W_n(R)$-linear map $f:M\to N$ satisfies $fF=Ff$ and $VfF_1=f$ if and only if the first relation holds on $T$ and the second one on $\bar L$.}:
\begin{equation}   \label{e:F circ g}
F\circ g=g\circ\varphi_{TT}+h\circ\varphi_{\bar L T},
\end{equation}
\begin{equation}  \label{e:equation for h}
h-V\circ h\circ\varphi_{\bar L\bar L}=V\circ g\circ \varphi_{T\bar L}.
\end{equation}
Since $V^n=0$, the map $h\mapsto V\circ h\circ\varphi_{\bar L\bar L}$ is nilpotent. So using \eqref{e:equation for h}, one can express $h$ in terms of $g$.
Then \eqref{e:F circ g} becomes condition \eqref{e:Fg}.
\end{proof}

\subsubsection{Proof of Proposition~\ref{p:n-cosmoothness of Dieudonne}} \label{sss:Proof of $n$-cosmoothness of Dieudonne}
Proposition~\ref{p:n-cosmoothness of Dieudonne} follows from Proposition~\ref{p:D(sP) economically} combined with \cite[Prop.~4.3.1]{KM}. (On the other hand, C.~Kothari noticed that one can deduce Proposition~\ref{p:n-cosmoothness of Dieudonne} from \cite[Prop.~90]{Zink}.)
\qed

\section{A formula for $\fZ_{\sP_1\otimes\ldots\otimes\sP_l}$, where $\sP_i\in\sDisp_n^{\weak}(R)$}  \label{s:n-cosmooth ones as tensor category}
Let $R$ be an $\BF_p$-algebra.

\subsection{Formulation of the result}
Recall that $\sDisp_n^{\weak}(R)$ is a tensor category, see \S\ref{ss:tensoring truncated semidisplays}.
By \S\ref{ss:Zink's functor on the weak category}, we have a functor
\begin{equation}  \label{e:3dual Zink}
\sDisp_n^{\weak}(R)^{\op}\to\Sm_n^*(R), \quad \sP\mapsto \fZ_{\sP}^*=\HHom (\sP,\bone_{n,R}).
\end{equation}
For each $\sP_1,\ldots, \sP_l\in\sDisp_n^{\weak}(R)$, we have the tensor product map
\begin{equation}  \label{e:tensor product map}
\HHom (\sP_1,\bone_{n,R})\times\ldots \times\HHom (\sP_l,\bone_{n,R})\to \HHom (\sP_1\otimes\ldots\otimes\sP_l,\bone_{n,R}),
\end{equation}
which is a poly-additive map
\begin{equation} \label{e:dual Zink as a pseuotensor functor}
\fZ_{\sP_1}^*\times\ldots\fZ_{\sP_l}^*\to \fZ_{\sP_1\otimes\ldots\otimes\sP_l}^*.
\end{equation} 
The map \eqref{e:dual Zink as a pseuotensor functor} induces a group homomorphism
\begin{equation} \label{e:the map to be proved to be an iso}
\fZ_{\sP_1\otimes\ldots\otimes\sP_l}\to \PPoly-add (\fZ_{\sP_1}^*\times \ldots \fZ_{\sP_l}^*,\BG_m)^{(F^n)};
\end{equation}
as before, the superscript $(F^n)$ means passing to the kernel of $F^n$.

\begin{prop}     \label{p:Zink functor of a tensor product}
The map \eqref{e:the map to be proved to be an iso} is an isomorphism.
\end{prop}

The proposition will be proved in \S\ref{ss:proving Zink functor of a tensor product}.

\subsubsection{Remarks}   \label{sss:tensor remarks}
(i) In the $n=1$ case  Proposition~\ref{p:Zink functor of a tensor product} is equivalent to Proposition~\ref{p:Cartier dual of tensor product}; this follows from \S\ref{sss:elementary case of Zink's functor}. 

(ii) Proposition~\ref{p:Zink functor of a tensor product} is a part of a bigger and ``cleaner'' picture, which is more or less described in \S 8 of an older version of this paper\footnote{See version 5 of the e-print arXiv:2307.06194.} (but without detailed proofs). The main point is that $\Sm_n^*(R)$ has a natural structure of a tensor category and the functor
\eqref{e:3dual Zink} has a natural structure of a tensor functor. In the case $n=1$ this was proved in \S\ref{ss:1-cosmooth group schemes}.

(iii) Let $\fD_{n,R}$ be the $n$-truncated Dieudonn\'e-Cartier ring of $R$, see \S\ref{ss:Dual Witt via Dieudonne}.
The ``$n$-truncated Cartier theory'' developed in \cite{KM} provides a fully faithful functor from $\Sm_n(R)$ to the category of $\fD_{n,R}$-modules; the functor is $G\mapsto \HHom (G^*, W_n)$. It is natural to expect that this is a tensor functor if $\Sm_n(R)$ is equipped with the tensor product mentioned in the previous remark and the category of $\fD_{n,R}$-modules is equipped with the Antieau-Nikolaus tensor product \cite[\S 4.2-4.3]{AN}. By \S\ref{ss:1-cosmooth group schemes}, this is true for $n=1$ (in this case a $\fD_{n,R}$-module is just an $R$-module equipped with a $p$-linear endomorphism, and the Antieau-Nikolaus tensor product is just the tensor product over $R$).

\subsection{Some lemmas}

\begin{lem}  \label{l:again tensor structure}
Let $G,\ldots ,G_l\in\Sm_n^*(R)$. Then the natural map
\[
\Hom (G_1,\BG_a)\otimes_R \ldots \otimes_R \Hom (G_l,\BG_a)\to\Poly-add (G_1\times \ldots G_l,\BG_a)
\]
is an isomorphism.
\end{lem}

\begin{proof}
To simplify the notation, assume that $l=2$. Let $G'_i:=\Coker (V:\Fr^*G_i\to G_i)$, then $G'_i\in\Sm_1^*(R)$.
Since $\BG_a$ is killed by $V$, we have $\Hom (G_i,\BG_a)=\Hom (G'_i,\BG_a)$. Since
\[
\Poly-add (G_1\times G_2,\BG_a)=\Hom (G_1, \HHom (G_2,\BG_a))=\Hom (G_2, \HHom (G_1,\BG_a))
\]
and $\BG_a$ is killed by $V$, we see that $\Poly-add (G_1\times G_2,\BG_a)=\Poly-add (G'_1\times G'_2,\BG_a)$.
So the lemma reduces to the particular case $n=1$, which was treated in Proposition~\ref{p:tensor structure}(i). 
\end{proof}

\begin{lem}  \label{l:Lie algebra statement}
The map 
\begin{equation}  \label{e:the Lie algebra map}
\Hom (\fZ_{\sP_1\otimes\ldots\otimes\sP_l}^*,\BG_a)\to\Poly-add (\fZ_{\sP_1}^*\times\ldots\fZ_{\sP_l}^*,\BG_a)
\end{equation}
induced by \eqref{e:dual Zink as a pseuotensor functor} is an isomorphism.
\end{lem}

\begin{proof}
To simplify the notation, assume that $l=2$. 
By essential surjectivity of the functor $\sDisp_n(R)\to\sDisp_n^{\weak}(R)$,
we can assume that $\sP_1, \sP_2$ are objects of $\sDisp_n(R)$ (rather than $\sDisp_n^{\weak}(R)$).
Then so is the tensor product $\sP:=\sP_1\otimes\sP_2$.

Recall that $\sP_i$ is a quadruple $(P_i,Q_i, F,F_1)$, see \S\ref{sss:definition of sDisp_n}. Then $\sP$ is a quadruple $(P,Q,\ldots )$ with
$P=P_1\otimes P_2$, $Q=(P_1\otimes Q_2)+(Q_1\otimes P_2)$. 

We have natural homomorphisms
\[
P_i\to\Hom (\fZ_{\sP_i}^*,W_n)\to\Hom (\fZ_{\sP_i}^*,\BG_a), \quad P\to\Hom (\fZ_{\sP}^*,W_n)\to\Hom (\fZ_{\sP}^*,\BG_a)
\]
and a commutative square
\begin{equation}   \label{e:Lau's diagram}
\xymatrix{
P=P_1\otimes P_2\ar[r]^-f \ar[d]_g & \Hom (\fZ_{\sP_1}^*,\BG_a)\otimes_R \Hom (\fZ_{\sP_2}^*,\BG_a)\ar[d]^\cong\\
\Hom (\fZ_{\sP}^*,\BG_a)\ar[r]^-h & \Poly-add (\fZ_{\sP_1}^*\times  \fZ_{\sP_2}^*,\BG_a)
}
\end{equation}
where the right vertical arrow is the isomorphism from Lemma~\ref{l:again tensor structure}. 
By Lemma~\ref{l:nasty}, $g$ is the canonical map $P\epi P/Q$, and $f$ is the canonical map  $P_1\otimes P_2\epi (P_1/Q_1)\otimes (P_2/Q_2)$.
Thus $f$ and $g$ are surjective maps with equal kernels. So $h$ is an isomorphism.
\end{proof}

\subsection{Proof of Proposition~\ref{p:Zink functor of a tensor product}}   \label{ss:proving Zink functor of a tensor product}
$\PPoly-add (\fZ_{\sP_1}^*\times \ldots \fZ_{\sP_l}^*,\BG_m)$ is a scheme of finite type over~$R$. So the scheme $H:=\PPoly-add (\fZ_{\sP_1}^*\times \ldots \fZ_{\sP_l}^*,\BG_m)^{(F^n)}$ is a finite group $R$-scheme killed by~$F^n$.
We have $\Lie (H)=\Poly-add (\fZ_{\sP_1}^*\times\ldots\fZ_{\sP_l}^*,\BG_a)$.

Let $H':=\fZ_{\sP_1\otimes\ldots\otimes\sP_l}$, then $\Lie (H')=\Lie (\HHom (\fZ_{\sP_1\otimes\ldots\otimes\sP_l}^*,\BG_m))=\Hom (\fZ_{\sP_1\otimes\ldots\otimes\sP_l}^*,\BG_a)$.

The homomorphism $\Lie (H')\to\Lie (H)$ corresponding to our homomorphism $f:H'\to H$ is the map 
\eqref{e:the Lie algebra map}, which is an isomorphism by Lemma~\ref{l:Lie algebra statement}. Moreover, for any $R$-algebra $\tilde R$, the map $\Lie (H'\otimes_R\tilde R)\to\Lie (H\otimes_R\tilde R)$ corresponding to $f$ is an isomorphism.

Thus $f:H'\to H$ is a closed immersion satitsfying the conditions of Lemma~\ref{l:technical lemma}. So $f$ is an isomorphism. \qed

\section{Explicit description of $\Lau_n^{d,d'}$}  \label{s:Explicit description}
\subsection{Formulation of the result}
\subsubsection{The goal}
Let $\Sm_n^r$ denote the stack of groupoids formed by commutative $n$-smooth group schemes of rank~$r$. 
Our goal is to describe the commutative group scheme $\Lau_n^{d,d'}$ on the stack $\Disp_n^{d,d'}$.
By Theorem~\ref{t:main}(ii), this group scheme is $n$-smooth of rank $d'(d-d')$. Thus $\Lau_n^{d,d'}$ corresponds to a morphism
$\Disp_n^{d,d'}\to\Sm_n^{d'(d-d')}$. Our goal is to describe this morphism explicitly.

\subsubsection{Remark}   \label{sss:rank of Zink}
By the second part of Lemma~\ref{l:Lie of Zink}(ii), Zink's functor defines a morphism
\[
\fZ :\sDisp_n^{d,d'}\to\Sm_n^{d'},
\]
where $\sDisp_n^{d,d'}$ is the stack of groupoids defined in \S\ref{ss:Algebraicity of sDisp_n}.

\subsubsection{A morphism $\Disp_n^{d,d'}\to\sDisp_n^{d^2,d'(d-d')}$}   \label{sss:the semidisplay on Disp}
There is a duality functor 
\begin{equation}   \label{e:the duality functor}
\Disp_n (R)^{\op}\iso\Disp_n (R), \quad \sP\mapsto\sP^t .
\end{equation}
For now, we use it as a ``black box''. Its definition (due to E.~Lau) will be given in \S\ref{sss:duality on DISP_n and Disp_n}.

The functor \eqref{e:the duality functor}  induces an isomorphism of stacks of groupoids
\[
\Disp_n^{d,d'}\iso\Disp_n^{d,d-d'}.
\]
Combining it with the morphisms $\Disp_n^{d,d'}\to\sDisp_n^{d,d'}$ and $\Disp_n^{d,d-d'}\to\sDisp_n^{d,d-d'}$, we get a morphism
\begin{equation}    \label{e:map to the product}
\Disp_n^{d,d'}\to\sDisp_n^{d,d'}\times\sDisp_n^{d,d-d'}.
\end{equation}
The tensor product from \S\ref{sss:sDisp_n(R) as a tensor category} gives a morphism
\begin{equation}  \label{e:the otimes morphism}
\sDisp_n^{d,d'}\times\sDisp_n^{d,d-d'}\overset{\otimes}\longrightarrow\sDisp_n^{d^2,d'(d-d')}.
\end{equation}
Composing \eqref{e:map to the product} and \eqref{e:the otimes morphism}, we get a morphism
\begin{equation}   \label{e:the semidisplay on Disp}
\Disp_n^{d,d'}\to\sDisp_n^{d^2,d'(d-d')}.
\end{equation}

\subsubsection{Remark}
In \S\ref{ss:the canonical semidisplay} we will give another description of the above morphism \eqref{e:the semidisplay on Disp}.

\begin{thm}   \label{t:Lau_n explicitly}
The $n$-smooth group scheme $\Lau_n^{d,d'}$ on $\Disp_n^{d,d'}$ corresponds to the composite morphism
\[
\Disp_n^{d,d'}\to\sDisp_n^{d^2,d'(d-d')}\overset{\fZ}\longrightarrow\Sm_n^{d'(d-d')},
\]
where the first arrow is \eqref{e:the semidisplay on Disp} and the second one is given by Zink's functor (see \S\ref{sss:rank of Zink}).
In other words, 
\[
\Lau_n^{d,d'}=\fZ (s(\sP_{\univ})\otimes s(\sP_{\univ}^t)),
\]
where $\sP_{\univ}$ is the universal $n$-truncated display, $\sP_{\univ}^t$ is its dual, $s$ is the functor from
$n$-truncated displays to $n$-truncated semidisplays, $\fZ$ is the Zink functor, and $\otimes$ is the tensor product from \S\ref{sss:sDisp_n(R) as a tensor category}.
\end{thm}

The proof will be given in \S\ref{ss:proof of explicit description of Lau_n}. It is based on the result of E.~Lau and T.~Zink described in the next subsection.

\subsection{The group schemes $\cA _n^{d,d'}$ and $\cB_n^{d,d'}$}
In \S\ref{ss:Lie of Lau} we defined commutative $n$-smooth group schemes $\cA _n^{d,d'}$ and $\cB_n^{d,d'}$ on $\Disp_n^{d,d'}$. The rank of 
$\cA _n^{d,d'}$ (resp.~$\cB_n^{d,d'}$) equals $d'$ (resp.~$d-d'$). Thus $\cA _n^{d,d'}$ corresponds to a morphism
\begin{equation}  \label{e:the map for cA_n}
\Disp_n^{d,d'}\to\Sm_n^{d'},
\end{equation}
and $\cB_n^{d,d'}$ corresponds to a morphism
\begin{equation}  \label{e:the map for cB_n}
\Disp_n^{d,d'}\to\Sm_n^{d-d'}.
\end{equation}

\begin{prop}    \label{p:LZ}
(i) The morphism \eqref{e:the map for cA_n} equals the composition 
$$\Disp_n^{d,d'}\to\sDisp_n^{d,d'}\overset{\fZ}\longrightarrow\Sm_n^{d'}.$$

(ii) The morphism \eqref{e:the map for cB_n} equals the composition 
$$\Disp_n^{d,d'}\iso\Disp_n^{d,d-d'}\to\sDisp_n^{d,d-d'}\overset{\fZ}\longrightarrow\Sm_n^{d-d'},$$
where the first arrow comes from the duality functor for $n$-truncated displays.
\end{prop}

\begin{proof}
Statement (i) follows from \cite[Lemma 3.12]{LZ} combined with Lemma~\ref{l:fully faithful}. By Lemma~\ref{l:fully faithful}, to deduce (ii) from (i) it suffices to prove commutativity of the diagram
\begin{equation}  \label{e:compatibility with duality}
\xymatrix{
\BBT_n^{d,d'}\ar[r]^-\sim \ar[d] & \BBT_n^{d,d-d'}\ar[d]\\
\Disp_n^{d,d'}\ar[r]^-\sim & \Disp_n^{d,d-d'}
}
\end{equation}
whose horizontal arrows are given by the duality functors; in other words, it suffices to show that the truncated display functor $\Phi_{n,R}$  from \cite[Prop.~4.1]{Lau13} commutes with duality.
As explained to me by E.~Lau, this can be easily deduced from a similar statement for the non-truncated display functor, which is known to be true (see \cite[Rem.~2.3]{Lau13}). One uses the following fact:
if an $n$-truncated Barsotti-Tate group $G$ is represented as a kernel of an isogeny of $p$-divisble groups $H_1\to H_2$ then $\Ker (H_2^\vee\to H_1^\vee )=G^*$; here $H_i^\vee$ is the dual $p$-divisible group,
i.e., $H_i^\vee =H_i^*\otimes (\BQ_p/\BZ_p)$, where $H_i^*$ is the Cartier dual.
\end{proof}

\subsubsection{Remark}  \label{sss:shortcut for n=1}
Combining Proposition~\ref{p:LZ} with Lemma~\ref{l:Lie of Zink}(ii), one gets an explicit description of the restricted Lie algebras $\Lie (\cA _n^{d,d'})$ and $\Lie (\cB_n^{d,d'})$ from Theorem~\ref{t:Lie of Lau}(ii).

\subsection{Proof of Theorem~\ref{t:Lau_n explicitly}}   \label{ss:proof of explicit description of Lau_n}  
Theorem~\ref{t:Lie of Lau}(iii) provides an isomorphism
\[
\Lau_n^{d,d'}\iso \PPoly-add ((\cA_n^{d,d'})^*\times (\cB_n^{d,d'})^*, \BG_m)^{(F^n)}.
\]
Proposition~~\ref{p:LZ} provides isomorphisms $\cA _n^{d,d'}\iso\fZ_{\sP_1}$, $\cB _n^{d,d'}\iso\fZ_{\sP_2}$, where $\sP_1,\sP_2$ are certain $n$-truncated semidisplays over the stack $\Disp_n^{d,d'}$.
Finally, by Proposition~\ref{p:Zink functor of a tensor product}, $\PPoly-add ((\cA_n^{d,d'})^*\times (\cB _n^{d,d'})^*, \BG_m)^{(F^n)}=\fZ_{\sP_1\otimes\sP_2}$. \qed

\section{The morphism $\Disp_n^{d,d'}\to\sDisp_n^{d^2,d'(d-d')}$ via higher displays}   \label{s:higher displays}
\subsection{The goal and plan of this section}
\subsubsection{The goal and the idea}   \label{sss:The goal and the idea}
In \S\ref{sss:the semidisplay on Disp} we defined a canonical morphism 
\begin{equation}  \label{e:2the semidisplay on Disp}
\Disp_n^{d,d'}\to\sDisp_n^{d^2,d'(d-d')}.
\end{equation}
Our goal is to describe this morphism in ``Shimurizable'' terms, i.e., in a way which makes it clear how to generalize\footnote{For the actual generalization, see \S\ref{ss:2from G-displays to semidisplays} of Appendix~\ref{s:Disp_n and Lau_n}.} \eqref{e:2the semidisplay on Disp} to a morphism $\Disp_n^{G,\mu}\to\sDisp_n$, 
where $G$ is a smooth affine group scheme over $\BZ/p^n\BZ$ and $\mu :\BG_m\to G$ is a 1-bounded homomorphism.

The idea is roughly as follows. The definition of the morphism \eqref{e:2the semidisplay on Disp} given in \S\ref{sss:the semidisplay on Disp}  involves something like $\rho\otimes\rho^*$, where $\rho$ is the standard $d$-dimensional representation of $GL(d)$.
Since $\rho\otimes\rho^*$ is the adjoint representation, it has an analog for any $G$ (while the decomposition of the adjoint representation as a tensor product is specific for $GL(d)$).

To implement this idea, we will use the notion of higher display, which was developed by Langer-Zink \cite{Langer-Zink} and then by Lau \cite{Lau21}.

\subsubsection{The plan}  \label{sss:plan for higher displays}
In \S\ref{ss:Lau equivalence}-\ref{ss:higher (pre)displays} we recall some material from \cite{Lau21}:
the $n$-truncated Witt frame, the category of finitely generated projective graded modules over it, and the category of $n$-truncated higher (pre)displays.
In \S\ref{ss:preDISP to sDisp} we construct a tensor functor from finitely generated projective higher predisplays to semidisplays.
In \S\ref{ss:preDISP=sDisp-strong} we discuss the interpretation of $\Disp_n$ and $\sDisp_n^{\strong}$ via higher predisplays.
In \S\ref{ss:the canonical semidisplay} we implement the idea from \S\ref{sss:The goal and the idea}.

\S\ref{ss:Tensor structure on a-restricted Vect}-\ref{ss:Zink complex} can be skipped by the reader.
The goal of \S\ref{ss:Tensor structure on a-restricted Vect}-\ref{ss:Tensor structure on a-restricted preDISP} is to introduce the tensor structure on $\sDisp_n^{\strong}(R)$ promised in \S\ref{s:n-truncated semidisplays},
and \S\ref{ss:Zink complex} will be used in a single sentence in~\S\ref{ss:Lau_n,G,mu}.

\subsection{The Lau equivalence} \label{ss:Lau equivalence} 
In this subsection we retell a part of E.~Lau's paper \cite{Lau21} (but not quite literally).

\subsubsection{The category $\cC$}  \label{sss:triples A,t,u} \label{sss:cC non-economic}
Let $\cC$ be the category of triples $(A,t,u)$, where $A=\bigoplus\limits _{i\in\BZ}A_i$ is a $\BZ$-graded ring and $t\in A_{-1}$, $u\in A_1$ are such that

(i) multiplication by $u$ induces an isomorphism $A_i\iso A_{i+1}$ for $i\ge 1$;

(ii) multiplication by $t$ induces an isomorphism $A_i\iso A_{i-1}$ for $i\le 0$.

Because of (i) and (ii), $\cC$ has an ``economic'' description. To formulate it, we will define a category $\cC^{\ec}$ (where ``ec'' stands for ``economic'') and construct an equivalence $\cC\iso\cC^{\ec}$.

\subsubsection{The category $\cC^{\ec}$} \label{sss:cC economic}
Let $\cC^{\ec}$ be the category of diagrams 
\begin{equation}   \label{e:objects of C ec}
A_0\underset{V}{\overset{F}\rightleftarrows} A_1,
\end{equation}
where $A_0$ and $A_1$ are rings, $F$ is a ring homomorphism, and $V$ is an additive map such that
\begin{equation}   \label{e:aVb}
a\cdot V(a')=V(F(a)a') \quad  \mbox{ for } a\in A_0, \, a'\in A_1
\end{equation}
and for $a'\in A_1$ we have
\begin{equation}   \label{e:FV=p}
F(V(a'))={\mathbf p}a', \quad  \mbox{ where } {\mathbf p}:=F(V(1))\in A_1.
\end{equation}
Note that by \eqref{e:aVb} we have $VF=V(1)$, which implies \eqref{e:FV=p} if $a'\in F(A_0)$ (but not in general).

\subsubsection{The functor $\cC\to\cC^{\ec}$}   \label{sss:from cC to economic category}
Given a triple $(A,t,u)\in \cC$, we construct a diagram \eqref{e:objects of C ec} as follows:

(i)  $A_0$ is the $0$-th graded component of $A$;

(ii) $A_1$ is the first graded component of $A$, and the product of $x,y\in A_1$ is as follows: first multiply $x$ by $y$ in $A$, then apply the isomorphism $A_2\iso A_1$ inverse to $u:A_1\iso A_2$; equivalently,
the product in $A_1$ comes from the product in $A/(u-1)A$ and the natural map $A_1\to A/(u-1)A$, which is an isomorphism by virtue of \S\ref{sss:triples A,t,u}(i);

(iii) $F:A_0\to A_1$ is multiplication by $u$, and $V:A_1\to A_0$ is multiplication by $t$.

\begin{prop}    \label{p:Lau equivalence}
The above functor $\cC\to\cC^{\ec}$ is an equivalence. The inverse functor $\fL :\cC^{\ec}\to\cC$ takes a diagram $A_0\underset{V}{\overset{F}\rightleftarrows} A_1$ to a certain graded subring of the graded ring 
$$A_0[t,t^{-1}]\times A_1[u,u^{-1}], \quad \deg t:=-1, \; \deg u=1;$$
namely, the $i$-th graded component of the subring is the set of pairs $(at^{-i},a'u^i)$, where $a\in A_0$ and $a'\in A_1$ satisfy the relation
\begin{equation}  \label{e:Lau equivalence}
a'={\mathbf p}^{-i}F(a) \mbox{ if } i\le 0, \quad a=V({\mathbf p}^{i-1}a')  \mbox{ if } i>0.
\end{equation}
(As before, $\mathbf p:=F(V(1))\in A_1$.) \qed
\end{prop}

The functor $\fL :\cC^{\ec}\to\cC$ will be called the \emph{Lau equivalence.}

The proof of the proposition is left to the reader. However, let us make some remarks.

\subsubsection{Remarks}   \label{sss:idea behind Lau equivalence}
(i) The description of $\fL$ from Proposition~\ref{p:Lau equivalence} is motivated by the following observation:
if $(A,t,u)\in \cC$ then the natural map $A\to A[1/t]\times A[1/u]$ is injective, $A[1/t]=A_0[t,t^{-1}]$, and $A[1/u]=A_1[u,u^{-1}]$, where the ring structure on $A_1$ is as in \S\ref{sss:from cC to economic category}(ii).

(ii) If $(A,t,u)\in \cC$ then the nonpositively graded part of $A$ identifies with $A_0[t]$ and the positively graded one identifies with $uA_1[u]$, where the ring structure on $A_1$ is as in \S\ref{sss:from cC to economic category}(ii).
So instead of describing $A$ as a subring of $A_0[t,t^{-1}]\times A_1[u,u^{-1}]$, one could describe $A$ as the group $A[t]\oplus uA_1[u]$ equipped with a ``tricky'' multiplication operation.

(iii) There exists a natural situation in which ${\mathbf p}\ne p$ and moreover, ${\mathbf p}\not\in p\cdot A_1^\times$. Namely, in the case $p=2$ this happens for the triple $(\SW ,F,\hV)$ mentioned in \S\ref{sss:Idea of the definition}.

\subsubsection{The Witt frame}   \label{sss:Witt frame}
For any ring $R$ the maps $F,V:W(R)\to W(R)$ satisfy the properties from \S\ref{sss:cC economic} (with $\mathbf p=p$). Applying the Lau equivalence to the diagram $W(R)\underset{V}{\overset{F}\rightleftarrows} W(R)$, one gets an object of $\cC$.
Following \cite{Lau21}, we call it the \emph{Witt frame}. Following \cite{Daniels}, we denote it by $W (R)^\oplus$ (in \cite[ Example 2.1.3]{Lau21} it is denoted by $\underline{W}(R)$).

\subsubsection{The $n$-truncated Witt frame}   \label{sss:2truncated Witt frame}
Let $n\in\BN$ and let $R$ be an $\BF_p$-algebra. Then we have a map $F:W_n(R)\to W_n(R)$ (in addition to $V:W_n(R)\to W_n(R)$).
Applying the Lau equivalence to the diagram $W_n(R)\underset{V}{\overset{F}\rightleftarrows} W_n(R)$, one gets an object of $\cC$.
Following \cite{Lau21}, we call it the \emph{$n$-truncated Witt frame}. Following \cite{Daniels}, we denote it by $W_n(R)^\oplus$ (in Example~2.1.6 of \cite{Lau21} it is denoted by $\underline{W_n}(R)$).

Let us note that $W (R)^\oplus$ and $W_n (R)^\oplus$ are particular examples of ``higher frames'' in the sense of \cite[\S 2]{Lau21}. We will not use more general higher frames.

\subsection{Recollections on the $n$-truncated Witt frame}   \label{ss:truncated Witt frame}
Let $n\in\BN$ and let $R$ be an $\BF_p$-algebra. Let us recall the material from \cite{Lau21} about the $n$-truncated Witt frame $W_n(R)^\oplus$.

\subsubsection{}  \label{sss:truncated Witt frame}
The definition of $W_n(R)^\oplus$ given in \S\ref{sss:2truncated Witt frame} amounts to the following. We
equip the rings $W_n(R)[t,t^{-1}]$ and $W_n(R)[u,u^{-1}]$ with the $\BZ$-grading such that $\deg t=-1$, $\deg u=1$, and $\deg a=0$ for all $a\in W_n(R)$. Then $$W_n(R)^\oplus\subset W_n(R)[u,u^{-1}]\times W_n(R)[t,t^{-1}]$$
is the graded subring whose $i$-th graded component consists of all pairs $(au^i, a't^{-i})$, where $a,a'\in W_n(R)$ are related as follows:
\begin{equation} \label{e:a' via a}
a=p^{-i} F(a')\mbox{ if } i\le 0,
\end{equation}
\begin{equation}    \label{e:a via a'}
a'=p^{i-1}V(a) \mbox{ if } i>0.
\end{equation}
It is easy to check that $W_n(R)^\oplus$ is indeed a subring of $W_n(R))[u,u^{-1}]\times W_n(R))[t,t^{-1}]$; this is believable because \eqref{e:a via a'} is an ``avatar'' of \eqref{e:a' via a} (since $FV=VF=p$).

\subsubsection{}  \label{sss:positive,nonpositive}
(i) The projection $W_n(R)^\oplus\to W_n(R)[t,t^{-1}]$ identifies the nonpositively graded part of $W_n(R)^\oplus$ with $W_n(R)[t]$.
In particular, the $0$-th graded component of $W_n(R)^\oplus$ identifies with $W_n(R)$.

(ii) The projection $W_n(R)^\oplus\to W_n(R)[u,u^{-1}]$ identifies the positively graded part of $W_n(R)^\oplus$ with $u\cdot W_n(R)[u]$.

\subsubsection{}
By \S\ref{sss:positive,nonpositive}, we can view $t,u$ as elements of $W_n(R)^\oplus$. As such, they satisfy the relation $tu=p$. The $0$-th graded component of $W_n(R)^\oplus$  identifies with $W_n(R)$.
So we get a graded homomorphism $W_n(R)[t,u]/(tu-p)\to W_n(R)^\oplus$. If $R$ is perfect, this is an isomorphism.

\subsubsection{}  \label{sss:two localizations of the truncated Witt frame}
As before, consider $t,u$ as elements of $W_n(R)^\oplus$. Then we have canonical isomorphisms
\[
W_n(R)^\oplus [u^{-1}]\iso W_n(R)[u,u^{-1}], \quad W_n(R)^\oplus [t^{-1}]\iso W_n(R)[t,t^{-1}]
\]
induced by the projections $W_n(R)^\oplus\to W_n(R)[u,u^{-1}]$ and $W_n(R)^\oplus\to W_n(R)[t,t^{-1}]$.

\subsubsection{}  \label{sss:R as a quotient of our frame}
The quotient of $W_n(R)^\oplus$ by the ideal generated by its graded components of nonzero degrees identifies with $W_n(R)/V(W_n(R))=R$. Thus we get a canonical homomorphism $W_n(R)^\oplus\to R$.

\subsubsection{The homomorphisms $\sigma,\tau :W_n(R)^\oplus\to W_n(R)$}   \label{sss:sigma,tau}
Let $\tau: W_n(R)^\oplus\to W_n(R)$ be the compo\-site map
\[
W_n(R)^\oplus\to W_n(R)[t,t^{-1}]\overset{t=1}\longrightarrow  W_n(R).
\]
Let $\sigma: W_n(R)^\oplus\to W_n(R)$ be the compo\-site map
\[
W_n(R)^\oplus\to W_n(R)[u,u^{-1}]\overset{u=1}\longrightarrow  W_n(R).
\]
The homomorphisms $\sigma,\tau :W_n(R)^\oplus\to W_n(R)$ are surjective, and one has
\[
\Ker\tau=(t-1)W_n(R)^\oplus , \quad \Ker\sigma=(u-1)W_n(R)^\oplus .
\]

The restriction of $\tau$ to the $0$-th graded component $W_n(R)\subset W_n(R)^\oplus$ equals the identity map.
But the restriction of $\sigma$ to $W_n(R)\subset W_n(R)^\oplus$ equals $F:W_n(R)\to W_n(R)$. Presumably, this was the motivation for introducing the notation $\sigma$ in the paper \cite{Lau21}
(where the Witt vector Frobenius is also denoted by $\sigma$).

\subsubsection{Relation with the notation of Bhatt's lectures \cite{Bh}}
In \cite[\S 3.3]{Bh} the letters $t,u$ have essentially the same meaning as above. However, his grading  is opposite (i.e., $\deg t=1$, $\deg u=-1$).

\subsection{The category $\Vect_n(R)$}   \label{ss:graded projective modules}
\subsubsection{}   \label{sss:Vect_n(R)}
Let  $R$ be an $\BF_p$-algebra. Let $S:=W_n(R)^\oplus$. Let $S_i$ be the $i$-th graded component of $S$, so $S_0=W_n(R)$. If $M=\bigoplus\limits_j M_j$ is a graded $S$-module, we write\footnote{This notation is motivated by the following. A graded $S$-module is the same as an $\cO$-module on the quotient stack $(\Spec S)/\BG_m$ . If $R$ is perfect this stack is the Nygaard-filtered prismatization $R^\cN$ (see \cite{Bh}), and our notation $M\{i\}$ agrees with the notation for Breuil-Kisin twists.} $M\{ i\}$ for the graded $S$-module whose $j$-th graded component is $M_{i+j}$.

Let $\Vect_n(R)$ be the $S_0$-linear category of finitely generated projective graded $S$-modules. By \cite[Lemma 3.1.4]{Lau21}, every $M\in\Vect_n(R)$ can be represented as
\begin{equation}  \label{e:decomposing M}
M=\bigoplus_i L_i\otimes_{S_0}S\{ -i\},
\end{equation}
where $L_i$'s are finitely generated projective $S_0$-modules. As noted in \cite[\S 3]{Lau21}, the ranks of $L_i$'s do not depend on the choice of a decomposition \eqref{e:decomposing M}
(to see this, use the homomorphism $S\to R$ from \S\ref{sss:R as a quotient of our frame}).

\subsubsection{}  \label{sss:the effective subcategory}
For every $a\ge 0$, let $\Vect_n^{[0,a]}(R)$ be the full subcategory of all $M\in\Vect_n(R)$ such that $L_i\ne 0$ only if $0\le i\le a$. Let $\Vect_n^+(R)$ be the union of $\Vect_n^{[0,a]}(R)$ for all $a\ge 0$.
It is easy to see that an object $M\in\Vect_n(R)$ belongs to $\Vect_n^+(R)$ if and only if the map $t:M_i\to M_{i-1}$ is an isomorphism for all $i\le 0$; this condition is often called \emph{effectivity.}

\subsubsection{Tensor structure}    \label{sss:obvious tensor structure}
$\Vect_n(R)$ is a rigid tensor category. $\Vect_n^+(R)$ is a tensor subcategory of $\Vect_n(R)$, which is not rigid if $R\ne 0$.
The subcategory $\Vect_n^{[0,a]}(R)\subset \Vect_n^+(R)$ is not closed under tensor product if $a>0$ and $R\ne 0$. 
(Nevertheless, $\Vect_n^{[0,a]}(R)$ has a natural structure of tensor category, see \S\ref{sss:non-obvious tensor structure} below.)

\subsection{$n$-truncated higher (pre)displays}   \label{ss:higher (pre)displays}
Let us recall the notion of higher (pre)display from \cite[Def.~3.2.1]{Lau21}. 

\subsubsection{Definitions}   \label{sss:preDISP}
If $M$ is a $W_n(R)^\oplus$-module then we write $M^\sigma ,M^\tau$ for the $W_n(R)$-modules obtained from $M$ by base change via the homomorphisms
$\sigma,\tau :W_n(R)^\oplus\to W_n(R)$ from \S\ref{sss:sigma,tau}. Let $\PREDISP_n(R)$ be the category of pairs $(M,f)$, where $M$ is a graded $W_n(R)^\oplus$-module and $f\in\Hom (M^\sigma ,M^\tau)$.
Let $\preDISP_n(R)\subset\PREDISP_n(R)$ be the full subcategory defined by the condition $M\in\Vect_n(R)$. Let $\DISP_n(R)$ be the full subcategory of objects
$(M,f)\in\preDISP_n(R)$ such that $f$ is an isomorphism; thus $\DISP_n(R)$ is the category
whose objects are pairs 
\begin{equation}   \label{e:(M,f)}
(M,f), \quad M\in\Vect_n(R), \; f\in\Isom (M^\sigma ,M^\tau).
\end{equation}
Objects of $\PREDISP_n(R)$ (resp.~$\DISP_n(R)$) are called \emph{$n$-truncated
higher predisplays} (resp.~\emph{displays}) over $R$.

Let $\DISP_n^+(R)\subset\DISP_n(R)$ be the full subcategory defined by the condition $M\in\Vect_n^+(R)$. The categories
$\DISP_n^{[0,a]}(R), \preDISP_n^+(R), \preDISP_n^{[0,a]}(R)$ are defined similarly.

The functors $M\mapsto M^\sigma$ and $M\mapsto M^\tau$ are tensor functors. So $\preDISP_n(R)$, $\preDISP_n^+(R)$, $\DISP_n(R)$, $\DISP_n^+(R)$ are tensor categories.

\subsubsection{On $M^\sigma$ and $M^\tau$}   \label{sss:the 2 basechanges}
Given a graded $W_n(R)^\oplus$-module $M$, set
\[
M_{-\infty}:=M/(t-1)M= \underset{\longrightarrow}\lim (M_0\overset{t}\longrightarrow M_{-1}\overset{t}\longrightarrow\ldots), 
\]
\[
M_\infty:= M/(u-1)M=\underset{\longrightarrow}\lim (M_0\overset{u}\longrightarrow M_{1}\overset{u}\longrightarrow\ldots).
\]
Define a $W_n(R)$-structure on $M_{\pm\infty}$ via the embedding $W_n(R)\mono W_n(R)^\oplus$.
Then $M_{-\infty}=M^\tau$, and $M_\infty$ is obtained from $M^\sigma$ by restriction of scalars via $F:W_n(R)\to W_n(R)$.

$\Hom (M^\sigma ,M^\tau)$ identifies with the group of $\sigma$-linear maps $M\to M_{-\infty}$. 

\subsubsection{Remark}
One can identify $\Hom (M^\sigma ,M^\tau)$ with a \emph{subgroup} of the group of $F$-linear maps $M_\infty\to M_{-\infty}$. 
In fact, an $F$-linear map $M_\infty\to M_{-\infty}$ is the same as an additive map $g:M^\sigma \to M^\tau=M_{-\infty}$ such that $g(F(a)x)=F(a)g(x)$ for $a\in W_n(R)$, $x\in M^\sigma$.
This condition for $g$ is weaker than $W_n(R)$-linearity.

\subsection{The tensor functor $\preDISP_n^+(R)\to\sDisp_n(R)$}   \label{ss:preDISP to sDisp}
The definition of $\sDisp_n(R)$ was given in \S\ref{ss:n-truncated semidisplays}. 
As before, we will use the notation $S:=W_n(R)^\oplus$ and the notation $S_i$ for the $i$-th graded component of $S$, so $S_0=W_n(R)$.

\subsubsection{The functor $\preDISP_n^+(R)\to\sDisp_n(R)$}  \label{sss:preDISP to sDisp}
Given $(M,f)\in\preDISP_n^+(R)$, we will define an object $(P,Q,F,F_1)\in \sDisp_n(R)$.

Since $M\in\Vect_n^+(R)$, we have a decomposition
\begin{equation}  \label{e:2decomposing M}
M=\bigoplus_{i\ge 0} L_i\otimes_{S_0}S\{ -i\},
\end{equation}
where each $L_i$ is a finitely generated $S_0$-module and $L_i=0$ for $i$ big enough. Set $$P:=M_0, \quad Q:=tM_1\subset P.$$
By \eqref{e:2decomposing M}, $P$ is a finitely generated $S_0$-module and $P/Q$ is a finitely generated projective module over $S_0/tS_1=R$.

By \S\ref{sss:the 2 basechanges}, we can think of $f\in\Hom (M^\sigma ,M^\tau)$ as a $\sigma$-linear map $M\to M_{-\infty}=M_0=P$.
Restricting it to $M_0$ and $M_1$, we get $F$-linear maps $F:P\to P$ and $\tilde F_1:M_1\to P$.
We will show that
\begin{equation}  \label{e:the key inclusion}
\tilde F_1 (\Ker (M_1\overset{t}\longrightarrow M_0))\subset J_{n,R}\cdot P, \quad \mbox{ where } J_{n,R}:=\Ker (W_n(R)\overset{V}\longrightarrow W_n(R))
\end{equation}
and therefore $\tilde F_1$ induces an $F$-linear map $F_1:Q\to P/J_{n,R}\cdot P$.

By \eqref{e:2decomposing M}, $\Ker (M_1\overset{t}\longrightarrow M_0)=S'_1\cdot L_0$, where $S'_1:=\Ker (S_1\overset{t}\longrightarrow S_0)$. So to prove \eqref{e:the key inclusion}, it suffices to show that 
$\sigma (S'_1)=J_{n,R}\,$. To see this, apply \eqref{e:a via a'} for $i=1$.

One checks that $(P,Q,F,F_1)\in \sDisp_n(R)$.

\subsubsection{Tensor structure}  \label{sss:tensor structure on Phi}
Both $\preDISP_n^+(R)$ and $\sDisp_n(R)$ are tensor categories (in the case of $\sDisp_n$, see \S\ref{ss:tensoring truncated semidisplays}).
Let us upgrade the functor $\preDISP_n^+(R)\to\sDisp_n(R)$ constructed in  \S\ref{sss:preDISP to sDisp} to a tensor functor.

If $M,M'\in\Vect_n(R)$ and $M''=M\otimes M'$ then we have a morphism
\begin{equation}  \label{e:beta}
\beta :M_0\otimes M'_0\to M''_0.
\end{equation}
If $M,M'\in\Vect_n^+(R)$ then \eqref{e:beta} is an isomorphism and 
$$\beta ((tM_1)\otimes M'_0+M_0\otimes tM'_1)=tM_1''$$
(it suffices to check these statements if $M=S\{ -i\}, M'=S\{ -i'\}$, where $i,i'\ge 0$).

Now let $(M,f)\in\preDISP_n^+(R)$ and $(M',f')\in\preDISP_n^+(R)$. Let 
$$(M'',f''):=(M,f)\otimes (M',f').$$
Then we have $F$-linear maps $F:M_0\to M_0$ and $F_1:tM_1\to M_0/J_{n,R}\cdot M_0$ defined in \S\ref{sss:preDISP to sDisp}.
We also have similar maps
\[
F':M'_0\to M'_0, \quad F'_1:tM'_1\to M'_0/J_{n,R}\cdot M'_0, \quad F'':M''_0\to M''_0, \quad F''_1:tM''_1\to M''_0/J_{n,R}\cdot M''_0.
\]
It remains to check the following properties of the map \eqref{e:beta}:
\[
F''(\beta (x\otimes x'))=\beta (F(x)\otimes F(x')) \quad \mbox{ for } x\in M_0, \; x'\in M'_0,
\]
\[
F_1''(\beta (x\otimes x'))=\beta (F(x)\otimes F'_1(x')) \quad \mbox{ for } x\in M_0, \; x'\in tM'_1,
\]
\[
F_1''(\beta (x\otimes x'))=\beta (F_1(x)\otimes F'(x')) \quad \mbox{ for } x\in tM_1, \; x'\in M'_0.
\]
This is straightforward.

\subsection{The equivalence $\DISP_n^{[0,1]}(R)\iso\Disp_n(R)$}  \label{ss:preDISP=sDisp-strong}
The categories $$\Disp_n(R) \mbox { and }\sDisp_n^{\strong}(R)$$  were discussed in \S\ref{ss:n-truncated displays}.

\subsubsection{}     \label{sss:upgrading the functor}
In \S\ref{sss:preDISP to sDisp} we defined a functor $\preDISP_n^+(R)\to\sDisp_n(R)$.
One upgrades it to a functor
\begin{equation}   \label{e:preDISP to sDisp-strong}
\preDISP_n^+(R)\to\sDisp_n^{\strong}(R)
\end{equation}
by making the following modifications. First, instead of setting $Q:=tM_1\subset M_0$ we set $Q:=M_1$.
Second, the map $\varepsilon :I_{n+1,R}\otimes_{W_n(R)}P\to Q$ mentioned in formula~\eqref{e:P,Q,iota,epsilon,F,F_1} is defined
by identifying $I_{n+1,R}$ with $S_1$.

\subsubsection{}   \label{sss:preDISP=sDisp-strong}
It is known to the experts that 

(i) the functor \eqref{e:preDISP to sDisp-strong} induces equivalences
\begin{equation}     \label{e:preDISP=sDisp-strong}
\preDISP_n^{[0,1]}(R)\iso\sDisp_n^{\strong}(R),
\end{equation}
\begin{equation}  \label{e:DISP=Disp}
\DISP_n^{[0,1]}(R)\iso\Disp_n(R);
\end{equation}

(ii) the construction of \S\ref{sss:upgrading the functor} gives a functor from $\Vect_n^{[0,1]}(R)$ to the category of ``truncated pairs'' from \cite[Def.~3.2]{Lau13}; moreover, this functor is an equivalence.

The verification of (ii) is straightforward, and (i) easily follows.
In the case $n=\infty$ the equivalence \eqref{e:DISP=Disp} is constructed in  \cite[Lemma 2.25]{Daniels}.

\subsubsection{Tensor structure on $\sDisp_n^{\strong}(R)$} \label{sss:tensor structure on the strong one}
In \S\ref{ss:Tensor structure on a-restricted preDISP} we will see that $\preDISP_n^{[0,a]}(R)$ is a tensor category. So \eqref{e:preDISP=sDisp-strong} yields a structure of tensor category on $\sDisp_n^{\strong}(R)$.

\subsection{The morphism $\Disp_n^{d,d'}\to\sDisp_n^{d^2,d'(d-d')}$}    \label{ss:the canonical semidisplay} 
\subsubsection{Duality on $\DISP_n^{[0,1]}$ and $\Disp_n$}   \label{sss:duality on DISP_n and Disp_n}
The category $\DISP_n^{[0,1]}(R)$ is equipped with a duality functor
\begin{equation}     \label{e:duality via DISP}
\sP\mapsto  \sP^t, \quad \sP^t:=\sP^*\{ -1\},
\end{equation}
where $\sP^*:=\Hom_S (\sP,S)$ and $S:=W_n(R)^\oplus$. One has $(\sP^t)^t=\sP$.

Using \eqref{e:DISP=Disp}, we get from \eqref{e:duality via DISP} a duality functor on $\Disp_n (R)$; the latter was used in \S\ref{sss:the semidisplay on Disp} as a ``black box''.

The above definition of the duality on $\Disp_n (R)$ was communicated to me by E.~Lau. His original definition (briefly mentioned in \cite[Rem.~4.4]{Lau13} and inspired by  \cite[Def.~19]{Zink}) was equivalent to it.

\subsubsection{}
In \S\ref{sss:the semidisplay on Disp} we defined a morphism $\Disp_n^{d,d'}\to\sDisp_n^{d^2,d'(d-d')}$ by associating to a display $\sP\in\Disp_n^{d,d'}(R)$ the tensor product
of the semidisplays corresponding to $\sP$ and $\sP^t$, where $\sP^t$ is the dual display.

By \S\ref{sss:preDISP=sDisp-strong}, we can think of $\sP$ as an object of $\DISP_n^{[0,1]}(R)$. Then the construction of \S\ref{sss:the semidisplay on Disp} can be written as
\begin{equation}  \label{e:original formula}
\sP\mapsto \Phi (\sP )\otimes \Phi (\sP^t ),  \quad \sP^t=\sP^*\{ -1\}, 
\end{equation}
where $\Phi :\preDISP_n^+(R)\to\sDisp_n(R)$ is the functor from \S\ref{sss:preDISP to sDisp}.

By \S\ref{sss:tensor structure on Phi}, $\Phi$ is a tensor functor. So formula \eqref{e:original formula} can be rewritten as
\begin{equation}   \label{e:new formula}
\sP\mapsto \Phi (\sP \otimes \sP^*\{ -1\} ).
\end{equation}

\subsubsection{}
Suppose that the graded projective $W_n(R)^\oplus$-module underlying $\sP$ has rank $d$.
Such a module is the same as a $\BG_m$-equivariant $GL(d)$-torsor on $\Spec W_n(R)^\oplus$.
Thus $\sP$ can be viewed as a $\BG_m$-equivariant $GL(d)$-torsor on $\Spec W_n(R)^\oplus$ equipped with an additional structure (the latter corresponds to $f$ from formula \eqref{e:(M,f)}).
$\sP \otimes \sP^*$ is just the vector bundle corresponding to this $GL(d)$-torsor and the adjoint representation of $GL(d)$.

\subsection{Tensor structure on $\Vect_n^{[0,a]}(R)$}   \label{ss:Tensor structure on a-restricted Vect}
The reader may prefer to skip the remaining part of~\S\ref{s:higher displays}. 
The goal of \S\ref{ss:Tensor structure on a-restricted Vect}-\ref{ss:Tensor structure on a-restricted preDISP} is to define a tensor structure on the category $\preDISP_n^{[0,a]}(R)$
(and therefore on $\sDisp_n^{\strong}(R)$), and \S\ref{ss:Zink complex} will be used in a single sentence in \S\ref{ss:Lau_n,G,mu}.

\begin{lem}   \label{l:3 equivalent properties}
Let $f:M'\to M$ be a morphism in $\Vect_n^+(R)$. The following properties of $f$ are equivalent:

(i) $f$ induces an isomorphism $\Hom (M'', M')\iso \Hom (M'', M)$ for all $M''\in\Vect_n^{[0,a]}(R)$;

(ii) $f$ induces an isomorphism $\Hom (S\{ -i\}, M')\iso \Hom (S\{ -i\}, M)$ for $i=0,1,\ldots ,a$;

(iii) $f$ induces an isomorphism $M'_i\to M_i$ for all $i\le a$. \qed
\end{lem}

\begin{lem}   \label{l:Pi_a}
(i) The inclusion $\Vect_n^{[0,a]}(R)\mono \Vect_n^+(R)$ has a right adjoint $$\Pi_a:\Vect_n^+(R)\to\Vect_n^{[0,a]}(R).$$

(ii) The functor $\Pi_a$ identifies $\Vect_n^{[0,a]}(R)$ with the localization of $\Vect_n^+(R)$ by the morphisms $M'\to M$ inducing an isomorphism $M'_i\to M_i$ for each $i\le a$.
\end{lem}

\begin{proof}
Let us prove (i). Given $M\in\Vect_n^+(R)$, we have to construct a pair $(M',f)$, where 
$M'\in\Vect_n^{[0,a]}(R)$ and $f\in\Hom (M',M)$ satisfies the equivalent conditions of Lemma~\ref{l:3 equivalent properties}.
We can assume that $M=S\{ -i\}$, $i>a$. In this case let $M'=S\{ -a\}$ and let $f$ be equal to $t^{i-a}\in S_{a-i}=\Hom (M',M)$.

Statement (ii) follows from Lemma~\ref{l:3 equivalent properties} and the fact that $\Pi_a$ is right adjoint to the inclusion $\Vect_n^{[0,a]}(R)\mono \Vect_n^+(R)$.
\end{proof}

\begin{lem}   \label{l:needed to define tensor product}
Let $M,M',N\in \Vect_n^+(R)$. Let $\tilde M:=M\otimes N$, $\tilde M':=M'\otimes N$.
If $f:M\to M'$ induces an isomorphism $M_i\to M'_i$ for all $i\le a$ then the morphism $\tilde M\to\tilde M'$ corresponding to $f$ induces an isomorphism
$\tilde M_i\to \tilde M'_i$ for all $i\le a$.
\end{lem}

\begin{proof}
This is clear if $N=S\{ -j\}$, $j\ge 0$. The general case follows.
\end{proof}

\subsubsection{Tensor structure on $\Vect_n^{[0,a]}(R)$}   \label{sss:non-obvious tensor structure}
Recall that $\Vect_n^+(R)$ is a tensor category. The subcategory $\Vect_n^{[0,a]}(R)\subset \Vect_n^+(R)$ is not closed under tensor product if $a>0$ and $R\ne 0$.
However, Lemma~\ref{l:Pi_a}(ii) combined with Lemma~\ref{l:needed to define tensor product} provides a structure of tensor category on $\Vect_n^{[0,a]}(R)$ and a structure of tensor functor
on the functor $\Pi_a:\Vect_n^+(R)\to\Vect_n^{[0,a]}$. 
Explicitly, for $M_1,M_2\in\Vect_n^{[0,a]}(R)$ one has
\begin{equation}  \label{e:truncated tensor product}
M_1\underset{a}\otimes M_2=\Pi_a (M_1\otimes M_2),
\end{equation}
where $\underset{a}\otimes$ (resp.~$\otimes$) denotes the tensor product in $\Vect_n^{[0,a]}(R)$ (resp.~in $\Vect_n^+(R)$).

\subsection{Tensor structure on $\preDISP_n^{[0,a]}(R)$}  \label{ss:Tensor structure on a-restricted preDISP}
\subsubsection{$\Pi_a$ as a functor $\preDISP_n^+(R)\to\preDISP_n^{[0,a]}(R)$} 
Let $M\in\preDISP_n^+(R)$. Define $M'\in\Vect_n^{[0,a]}(R)\subset\Vect_n^+(R)$ by $M':=\Pi_a(M)$, where $\Pi_a$ is as in Lemma~\ref{l:Pi_a}.
In $\Vect_n^+(R)$ we have a canonical morphism $M'\to M$. It induces isomorphisms $M'_i\iso M_i$ for $i\le a$ and therefore an isomorphism
$(M')^\tau=M'_{-\infty}\iso M_{-\infty}=M^\tau$. The composite map
\[
(M')^\sigma\to M^\sigma\to M^\tau\iso (M')^\tau
\]
makes $M'$ into an object of $\preDISP_n^{[0,a]}(R)$. Thus we have defined $\Pi_a$ as a functor $$\preDISP_n^+(R)\to\preDISP_n^{[0,a]}(R).$$
The functor $\Pi_a: \preDISP_n^+(R)\to\preDISP_n^{[0,a]}(R)$ is right adjoint to the inclusion $$\preDISP_n^{[0,a]}(R)\mono\preDISP_n^+(R).$$

\subsubsection{Tensor structure on $\preDISP_n^{[0,a]}(R)$} 
Now one gets a structure of tensor category on $\preDISP_n^{[0,a]}(R)$ and a structure of tensor functor
on the functor $$\Pi_a:\preDISP_n^+(R)\to\preDISP_n^{[0,a]}$$ just as in \S\ref{sss:non-obvious tensor structure}, see formula~\eqref{e:truncated tensor product}.

\subsection{The Zink complex of an object of $\preDISP_n(R)$}   \label{ss:Zink complex}
To any $M\in\preDISP_n(R)$ we will associate a complex $C_M^\bcdot$ of commutative group ind-schemes over $R$. If $M$ is in $\preDISP_n^+(R)$ then $C_M^\bcdot =C_\sP^\bcdot$, where
$\sP$ is the semidisplay corresponding to $M$ (in the sense of \S\ref{ss:preDISP to sDisp}) and $C_\sP^\bcdot$ is given by formula~\eqref{e:Zink complex}. We will use the notation of \S\ref{sss:preDISP}.

\subsubsection{The complex $K_N^\bcdot$}   \label{sss:d=alpha-fbeta}
Let $(N,f)\in\PREDISP_n(R)$; according to \S\ref{sss:preDISP}, this means that $N$ is a graded $W_n(R)^\oplus$-module and $f\in\Hom (N^\sigma ,N^\tau)$. 
We have canonical maps $\alpha:N_0\to N^\tau$ and $\beta:N_0\to N^\sigma$, where $\alpha$ is $W_n(R)$-linear and $\beta$ is $F$-linear.
Define $K_N^\bcdot$ to be the following complex of abelian groups: 

(i) $K_N^i=0$ for $i\ne 0,-1$, $K_N^{-1}=N_0$, $K_N^0=N^\tau$,

(ii) the differential $d:K_N^{-1}\to K_N^0$ is the additive map $\alpha -f\circ\beta:N_0\to N^\tau$.

\subsubsection{Remark}  \label{sss:the subgroup of Cartesian square}
Suppose that $(N,f)\in\PREDISP_n(R)$ satisfies the following conditions:

(i) $f$ is an isomorphism;

(ii) $\Ker (N\overset{t}\longrightarrow N)\cap \Ker (N\overset{u}\longrightarrow N)=0$.

\noindent
Then the map $(\alpha ,f\circ\beta) :N_0\to N^\tau\times N^\tau$ is injective, so it identifies $K_N^{-1}=N_0$ with a subgroup of $K_N^0\times K_N^0$. After this identification, $d:K_N^{-1}\to K_N^0$
is just the difference between the two projections $K_N^{-1}\to K_N^0$.

\subsubsection{Cohomological interpretation}
If $f:N^\sigma\to N^\tau$ is an isomorphism then $(N,f)$ can be interpreted as an $\cO$-module on a certain stack $\sY_R$, and $K_N^\bcdot [-1]$ computes the cohomology of this $\cO$-module.
The stack $\sY_R$ is obtained from the quotient stack $(\Spec W_n(R)^\oplus )/\BG_m$ by gluing together the disjoint open substacks $(\Spec W_n(R)^\oplus [t^{-1}])/\BG_m$ and
$(\Spec W_n(R)^\oplus [u^{-1}])/\BG_m$, which are both equal to $\Spec W_n(R)$ by \S\ref{sss:two localizations of the truncated Witt frame}.
If $R$ is perfect then $\sY_R=R^{\Syn}\otimes \BZ/p^n\BZ$, where $R^{\Syn}$ is the \emph{syntomification} stack from \cite{Bh}. 

\subsubsection{A canonical object of $\PREDISP_n(R)$}  \label{sss:canonical object of PREDISP_n}
Let $W(R)^\oplus$ be the non-truncated Witt frame, see \S\ref{sss:Witt frame}. To describe $W(R)^\oplus$ explicitly, one just replaces $W_n$ by $W$ in all formulas from \S\ref{sss:truncated Witt frame}.

As a group, each graded component of $W(R)^\oplus$ is isomorphic to $W(R)$. Replacing each $W(R)$ by $\hat W^{(F^n)}(R)$ (where $\hat W^{(F^n)}\subset W$ is as in \S\ref{sss:Zink's diagram}), one gets a graded ideal of $W(R)^\oplus$, which we denote by $\hat W^{(F^n)}(R)^\oplus$.
The $W(R)^\oplus$-module $\hat W^{(F^n)}(R)^\oplus$ is, in fact, a $W_n(R)^\oplus$-module. Moreover, one has a canonical isomorphism 
$$(\hat W^{(F^n)}(R)^\oplus)^\sigma\iso (\hat W^{(F^n)}(R)^\oplus)^\tau,$$
so $\hat W^{(F^n)}(R)^\oplus$ is an object of $\PREDISP_n(R)$. (Usually, it is not in $\preDISP_n(R)$.) 

\subsubsection{The Zink complex}  \label{sss:Zink complex}
Let $M\in\PREDISP_n(R)$. For any $R$-algebra $\tilde R$, we set 
\[
C_M^\bcdot (\tilde R):=K_{N(\tilde R)}^\bcdot ,
\]
where $N(\tilde R)$ is the tensor product of $M\{ 1\}$ and $\hat W^{(F^n)}(\tilde R)^\oplus$ over $W_n(R)^\oplus$ (the goal of shifting the grading of $M$ in the definition of $N(\tilde R)$ is to simplify the formulation of Lemma~\ref{l:comparing with the old Zink complex}). The complex $C_M^\bcdot$ will be called the \emph{Zink complex} of $M$. Each of its terms is a  functor from $R$-algebras to abelian groups.
According to the next lemma, the functors $C_M^i$ are nice if $M\in\preDISP_n(R)$ (i.e., if $M$ is fintely generated and projective as a $W_n(R)^\oplus$-module).

\begin{lem}
Let $M\in\preDISP_n(R)$. Let $r$ be the rank of $M$ viewed as a finitely generated projective $W_n(R)^\oplus$-module.

(i) For each $i$, the functor $C_M^i$  is an ind-scheme over $R$.

(ii) If $i\in\{0,-1\}$ then Zariski-locally on $\Spec R$, the group ind-scheme $C_M^i$ is isomorphic to the direct sum of $r$ copies of $\hat W^{(F^n)}_R$. If $i\not\in\{0,-1\}$ then $C_M^i=0$.
\end{lem}

\begin{proof}
Use the decomposition \eqref{e:decomposing M}.
\end{proof}

\begin{lem}   \label{l:comparing with the old Zink complex}
Let $M\in\preDISP_n^+(R)$. Then one has a canonical isomorphism $$C_M^\bcdot\iso C_\sP^\bcdot,$$ where
$\sP$ is the semidisplay corresponding to $M$ (in the sense of \S\ref{ss:preDISP to sDisp}) and $C_\sP^\bcdot$ is given by formula~\eqref{e:Zink complex}. \qed
\end{lem}

\begin{cor}
Let $M\in\preDISP_n(R)$.

(i) If $M\in\preDISP_n^+(R)$ then $C_M^\bcdot$ is quasi-isomorphic to $\fZ_\sP$, where $\sP$ is the semidisplay corresponding to $M$ and $\fZ_\sP\in\Sm_n(R)$ is defined by formula~\eqref{e:Zink functor}.

(ii) If $M\{1\}\in\preDISP_n^+(R)$ then $C_M^\bcdot$ is acyclic.

\end{cor}

Here the words ``acyclic'' and ``quasi-isomorphic'' are understood in the sense of complexes of functors $\{R\mbox{-algebras}\}\to\Ab$.

\begin{proof}
Statement (i) immediately follows from Lemma~\ref{l:comparing with the old Zink complex} and the definition of $\fZ_\sP$.
In the situation of (ii), the semidisplay $\sP=(P,Q,F,F_1)$ satisfies the condition $Q=P$. This condition implies that $\fZ_\sP=0$ by Lemma~\ref{l:Lie of Zink}(ii).
\end{proof}

\subsubsection{Remarks}   \label{sss:Ker d could be nonzero}
(i) If $M\not\in\preDISP_n^+(R)$ then $H^{-1}(C_M^\bcdot )$ can be nonzero, see formula~\eqref{e:lower bound} below.

(ii) Let $N(\tilde R)$ be as in \S\ref{sss:Zink complex}. If $M\in\DISP_n(R)$ then $N(\tilde R)$ satisfies the two conditions\footnote{To check condition (ii), note that $\Ker (L\overset{t}\longrightarrow L)\cap \Ker (L\overset{u}\longrightarrow L)=0$ if $L= (\hat W^{(F^n)})^\oplus$.}
from \S\ref{sss:the subgroup of Cartesian square}. So we can identify the group $K_N^{-1}$ with a subgroup of $K_N^0\times K_N^0$, after which $d:K_N^{-1}\to K_N^0$
is just the difference between the two projections $K_N^{-1}\to K_N^0$.
 
\subsubsection{The case $n=1$}  \label{sss:Zink complex for n=1}
By \cite[Example 3.6.4]{Lau21}, $\DISP_1(R)$ is canonically equivalent\footnote{Warning: in \cite[Example 3.6.4]{Lau21} the filtration $C^*$ is always descending and the filtration $D_*$ is always ascending, even if stated otherwise!
(The confusion in loc.\,cit. is purely terminological.)
} to the category of \emph{$F$-zips} over $R$.
An $F$-zip over $R$ is an $R$-module $L$ equipped with a descending filtration $\Fil^\bcdot L$, an ascending filtration $\Fil_\bcdot L$,  and an isomorphism
\begin{equation}   \label{e:relation between the gr's}
\gr_\bcdot L\iso\Fr^*\gr^\bcdot L;
\end{equation}
it is assumed that $\gr^\bcdot L$  is a finitely generated projective $R$-module and
\[
\Fil^iL=L \mbox{ for }i\ll 0,\quad \Fil^iL=0 \mbox{ for }i\gg 0,\quad \Fil_iL=0 \mbox{ for }i\ll 0,\quad \Fil_iL=L \mbox{ for }i\gg 0.
\]
Let $M\in\DISP_1(R)$ correspond to an $F$-zip $L$. Then one checks\footnote{Keep in mind the twist $\{ 1\}$ in \S\ref{sss:Zink complex}.} that the Zink complex $C_M^{\bcdot}$ is as follows. First, 
$C_M^0=L^{\hat\sharp}$, where $L^{\hat\sharp}:=L\otimes_R (\hat W^{(F)}_R)$. So by \S\ref{sss:Ker d could be nonzero}(ii), we can think of $C_M^{-1}$ as a subgroup of 
$L^{\hat\sharp}\times_{\Spec R}L^{\hat\sharp}$, and then $d:C_M^{-1}\to C_M^0=L^\hatsharp$ is just the difference between  the two projections $C_M^{-1}\to L^\hatsharp$.
 Second, this subgroup
is the fiber product of $(\Fil^0L)^{\hat\sharp}$ and $(\Fil_0L)^{\hat\sharp}$ over $(\gr^0L)^{\hat\sharp}$, where the map
$(\Fil_0L)^{\hat\sharp}\to (\gr^0L)^{\hat\sharp}$ is the composition
\begin{equation}  \label{e:the composition}
(\Fil_0L)^{\hat\sharp}\epi (\gr_0L)^{\hat\sharp}\iso \Fr^* (\gr^0L)^{\hat\sharp}\overset{V}\longrightarrow (\gr^0L)^{\hat\sharp}.
\end{equation}  
Note that $C_M^{-1}\supset (\Fil^1L)^{\hat\sharp}\times_{\Spec R}(\Fil_{-1}L)^{\hat\sharp}$, so 
\begin{equation}   \label{e:lower bound}
H^{-1}(C_M^{\bcdot})\supset (\Fil^1L)^{\hat\sharp}\cap (\Fil_{-1}L)^{\hat\sharp}.
\end{equation}
In particular, $H^{-1}(C_M^{\bcdot})$ can be nonzero.

Let $\tilde C_M^{\bcdot}$ be the quotient of $C_M^{\bcdot}$ by the acyclic subcomplex
$0\to (\Fil^1L)^\hatsharp\overset{\id}\longrightarrow (\Fil^1L)^\hatsharp\to 0$
(so the map $C_M^{\bcdot}\to\tilde C_M^{\bcdot}$ is a quasi-isomorphism). To describe $\tilde C_M^{\bcdot}$ more explicitly,
note that $\tilde C_M^{-1}\subset (\gr^0L)^\hatsharp\times (\Fil_0L)^\hatsharp$ is just the graph of the map \eqref{e:the composition}, so $\tilde C_M^{-1}$ identifies with 
$(\Fil_0L)^\hatsharp$ via the projection 
$(\gr^0L)^\hatsharp\times (\Fil_0L)^\hatsharp\to (\Fil_0L)^\hatsharp$. After this identification, $\tilde C_M^{\bcdot}$ becomes the complex
\begin{equation} \label{e:the simplified complex}
0\to (\Fil_0L)^\hatsharp\overset{\gamma-\delta}\longrightarrow (L/\Fil^1L)^\hatsharp\to 0 ,
\end{equation} 
where $\gamma$ is the composite map
\begin{equation}  \label{e:the longer composition}
(\Fil_0L)^{\hat\sharp}\epi (\gr_0L)^{\hat\sharp}\iso \Fr^* (\gr^0L)^{\hat\sharp}\overset{V}\longrightarrow (\gr^0L)^{\hat\sharp}\mono (L/\Fil^1L)^\hatsharp
\end{equation} 
and $\delta$ is the composite map $(\Fil_0L)^\hatsharp\mono L^\hatsharp\epi (L/\Fil^1L)^\hatsharp$. 

Note that if $M\in\DISP_n^+(R)$ then the first and last arrow of \eqref{e:the longer composition} are isomorphisms, so \eqref{e:the simplified complex} is isomorphic to the complex
\[
0\to\Fr^*(\gr^0L)^\hatsharp\overset{V-A^\hatsharp}\longrightarrow (\gr^0L)^\hatsharp\to 0,
\]
where $A:\Fr^*\gr^0L\to\gr^0L$ is the composition of the following linear maps:
\[
\Fr^*\gr^0L\iso\gr_0L\mono L\epi\gr^0L.
\]

\appendix

\section{Recollections on $\hat W$}   \label{s:hat W}
\subsection{The group ind-schemes $\hat W, \hat W^{\bbig}$}  \label{ss:hat W}
For a ring $R$, let $\hat W(R)$ be the set of all $x\in W(R)$ such that all components of the Witt vector $x$ are nilpotent and almost all of them are zero.
Then $\hat W(R)$ is an ideal in $W(R)$ preserved by $F$ and $V$. Quite similarly, one defines an ideal $\hat W^{\bbig}(R)$ in the ring $W^{\bbig}(R)$ of big Witt vectors, which is preserved by $F_n$ and $V_n$ for all $n\in\BN$.

The functors $\hat W$ and $\hat W^{\bbig}$ are group ind-schemes over $\BZ$. It is known that $\hat W^{\bbig}$ is Cartier dual to $W^{\bbig}$, and $\hat W_{\BZ_{(p)}}$ is dual to $W_{\BZ_{(p)}}$ (here $\BZ_{(p)}$ is the localization of $\BZ$ at $p$, and $\hat W_{\BZ_{(p)}}$ is the base change of $\hat W$). Details are explained below.

\subsection{General remarks on Cartier duality} \label{ss:generalities on Cartier duality}
\subsubsection{} \label{sss:generalities on Cartier duality}
The theory of Cartier duality for commutative finite locally free group schemes over an arbitrary ring $R$ is standard.
There is a quite similar theory of Cartier duality between commutative affine group $R$-schemes whose coordinate ring is a projective $R$-module and commutative group ind-schemes over $R$ whose underlying ind-scheme has the form $\Spf A^*$, where $A$ is a cocommutative $R$-coalgebra which is a projective $R$-module.
In particular, if $G$ is a group (ind-)scheme of this class then $(G^*)^*=G$. Note that the schemes $W, W^{\bbig}$ belong to the above class; the same is true for the ind-schemes $\hat W,\hat W^{\bbig}$.

\subsubsection{}
Here is a piece of good news (which is irrelevant for our paper): the class of ind-schemes from \S\ref{sss:generalities on Cartier duality} has a better description under a countability assumption, see \cite{AM}. The approach of \cite{AM} is based on the notion of \emph{Mittag-Leffler module} and the fact that a countably generated flat module is Mittag-Leffler 
if and only if it is projective. The notion and the fact are due to M~Raynaud and L.~Gruson; precise references are given in ~\cite{AM}.

\subsection{Duality between $W^{\bbig}$ and $\hat W^{\bbig}$}  
\subsubsection{The canonical character of $\hat W^{\bbig}$}   \label{sss:2lambda}
One has 
\[
W^{\bbig}(R)=\Ker (R[[t]]^\times\epi R^\times ), \quad \hat W^{\bbig}(R)=\Ker (R[t]^\times\epi R^\times ),
\]
where the maps $R[[t]]^\times\epi R^\times$ and $R[t]^\times\epi R^\times$ are given by evaluation at $t=0$. Define $\lambda :\hat W^{\bbig} (R)\to R^\times$ to be the map that takes 
$f\in \Ker (R[t]^\times\epi R^\times )$ to $f(1)$.

One has
\begin{equation} \label{e:lambda preserved by V_n}
\lambda\circ V_n=\lambda \quad \mbox{ for all } n\in\BN.
\end{equation}

\subsubsection{The pairing}   \label{sss:2pairing}
We have a pairing
\begin{equation} \label{e:big pairing}
W^{\bbig}(R)\times \hat W^{\bbig}(R)\to R^\times, \quad (x,y)\mapsto\langle x,y\rangle:=\lambda (xy).
\end{equation}
Formula \eqref{e:lambda preserved by V_n} implies that
\begin{equation}    \label{e:F_n dual to V_n}
\langle F_n (x),y\rangle =\langle x,V_n(y)\rangle , \quad \langle V_n(x),y\rangle =\langle x,F_n(y)\rangle .
\end{equation}
One can view \eqref{e:big pairing} as a paring
\begin{equation} \label{e:2 big pairing}
W^{\bbig}\times \hat W^{\bbig}\to \BG_m\, .
\end{equation}

\subsubsection{Nondegeneracy of the pairing}  \label{sss:Cartier's theorem}
P.~Cartier proved that \eqref{e:2 big pairing} induces an isomorphism 
\begin{equation} \label{e:Cartier's theorem}
W^{\bbig}\iso\HHom (\hat W^{\bbig},\BG_m),
\end{equation} 
see \cite[Thm.~2]{Ca}. A detailed exposition is given in \cite[\S 37.5]{H}; according to  \cite[\S E.6.2]{H}, it is based on some lecture notes of Cartier.
Key idea: $\hat W^{\bbig}$ identifies with the inductive limit of $\Sym^n\hat\BA^1$, where
the transition map $\Sym^n\hat\BA^1\to\Sym^{n+1}\hat\BA^1$ comes from $0\in\hat\BA^1$.

\subsubsection{Relation to Contou-Carr\`ere's symbol}  \label{sss:Contou-Carrere}
The pairing \eqref{e:2 big pairing} is closely related to Contou-Carr\`ere's tame symbol, see \cite{Co}, \cite[\S 2.9]{De}, and
\cite[Prop.~3.3(ii)(c)]{BBE}.

\subsection{Duality between $W_{\BZ_{(p)}}$ and $\hat W_{\BZ_{(p)}}$}  \label{ss:hat W dual to W}
\subsubsection{The pairing}   \label{sss:pairing}
Let $R$ be a $\BZ_{(p)}$-algebra. Then one has a canonical embedding $$W(R)\mono W^{\bbig}(R).$$
So \eqref{e:big pairing} induces a pairing
\begin{equation} \label{e:p-typical pairing}
W(R)\times \hat W(R)\to R^\times, \quad (x,y)\mapsto\langle x,y\rangle ,
\end{equation}
which can be viewed as a paring
\begin{equation} \label{e:2 p-typical pairing}
W_{\BZ_{(p)}}\times \hat W_{\BZ_{(p)}}\to (\BG_m)_{\BZ_{(p)}}.
\end{equation}

By \eqref{e:F_n dual to V_n}, we have
\begin{equation}    \label{e:F dual to V}
\langle Fx,y\rangle =\langle x,Vy\rangle , \quad \langle Vx,y\rangle =\langle x,Fy\rangle .
\end{equation}

\subsubsection{Nondegeneracy of the pairing}
It is known that the pairing \eqref{e:2 p-typical pairing} induces an isomorphism 
\begin{equation}
W_{\BZ_{(p)}}\iso\HHom (\hat W_{\BZ_{(p)}},(\BG_m)_{\BZ_{(p)}}).
\end{equation}
This follows from Cartier's theorem mentioned in \S\ref{sss:Cartier's theorem}: indeed, $\hat W_{\BZ_{(p)}}$ identifies with
$\hat W^{\bbig}_{\BZ_{(p)}}/\sum\limits_{(n,p)=1} V_n(\hat W^{\bbig}_{\BZ_{(p)}})$, so \eqref{e:Cartier's theorem} implies that
$\HHom (\hat W_{\BZ_{(p)}},(\BG_m)_{\BZ_{(p)}})$ identifies with
\[
\bigcap_{(n,p)=1} \Ker (F_n:  W^{\bbig}_{\BZ_{(p)}}\to W^{\bbig}_{\BZ_{(p)}})=W_{\BZ_{(p)}}.
\]

Let us note that in this article we only need nondegeneracy of the pairing $$W_{\BF_p}\times \hat W_{\BF_p}\to (\BG_m)_{\BF_p}\, ,$$ which is proved in \cite[Ch.~V, \S 4.5]{DG}. 
A related fact is proved in  \S 4 of Chapter~III of \cite{Dem}.

\section{Explicit presentations of the stacks $\sDisp_n^{d,d'}$, $\sDisp_n^{d,d',\weak}$, $\sDisp_n^{d,d',\strong}$, and $\Disp_n^{d,d'}$ }  \label{s:Explicit presentations}
\subsubsection{The goal}

Given integers  $d$ and $d'$ such that $0\le d'\le d$ and an integer $n\ge 0$, we defined in \S\ref{ss:Algebraicity of sDisp_n} 
the stacks $\sDisp_n^{d,d'}$, $\sDisp_n^{d,d',\weak}$, $\sDisp_n^{d,d',\strong}$, and $\Disp_n^{d,d'}$.
We are going to describe each of them as a quotient of an explicit scheme by an explicit group action. 

Let $R$ be an $\BF_p$-algebra. Recall that $\sDisp_n^{d,d'}(R)$ is the full subgroupoid of the underlying groupoid of $\sDisp_n(R)$ whose objects are quadruples $(P,Q,F,F_1)\in\sDisp_n(R)$ such that
$\rank P=d$ and $\rank (P/Q)=d'$. The groupoids $\sDisp_n^{d,d',\weak}(R)$, $\sDisp_n^{d,d',\strong}(R)$, and $\Disp_n^{d,d'}(R)$ are defined similarly but with the following changes:

(i) in the case of  $\Disp_n^{d,d'}(R)$ and $\sDisp_n^{d,d',\strong}(R)$ replace $P/Q$ by $\Coker (Q\to P)$;

(ii) in the case of $\sDisp_n^{d,d',\weak}(R)$ the condition for $P$ is that $\rank (P/I_{n,R}P)=d$.

Note that $\sDisp_1^{d,d'}\ne\emptyset$ only if $d'=d$.

\subsubsection{Rigidifications}
Let $T_0:=W_n(R)^{d'}$, $L_0:=W_n(R)^{d-d'}$.

By a rigidification of an object of $(P,Q,F,F_1)\in\sDisp_n^{d,d}(R)$ we mean an isomorphism $(P,Q)\iso (T_0\oplus  L_0, I_{n,R}\cdot T_0\oplus L_0)$.

By a rigidification of an object of $(M,\cQ,F,F_1)\in\sDisp_n^{d,d',\weak}(R)$ we mean an isomorphism $(M,\cQ)\iso (T_0\oplus \bar L_0, I_{n,R}\cdot T_0\oplus \bar L_0)$, where
$\bar L_0:=L_0/J_{n,R}\cdot L_0$.

By a rigidification of an object of $\Disp_n^{d,d'}(R)$ or $\sDisp_n^{d,d',\strong}(R)$ we mean a normal decomposition (in the sense of \cite[\S 3.2]{Lau13}) plus an isomorphism between the corresponding pair $(T,L)$ and the pair $(T_0,L_0)$.

Let $\sDisp_{n,\rig}^{d,d'}(R)$ be the set of isomorphism classes of pairs consisting of an object of $\sDisp_n^{d,d'}(R)$ and a rigidification of it.
Similarly, define functors $\sDisp_{n,\rig}^{d,d',\weak}$, $\sDisp_{n,\rig}^{d,d',\strong}$, and $\Disp_{n,\rig}^{d,d'}$. Each of the four functors can be described in terms of matrices (see \S\ref{sss:Rigidfications in terms of matrices} below).
This description shows that each of the functors is representable by an affine scheme.

\subsubsection{Rigidfications in terms of matrices}   \label{sss:Rigidfications in terms of matrices}
One has
\[
\Disp_{n,\rig}^{d,d'}(R)=GL(d, W_n(R)), \quad \sDisp_{n,\rig}^{d,d',\strong}(R)=\Mat (d, W_n(R)).
\]
The first equality is explained in \cite[\S 1.2]{LZ}; see also \cite[\S 2.3]{BP}, where only the case $n=\infty$ is considered.
The second equality is quite similar to the first one.

$\sDisp_{n,\rig}^{d,d'}(R)$ is the set of block matrices of the following shape:
\[
\begin{pmatrix}
W_n& W_{n-1}\\ W_n& W_{n-1}
\end{pmatrix}
\]
By this we mean that $\sDisp_{n,\rig}^{d,d'}(R)$ is the set of block matrices
\begin{equation}  \label{e:the matrix x}
\begin{pmatrix}
x_{11}& x_{12}\\ x_{21}& x_{22}
\end{pmatrix}
\end{equation}
where $x_{11}$ (resp.~$x_{22}$) is a square matrix of size $d'$ (resp.~$d-d'$) and the entries of the matrix $x_{ij}$ are in $W_n(R)$ if $j=1$ and in $W_{n-1}(R)$ if $j=2$.
The first $d$ columns of \eqref{e:the matrix x} are the images of the basis vectors of $T_0$ under $F:P\to P$; the other columns are the images of the basis vectors of $L_0$ under $F_1:Q\to P/J_{n,R}\cdot P$.

Similarly, $\sDisp_{n,\rig}^{d,d',\weak}(R)$ is the set of block matrices of the following shape:
\[
\begin{pmatrix}
W_n& W_{n-1}\\ W_{n-1}& W_{n-1}
\end{pmatrix}
\]

\subsubsection{$\sDisp_n^{d,d'}$ as a quotient stack}  \label{sss:sDisp_n as a quotient}
Let $H_n^{d,d'}(R):=\Aut (P_0,Q_0)$, where $P_0:=T_0\oplus  L_0$, $Q_0:=I_{n,R}\cdot T_0\oplus L_0$.
Then $H_n^{d,d'}$ is a group scheme acting on the scheme $\sDisp_{n,\rig}^{d,d'}$, and
\[
\sDisp_n^{d,d'}=\sDisp_{n,\rig}^{d,d'}/H_n^{d,d'}.
\]
In the language of \S\ref{sss:Rigidfications in terms of matrices}, $H_n^{d,d'}$ is the group of invertible block matrices of the following shape:
\begin{equation}  \label{e: the shape for H_n}
\begin{pmatrix}
W_n& I_n\\ W_n& W_n
\end{pmatrix}
\end{equation}
(here $I_n$ denotes the functor $R\mapsto I_{n,R}$). An element 
\[
h=
\begin{pmatrix}
h_{11}& h_{12}\\ h_{21}& h_{22}
\end{pmatrix}
\in H_n^{d,d'}(R)
\]
acts on $\sDisp_{n,\rig}^{d,d'}(R)$ by 
\begin{equation}  \label{e:action of h}
x\mapsto hx\Phi (h)^{-1}, \quad \mbox{ where } \Phi (h):=
\begin{pmatrix}
F(h_{11})& V^{-1}(h_{12})\\ pF(h_{21})& F(h_{22})
\end{pmatrix}.
\end{equation}
Informally,
\[
\Phi (h)=
\begin{pmatrix}
p^{-1}& 0\\ 0& 1
\end{pmatrix}
F(h)
\begin{pmatrix}
p& 0\\ 0& 1
\end{pmatrix}.
\]

Let us note that $\Phi$ is a homomorphism from the ring of block-matrices of shape \eqref{e: the shape for H_n} to the ring of block-matrices of shape
\[
\begin{pmatrix}
W_n& W_{n-1}\\ I_n& W_n
\end{pmatrix}
\]
So $\Phi (h)^{-1}=\Phi (h^{-1})$.

\subsubsection{$\sDisp_n^{d,d',\weak }$ as a quotient stack}
This is parallel to \S\ref{sss:sDisp_n as a quotient}, but instead of \eqref{e: the shape for H_n} one has to consider invertible matrices of the following shape:
\[
\begin{pmatrix}
W_n& I_n\\ W_{n-1}& W_{n-1}
\end{pmatrix}
\]

\subsubsection{$\Disp_n^{d,d'}$ and $\sDisp_n^{d,d',\strong}$ as quotient stacks}  \label{sss:Disp_n & sDisp_n as quotients}
One has 
\[
\Disp_n^{d,d'}=\Disp_{n,\rig}^{d,d'}/\BP_n^{d,d'},\quad \sDisp_n^{d,d',\strong}=\sDisp_{n,\rig}^{d,d',\strong}/\BP_n^{d,d'},
\]
where $\BP_n^{d,d'}$ is a certain group scheme acting on $\sDisp_{n,\rig}^{d,d',\strong}$ and preserving $\Disp_{n,\rig}^{d,d'}$ (which is an open subscheme of $\sDisp_{n,\rig}^{d,d',\strong}$). 
The group scheme $\BP_n^{d,d'}$ was introduced by O.~B\"ultel and G.~Pappas\footnote{See \cite[\S 2.3]{BP}. Let us note that the authors of \cite{BP} refer (after their Definition 1.0.1) to other works in which the subgroup was introduced and used; one of them is a 2008 e-print by B\"ultel.} in the case $n=\infty$ and then used (for arbitrary $n$) in \cite[\S 1.2]{LZ} to describe $\Disp_n^{d,d'}$. The definition of $\BP_n^{d,d'}$ and its action on $\sDisp_{n,\rig}^{d,d',\strong}$ is recalled below.

Let $GL(d,W_n)$ be the group scheme over $\BF_p$ representing the functor $R\mapsto GL(d,W_n(R))$. The group $\BP_n^{d,d'}$ is a certain subgroup of $GL(d,W_n)\times GL(d,W_n)$, and it acts on $\sDisp_{n,\rig}^{d,d',\strong}$ by two-sided translations; more precisely, a pair
$$(g,h)\in \BP_n^{d,d'}(R)\subset GL(d,W_n(R))\times GL(d,W_n(R))$$
acts by
\[
x\mapsto hxg^{-1}, \quad \mbox{ where }x\in \Mat (d,W_n(R))=\sDisp_{n,\rig}^{d,d',\strong}(R).
\]
The subgroup $\BP_n^{d,d'}(R)\subset GL(d,W_n(R)))\times GL(d,W_n(R))$ consists of pairs $(g,h)$ such that
\begin{equation} \label{e:g via h}
g_{ij}=p^{i-j}F(h_{ij}) \mbox{ for } 1\le j\le i\le 2,
\end{equation}
\begin{equation}    \label{e:h via g}
h_{12}=V(g_{12}),
\end{equation}
where $g_{ij},h_{ij}$ are the blocks of the block matrices $g,h$ (as before, the blocks $g_{11},h_{11}$ have size~$d'$).
Note that \eqref{e:h via g} is an ``avatar'' of \eqref{e:g via h} because $FV=VF=p$. 

Informally, $\BP_n^{d,d'}(R)$ is the graph of a \emph{multivalued} map $\Phi:H_n^{d,d'}(R)\to GL(d,W_n(R))$, where $H_n^{d,d'}\subset GL(d,W_n)$ is the group of invertible block matrices of the shape
\[
\begin{pmatrix}
W_n& I_n\\ W_n& W_n
\end{pmatrix}
\]
and $\Phi $ is essentially as in formula \eqref{e:action of h} (except that $\Phi$ is multi-valued now).

\section{$\Disp_n^{G,\mu}$ and $\Lau_n^{G,\mu}$}   \label{s:Disp_n and Lau_n}
\subsection{The setting and the plan}
\subsubsection{The setting}   \label{sss:my setting}
Let $G$ be a smooth affine group scheme over $\BZ/p^n\BZ$ equipped with a homomorphism $\mu :\BG_m\to G$.

\subsubsection{Key example}  \label{sss:the key example}
Let $d\ge d'\ge 0$. Let $G=GL(d)$ and 
\begin{equation}
\mu  (\lambda )=\diag (\lambda ,\ldots ,\lambda ,1,\ldots , 1),
\end{equation}
where $\lambda $ appears $d'$ times and $1$ appears $d-d'$ times. 

\subsubsection{Plan}
In the setting of \S\ref{sss:my setting}, there is a stack $\Disp_n^{G,\mu}$ over $\BF_p$ such that in the situation of \S\ref{sss:the key example} one has $\Disp_n^{G,\mu}=\Disp_n^{d,d'}$.
We will recall the definition of $\Disp_n^{G,\mu}$ in~\S\ref{ss:Disp_n,G,mu}. Assuming that $\mu$ is \emph{1-bounded} in the sense of \S\ref{sss:again 1-boundedness}, we will define in \S\ref{ss:2from G-displays to semidisplays}-\ref{ss:Lau_n,G,mu} a group scheme $\Lau_n^{G,\mu}$ over $\Disp_n^{G,\mu}$ such that in the situation of \S\ref{sss:the key example} one has $\Lau_n^{G,\mu}=\Lau_n^{d,d'}$. If $\mu$ is 1-bounded then according to \cite{GMM}, one has an algebraic stack $\BT_n^{G,\mu}\otimes\BF_p$. We conjecture that it is a gerbe over $\Disp_n^{G,\mu}$ banded by $\Lau_n^{G,\mu}$ (see \S\ref{ss:Lau_n,G,mu}).
In \S\ref{ss:the Cartier dual explicitly} we describe the Cartier dual of $\Lau_n^{G,\mu}$ very explicitly.

\subsubsection{On the setting of \S\ref{sss:my setting}}
The setting of \cite{BP,Lau21,GMM} is more general than that of  \S\ref{sss:my setting}: $\mu$ is there a cocharacter of $G\otimes O/p^nO$, where $O$ is the ring of integers of a finite unramified extension of $\BQ_p$.
This generalization is important.

\subsection{The stack $\Disp_n^{G,\mu}$}  \label{ss:Disp_n,G,mu}
Let us recall the definition of $\Disp_n^{G,\mu}$ from \cite{Lau21} (in the earlier article \cite{BP} this was a description rather than the definition). It is quite parallel to the description of $\Disp_n^{d,d'}$ given in \S\ref{sss:Disp_n & sDisp_n as quotients} of Appendix~\ref{s:Explicit presentations}.

\subsubsection{Outline}
Let $G(W_n)$ be the affine group scheme over $\BF_p$ representing the functor $R\mapsto G(W_n(R))$ on the category of $\BF_p$-algebras.
The group scheme  $G(W_n)\times G(W_n)$ acts on the scheme $G(W_n)$ as follows:
a pair $(g,h)$ in $G(W_n(R))\times G(W_n(R))$ acts by
\begin{equation}   \label{e:2-sided translations}
U\mapsto hUg^{-1}, \quad \mbox{ where }U\in G(W_n(R)).
\end{equation}
The stack $\Disp_n^{G,\mu}$ is defined to be the quotient of the scheme $G(W_n)$ by the action of a certain subgroup\footnote{In \cite[\S 5.1]{Lau21} it is called the \emph{display group.}} $\BP_n^{G,\mu}\subset G(W_n)\times G(W_n)$.
The subgroup is defined below.

\subsubsection{Definition of $\BP_n^{G,\mu}$}  \label{sss:2Definition of BP_n}
$\BG_m$ acts on $G$ and $H^0(G,\cO_G)$: namely, $\lambda \in\BG_m$ acts on $G$ by $g\mapsto \mu (\lambda )g\mu (\lambda )^{-1}$, and it acts on $H^0(G,\cO_G)$
by taking $\varphi\in H^0(G,\cO_G)$ to the function $$g\mapsto \varphi ( \mu (\lambda )g\mu (\lambda )^{-1}).$$
The action of $\BG_m$ on $H^0(G,\cO_G)$ induces a grading
\begin{equation}  \label{e:grading of Fun(G)}
H^0(G,\cO_G)=\bigoplus\limits_{k\in\BZ} H^0(G,\cO_G)_k\, .
\end{equation}

For a $\BZ$-graded ring $A$, let $G(A)^{\BG_m}\subset G(A)$ denote the subgroup of \emph{graded} ring ho\-mo\-morphisms $H^0(G,\cO_G)\to A$; this is indeed a subgroup because $G(A)^{\BG_m}=\Mor^{\BG_m}(\Spec A,G)$.
Finally,
\[
\BP_n^{G,\mu} (R):=G(W_n(R)^\oplus)^{\BG_m}, 
\]
where $W_n(R)^\oplus$ is the $n$-truncated Witt frame, see \S\ref{sss:truncated Witt frame}. Since $W_n(R)^\oplus$ is a graded subring of $W_n(R)[u,u^{-1}]\times W_n(R)[t,t^{-1}]$, we see that
$\BP_n^{G,\mu} (R)$ is a subgroup of the group
\[
G(W_n(R)[u,u^{-1}]\times W_n(R)[t,t^{-1}])^{\BG_m}=G(W_n(R)\times W_n(R))=G(W_n(R))\times G(W_n(R))
\]
(we have used the homomorphisms $W_n(R)[u,u^{-1}]\to W_n(R)$ and $W_n(R)[t,t^{-1}]\to W_n(R)$ given by evaluation at $u=1$ and $t=1$).

The following lemma describes $\BP_n^{G,\mu}$ in matrix terms.

\begin{lem}
Let $\rho :G\to GL(r)$ be a homomorphism such that 
$$\rho (\mu (\lambda ))=\diag (\lambda ^{m_1},\ldots ,\lambda ^{m_r}).$$
Let $\rho_{ij}\in H^0(G,\cO_G)$ be its matrix elements.

(i) If $(g,h)\in \BP_n^{G,\mu} (R)\subset G(W_n(R))\times G(W_n(R))$ then
\begin{equation}  \label{e:again g in terms of h}
\rho_{ij}(g)=p^{m_j-m_i}F(\rho_{ij}(h)) \quad\mbox{ if }\; m_i\le m_j ,
\end{equation}
\begin{equation}    \label{e:again h in terms of g}
\rho_{ij}(h)=p^{m_i-m_j-1}V(\rho_{ij}(g)) \quad\mbox{ if }\; m_i> m_j .
\end{equation}

(ii) If $\rho :G\to GL(r)$ is a closed immersion and $g,h\in G(W_n(R))$ satisfy \eqref{e:again g in terms of h}-\eqref{e:again h in terms of g} then  $(g,h)\in \BP_n^{G,\mu} (R)$.
\end{lem}

\begin{proof}
(i) Note that in terms of the grading \eqref{e:grading of Fun(G)}, we have $\rho_{ij}\in H^0(G,\cO_G)_{m_i-m_j}$. 

Let $(g,h)\in \BP_n^{G,\mu} (R)$. 
According to the definition of $\BP_n^{G,\mu}$ (see \S\ref{sss:2Definition of BP_n}), this implies that the pair
$(\rho_{ij}(g),\rho_{ij}(h))\in W_n(R)\times W_n(R)$ belongs to $f(S_{m_i-m_j})$, where
$$S:=W_n(R)^\oplus\subset W_n(R)[u,u^{-1}]\times W_n(R)[t,t^{-1}]$$
and the map $f:W_n(R)[u,u^{-1}]\times W_n(R)[t,t^{-1}]\to W_n(R)\times W_n(R)$ is given by evaluation at $u=t=1$. 
So formulas \eqref{e:again g in terms of h}-\eqref{e:again h in terms of g} follow from \eqref{e:a' via a}-\eqref{e:a via a'}.

(ii) If $\rho :G\to GL(r)$ is a closed immersion then the ring $H^0(G,\cO_G)$ is generated by the functions $\rho_{ij}$ and $(\det\rho )^{-1}$, so the above argument can be reversed.
\end{proof}

\begin{cor}
In the situation of \S\ref{sss:the key example} one has $\BP_n^{G,\mu}=\BP_n^{d,d'}$, $\Lau_n^{G,\mu}=\Lau_n^{d,d'}$. 
\end{cor}

\begin{proof}
Compare equations \eqref{e:again g in terms of h}-\eqref{e:again h in terms of g} with \eqref{e:g via h}-\eqref{e:h via g}.
\end{proof}

\subsection{$\BP_n^{G,\mu}$ and $\Disp_n^{G,\mu}$ in terms of $G$-torsors} \label{ss:in terms of G-torsors}
For an $\BF_p$-algebra $R$, let $$X_R:=\Spec W_n(R)^\oplus.$$

\subsubsection{The $G$-torsor $\cP_R^\mu$}
The group $\BG_m$ acts on $X_R$. On the other hand, $\BG_m$ acts on $G$ by left translations via $\mu :\BG_m\to G$. Thus we get an action of $\BG_m$ on $X_R\times G$. It commutes with the action of $G$ on $X_R\times G$ by right translations. So
$X_R\times G$ is a $\BG_m$-equivariant $G$-torsor over $X_R$.  The corresponding $G$-torsor over the stack $X_R/\BG_m$ is denoted by $\cP_R^\mu$. 
(In other words, $\cP_R^\mu$ is the $G$-torsor over $X_R/\BG_m$ induced via $\mu$ by the $\BG_m$-torsor $X_R\to X_R/\BG_m$.)

One checks that
\begin{equation}   \label{e:the display group as Aut}
\Aut \cP_R^\mu=\BP_n^{G,\mu}(R).
\end{equation}

\subsubsection{$G$-torsors of type $\mu$}  \label{sss:G-torsors of type mu}
We say that a $G$-torsor over $X_R/\BG_m$ has \emph{type $\mu$} if it becomes isomorphic to $\cP_R^\mu$ after etale localization with respect to $R$.
By \eqref{e:the display group as Aut}, \emph{a $G$-torsor of type $\mu$ over $X_R/\BG_m$ is the same as an etale $\BP_n^{G,\mu}$-torsor over $\Spec R$.} In fact, ``etale'' can be replaced by ``fpqc'' because $\BP_n^{G,\mu}$ is smooth by \cite[Lemma~2.3.8]{Lau21}.

\subsubsection{Remarks}    \label{sss:on G-torsors of type mu}
(i)  If $\mu'$ and $\mu$ are conjugate then $\cP_R^{\mu'}\simeq\cP_R^\mu$. 

(ii) Several conditions equivalent to having type $\mu$ are formulated in \cite[Prop.~5.5.2]{GMM}.

\subsubsection{$\Disp_n^{G,\mu}$ in terms of $G$-torsors}   \label{sss:in terms of G-torsors}
By \S\ref{sss:G-torsors of type mu}, one can reformulate \cite[Lemma~5.3.8]{Lau21} as follows: $\Disp_n^{G,\mu}(R)$ identifies with the groupoid of $G$-torsors $\cF$ of type $\mu$ over $X_R/\BG_m$ equipped with an isomorphism $f:\cF^\sigma\iso\cF^\tau$. 
Here $\cF^\sigma ,\cF^\tau$ are the pullbacks of $\cF$ via the morphisms $\Spec W_n(R)\to X_R/\BG_m$ corresponding to the homomorphisms
$$\sigma: W_n(R)^\oplus\to W_n(R), \quad \tau: W_n(R)^\oplus\to W_n(R)$$
from \S\ref{sss:sigma,tau}. In this language, the canonical map $G(W_n(R))\to\Disp_n^{G,\mu}(R)$ takes an element $U\in G(W_n(R))$  to $(\cF,f)$, where $\cF=\cP^\mu_R$ and $f:\cF^\sigma\iso\cF^\tau$ is given by $U$.

\subsection{A morphism $\Disp_n^{G,\mu}\to\sDisp_n$}    \label{ss:2from G-displays to semidisplays}
\subsubsection{The 1-boundedness condition}   \label{sss:again 1-boundedness}
Let $\fg:=\Lie (G)$. Composing the adjoint representation of $G$ with $\mu :\BG_m\to G$, one gets an action of $\BG_m$ on $\fg$ or equivalently, a grading 
\begin{equation}  \label{e:grading on Lie(G)}
\fg=\bigoplus\limits_{i\in\BZ} \fg_i
\end{equation}
 compatible with the Lie bracket.
From now on, we assume that the $\BG_m$-action is \emph{1-bounded}, by which we mean that
\begin{equation}   \label{e:again 1-boundedness}
\fg_i=0 \;\mbox{ for all } i>1.
\end{equation}
Under this assumption, we will define in \S\ref{sss:from G-displays to semidisplays} a morphism $\Disp_n^{G,\mu}\to\sDisp_n$. In \S\ref{ss:Disp_n,G,mu to sDisp_n explicitly} we will describe it more explicitly.

\subsubsection{Twists of $\fg$}  \label{sss:twists of fg}
Let $\cF$ be a $G$-torsor of type $\mu$ over $X_R/\BG_m$. Let $\fg_\cF$ be the $\cF$-twist of the $G$-module $\fg$; this is a vector bundle on $X_R/\BG_m$. Pulling it back to $X_R$, we get a finitely generated projective graded $S$-module, where $S:=W_n(R)^\oplus$; by abuse of notation, we still denote it by $\fg_\cF$.
We claim that
\begin{equation}   \label{e:effectivity of the shifted twist}
\fg_\cF\{ -1\}\in\Vect_n^+(R),
\end{equation}
where $\{ -1\}$ denotes the shift of grading (see \S\ref{sss:Vect_n(R)}) and $\Vect_n^+(R)$ is as in \S\ref{sss:the effective subcategory}. It suffices to check this if $\cF\simeq\cP^\mu_R$. In this case $\fg_\cF\simeq\bigoplus\limits_i\fg_i\otimes S\{ i\}$, where $\{ i\}$ denotes the shift of grading
and the $\fg_i$'s are as in \eqref{e:grading on Lie(G)}. So \eqref{e:effectivity of the shifted twist} follows from the 1-boundedness assump\-tion~\eqref{e:grading on Lie(G)}.

\subsubsection{A morphism $\Disp_n^{G,\mu}\to\sDisp_n$}    \label{sss:from G-displays to semidisplays}
By~\S\ref{sss:in terms of G-torsors}, we can think of an object of $\Disp_n^{G,\mu}(R)$ as a pair $(\cF,f:\cF^\sigma\iso\cF^\tau)$, where $\cF$ is a $G$-torsor of type $\mu$ over $X_R/\BG_m$.
By \S\ref{sss:twists of fg}, we have an object $\fg_\cF\{ -1\}\in\Vect_n^+(R)$. Using $f:\cF^\sigma\iso\cF^\tau$, one upgrades it to an object of the category $\DISP_n^+(R)$ from \S\ref{ss:higher (pre)displays}. Thus we have constructed a functor
\begin{equation}  \label{e:the higher display on Disp_n}
\Disp_n^{G,\mu}(R)\to\DISP_n^+(R).
\end{equation}  
Composing it with the functor $\DISP_n^+(R)\to\sDisp_n(R)$ from \S\ref{ss:preDISP to sDisp}, one gets a functor
\begin{equation}    \label{e:2the semidisplay on Disp_n}
\Disp_n^{G,\mu}(R)\to\sDisp_n(R).
\end{equation}

\subsection{The group scheme $\Lau_n^{G,\mu}$}   \label{ss:Lau_n,G,mu}
\subsubsection{}    \label{sss:Lau_n,G,mu}
Let $\Lau_n^{G,\mu}$ be the commutative $n$-smooth group scheme over $\Disp_n^{G,\mu}$ corresponding to the composite morphism
\begin{equation}
\Disp_n^{G,\mu}\to\sDisp_n\overset{\fZ}\longrightarrow\Sm_n,
\end{equation}
where the first arrow is \eqref{e:2the semidisplay on Disp_n} and the second one is given by Zink's functor (see \S\ref{ss:truncated Zink's functor}).
By~Lemma~\ref{l:comparing with the old Zink complex}, the functor $\Disp_n^{G,\mu}(R)\to \Sm_n (R)$ can also be described as follows\footnote{This description shows that even without assuming $\mu$ to be 1-bounded,
$\Lau_n^{G,\mu}$ is defined as a \emph{complex of group ind-schemes} concentrated in degrees $0,-1$, but 1-boundedness ensures that this complex is a group scheme.}: it takes a pair $(\cF,f:\cF^\sigma\iso\cF^\tau)$ to the $0$-th cohomology of the Zink complex $C_{\fg_\cF\{ -1\}}^\bcdot $ (the latter is defined in \S\ref{sss:Zink complex}).

In the situation of \S\ref{sss:the key example} one has 
\begin{equation}   \label{e:Lau=Lau}
\Lau_n^{G,\mu}=\Lau_n^{d,d'};
\end{equation}
this follows from Theorem~\ref{t:Lau_n explicitly} combined with \S\ref{ss:the canonical semidisplay}.

\subsubsection{}
On the other hand, let $\BT_n^{G,\mu}$ be the stack defined in \cite{GMM}; the definition of $\BT_n^{G,\mu}$ uses the \emph{syntomification} functor $R\mapsto R^{\Syn}$ from \cite{Bh}. By \cite[Theorem~D]{GMM}, if $\mu$ is 1-bounded then the stack $\BT_n^{G,\mu}\otimes\BZ/p^m\BZ$ is algebraic for every $m$; moreover,
in the situation of \S\ref{sss:the key example} one has $\BT_n^{G,\mu}\otimes\BZ/p^m\BZ=\BTst_n^{d,d'}\otimes\BZ/p^m\BZ$. 

One has a canonical morphism $\phi_n:\BT_n^{G,\mu}\otimes\BF_p\to\Disp_n^{G,\mu}$, see \cite[Rem.~9.1.3]{GMM} (it is given by pullback with respect to a certain morphism $X_R/\BG_m\to R^\cN\otimes\BZ/p^n\BZ$, where $X_R$ is as in \S\ref{ss:in terms of G-torsors} and $R^\cN$ is the Nygaard-filtered prismatization\footnote{See  \cite[Def.~5.3.10]{Bh}. By \cite[\S 5.4]{Bh}, if $R$ is smooth over $\BF_p$ one can use \cite[Def.~3.3.13]{Bh}. } of $\Spec R$). In the situation of \S\ref{sss:the key example} this is hopefully the morphism mentioned in Theorem~\ref{t:Lau's theorem}.

\begin{conj}    \label{conjecture}
The morphism $\phi_n:\BT_n^{G,\mu}\otimes\BF_p\to\Disp_n^{G,\mu}$ is a gerbe banded by $\Lau_n^{G,\mu}$.
\end{conj}

\subsubsection{Remarks}     \label{sss:conjecture}
(i) Conjecture~\ref{conjecture} holds for $n=1$ (see \cite[Rem.~9.3.4]{GMM} and the related part of the proof of \cite[Thm.~9.3.2]{GMM}).

(ii) It is expected that in the situation of \S\ref{sss:the key example} the morphism $\phi_n:\BT_n^{G,\mu}\otimes\BF_p\to\Disp_n^{G,\mu}$ from \cite[Rem.~9.1.3]{GMM} is equal to the one mentioned in Theorem~\ref{t:Lau's theorem}. By formula~\eqref{e:Lau=Lau} combined with Lau's Theorem~\ref{t:Lau's theorem}, this would imply that
in the situation of \S\ref{sss:the key example} Conjecture~\ref{conjecture} holds for all $n$. 

(iii) As far as I understand, E.~Lau \cite{Lau25} has proved Conjecture~\ref{conjecture} in general.

\subsection{Explicit description of the morphism $\Disp_n^{G,\mu}\to\sDisp_n$}  \label{ss:Disp_n,G,mu to sDisp_n explicitly}
In \S\ref{sss:G(W_n) to sDisp_n explicitly}-\ref{sss:Disp_n,G,mu to sDisp_n explicitly} we give an explicit description of the morphism $\Disp_n^{G,\mu}\to\sDisp_n$ defined in \S\ref{ss:2from G-displays to semidisplays}.
This description is obtained by unraveling the definitions in a straightforward way, so we just formulate the answer.

\subsubsection{A preliminary remark}  \label{sss:preliminary remark}
The composition $\Disp_n^{G,\mu}\to\DISP_n^+\to\Vect_n^+$ was described in~\S\ref{sss:twists of fg}. Combining this description with \S\ref{sss:in terms of G-torsors}, we see that the composite map
\begin{equation}   \label{e:G(W_n) to Vec_n}
G(W_n)\to\Disp_n^{G,\mu}\to\DISP_n^+\to\Vect_n^+
\end{equation}
is \emph{constant.} More precisely, the map \eqref{e:G(W_n) to Vec_n} takes every $U\in G(W_n(R))$ to the object
\[
M=\bigoplus\limits_i \fg_i\otimes S\{ i-1\} \in\Vect_n^+(R), \quad \mbox{ where } S:=W_n(R)^{\oplus} .
\]
Thus $M_j=\bigoplus\limits_i S_{i+j-1}\otimes\fg_i$. Since $\fg_i=0$ for $i>1$, we get
\[
M_0= W_n(R)\otimes\fg , \quad M_1:= (W_n(R)\otimes\fg_{\le 0})\oplus (I_{n+1,R}\otimes\fg_1),
\]
and the map $t:M_1\to M_0$ comes from the canonical map $I_{n+1,R}\to W_n(R)$.

\subsubsection{The morphism $G(W_n)\to\sDisp_n$ in explicit terms}  \label{sss:G(W_n) to sDisp_n explicitly}
The composite morphism $$G(W_n)\to \Disp_n^{G,\mu}\to\sDisp_n$$ takes
an element $U\in G(W_n(R))$ to the quadruple\footnote{In this quadruple $\Phi_U :P_U\to P_U$ and $\Phi'_U:Q_U\to P_U$ play the role of $F$ and $F_1$ from \S\ref{sss:definition of sDisp_n}} 
$(P_R,Q_R,\Phi_U, \Phi'_U)\in\sDisp_n (R)$, where the $W_n(R)$-modules
$P_R$, $Q_R$,  and the $F$-linear maps $\Phi_U:P_U\to P_U$, $\Phi'_U:Q_U\to P_U/J_{n,R}P_U$ are as follows:

(i) in terms of \S\ref{sss:preliminary remark}, $P_R=M_0$ and $Q_R=\im (M_1\overset{t}\longrightarrow M_0)$; explicitly,
\[
P_R= W_n(R)\otimes\fg, \quad Q_R= (W_n(R)\otimes\fg_{\le 0})\oplus (I_{n,R}\otimes\fg_1)\subset P_R\, .
\]

(ii) $\Phi_U :P_R\to P_R$ and $\Phi'_U :Q_R\to P_R/J_{n,R}\cdot P_R$ are the $F$-linear maps such that
\begin{equation}  \label{e:defining Phi}
\Phi_U (a\otimes x)=p^{1-i}\Ad_U(F(a)\otimes x) \quad \mbox{ for } a\in W_n(R), x\in \fg_i,
\end{equation}
\begin{equation}   \label{e:defining Phi'}
\Phi'_U (a\otimes x)=p^{-i}\Ad_U(\bar F(a)\otimes x) \quad \mbox{ for } a\in W_n(R), x\in \fg_i \;\mbox{ if } i\le 0,
\end{equation}
\begin{equation}   \label{e:2defining Phi'}
\Phi'_U (a\otimes x)=\Ad_U(V^{-1}(a)\otimes x) \quad \mbox{ for } a\in I_{n,R}, x\in \fg_1.
\end{equation}
(here $\bar F(a)\in W_{n-1}(R)$ is the image of $F(a)$).

Thus $\Phi_U=\Ad_U\circ\Phi_1$, $\Phi'_U=\Ad_U\circ\Phi'_1$. Note that \eqref{e:2defining Phi'} is an ``avatar'' of \eqref{e:defining Phi'}.

\subsubsection{The morphism $\Disp_n^{G,\mu}\to\sDisp_n$ in explicit terms}    \label{sss:Disp_n,G,mu to sDisp_n explicitly}
Recall that the stack $\Disp_n^{G,\mu}$ is the quotient of the scheme $G(W_n)$ by the following action of $\BP_n^{G,\mu}$: an element $$(g,h)\in\BP_n^{G,\mu}(R)\subset G(W_n(R))\times G(W_n(R)$$
takes $U\in G(W_n(R))$ to $hUg^{-1}$. So descending the morphism $G(W_n)\to\sDisp_n$ described in \S\ref{sss:G(W_n) to sDisp_n explicitly} to a morphism $\Disp_n^{G,\mu}\to\sDisp_n$ amounts to specifying an isomorpism
\[
(P_R,Q_R,\Phi_U, \Phi'_U)\iso (P_R,Q_R,\Phi_{hUg^{-1}}, \Phi'_{hUg^{-1}}),
\]
where $(g,h)\in\BP_n^{G,\mu}(R)$. This isomorphism is $\Ad_h :P_R\iso P_R$.

\subsection{Explicit description of $(\Lau_n^{G,\mu})^*$}  \label{ss:the Cartier dual explicitly}
Let $(\Lau_n^{G,\mu})^*$ be the Cartier dual of $\Lau_n^{G,\mu}$. We will describe $(\Lau_n^{G,\mu})^*$ as a subgroup of a very simple group scheme $\tilde A$ over $\Disp_n^{G,\mu}$.

\subsubsection{The group scheme $\tilde A$}
Let $A$ be the group scheme over $\BF_p$ whose $R$-points are additive maps $\fg\to W_n(R)$ or equivalently, $W_n(R)$-linear maps $W_n(R)\otimes\fg\to W_n(R)$. We have the coadjoint action of $G(W_n(R))$ on $A(R)$.
So $G(W_n)$ acts on $A$. Precomposing this action with the homomorphism $\BP_n^{G,\mu}\mono G(W_n)\times G(W_n)\overset{\prr_2}\longrightarrow G(W_n)$, we get an action of $\BP_n^{G,\mu}$ on~$A$. On the other hand, we have the $\BP_n^{G,\mu}$-torsor 
$G(W_n)\to\Disp_n^{G,\mu}$. Define $\tilde A$ to be the twist of $A$ by this torsor. Thus $\tilde A$ is a smooth commutative group scheme over $\Disp_n^{G,\mu}$.

\subsubsection{Explicit description of $(\Lau_n^{G,\mu})^*$}
Recall that $\Lau_n^{G,\mu}$ corresponds to the composition $\Disp_n^{G,\mu}\to\sDisp_n\overset{\fZ}\longrightarrow\Sm_n$.
If $\sP =(P,Q,F,F_1)\in\sDisp_n (R)$ then by \S\ref{ss:Cartier dual of Z_P}, $\fZ_\sP^*$ is the subgroup of the group scheme $\Hom_{W_n(R)} (P,W_{n,R})$ defined by equations~\eqref{e:Cond_x}-\eqref{e:Cond_y}.
Combining this with \S\ref{ss:Disp_n,G,mu to sDisp_n explicitly}, we see that $(\Lau_n^{G,\mu})^*$ is a closed subgroup of $\tilde A$. It remains to describe the closed subscheme
\begin{equation} \label{e:Lau* as a subgroup of A}
(\Lau_n^{G,\mu})^*\times_{\Disp_n^{G,\mu}}G(W_n)\subset A\times G(W_n).
\end{equation}
Combining equations \eqref{e:Cond_x}-\eqref{e:Cond_y} with \eqref{e:defining Phi}-\eqref{e:2defining Phi'}, we get the following description of the subscheme \eqref{e:Lau* as a subgroup of A}.

Let $(\eta,U)$ be an $R$-point of $A\times G(W_n)$, so $\eta\in\Hom (\fg, W_n (R))$, $U\in G(W_n(R))$. Then
\emph{$(\eta,U)$ is an $R$-point of $(\Lau_n^{G,\mu})^*\times_{\Disp_n^{G,\mu}}G(W_n)$ if and only if $\eta$ and $U$ satisfy the equations}
\begin{equation} \label{e:the equation for i=1}
F(\eta (x))=\eta (\Ad_U(x)) \quad \mbox{ if } x\in\fg_1,
\end{equation}
\begin{equation} \label{e:the equation for i<1}
\eta (x)=p^{-i} V(\eta (\Ad_U(x))) \quad \mbox{ if } x\in\fg_i \mbox{ and } i\le 0.
\end{equation}

Note that \eqref{e:the equation for i=1} is an ``avatar'' of \eqref{e:the equation for i<1}.

\medskip

In the next remark we describe a way to transform the system of equations \eqref{e:the equation for i=1}-\eqref{e:the equation for i<1} by eliminating some of the unknowns
(although it is not clear whether this is worth doing).

\subsubsection{Remark}  \label{sss:eliminating some unknowns}
Recall that $A:=\Hom (\fg ,W_n)$. Let $A_1:=\Hom (\fg_1 ,W_n)$; this is a direct summand of $A$. Proposition~\ref{p:D(sP) economically} implies that the composite map
\[
(\Lau_n^{G,\mu})^*\times_{\Disp_n^{G,\mu}}G(W_n)\mono A\times G(W_n)\epi A_1\times G(W_n)
\]
is still a closed immersion. It also describes its image by rather explicit equations.

\section{A conjectural description of $\BT_n^{G,\mu}$}  \label{s:conjectural description of BT_n}
Just as in Appendix~\ref{s:Disp_n and Lau_n}, let $G$ be a smooth affine group scheme over $\BZ/p^n\BZ$ equipped with a homomorphism $\mu :\BG_m\to G$.
Let $\BT_n^{G,\mu}$ be the stack defined in \cite{GMM} along the lines of~\cite[Appendix~C]{Dr}.

\subsection{The goal}
Suppose that $\mu$ is 1-bounded. Then by \cite[Theorem~D]{GMM}, the stack $\BT_n^{G,\mu}\otimes\BZ/p^m\BZ$ is algebraic for every $m$, i.e., it can be represented as a quotient of a scheme by an action of a flat groupoid. 
One would like to have an \emph{explicit} presentation of $\BT_n^{G,\mu}\otimes\BZ/p^m\BZ$ as such a quotient. The goal of this Appendix is to formulate a conjecture in this direction, see Conjecture~\ref{conjecture in characteristic p}.

\subsection{Format of the conjecture}   \label{ss:the format}
The formulation of the conjecture uses ideas from the theory of prismatization \cite{Bh,Prismatization} and ``sheared prismatization'' \cite{BMVZ, Sheared}.

\subsubsection{The format}  \label{sss:the format}
We will define a stack of $\BZ/p^n\BZ$-algebras denoted by $\SR_n$. Using $\SR_n$, we will define a group stack $G(\SR_n)$. We will also define another group stack, denoted by $G(\SR_n^\oplus)^{\BG_m}$; 
its definition involves the homomorphism $\mu :\BG_m\to G$. One has a canonical homomorphism $G(\SR_n^\oplus)^{\BG_m}\to G(\SR_n)\times G(\SR_n)$. 
It defines an action of $G(\SR_n^\oplus)^{\BG_m}$ on the stack $G(\SR_n)$ by two-sided translations. We conjecture that the quotient with respect to this action is canonically isomorphic to $\BT_n^{G,\mu}$ (assuming that $\mu$ is 1-bounded).

\subsubsection{Analogy with the definition of $\Disp_n^{G,\mu}$}
The format from \S\ref{sss:the format} is similar to the format of the definition of  $\Disp_n^{G,\mu}$ given in \S\ref{ss:Disp_n,G,mu}: namely, $G(\SR_n)$ and $G(\SR_n^\oplus)^{\BG_m}$ from \S\ref{sss:the format} are analogs of $G(W_n)$ and $\BP_n^{G,\mu}$ from  \S\ref{ss:Disp_n,G,mu}.

\subsubsection{Remarks}   \label{sss:what this really means}
(i) Let $\BT_{n,\BF_p}^{G,\mu}:=\BT_n^{G,\mu}\times\Spec\BF_p$, $W_{n,\BF_p}:=W_n\times\Spec\BF_p$. Since $W_{n,\BF_p}$ is a scheme of $(\BZ/p^n\BZ)$-algebras, we have the group scheme $G(W_{n,\BF_p})$. The $(\BZ/p^n\BZ)$-algebra stack $\SR_{n,\BF_p}$ is a quotient of 
$W_{n,\BF_p}$ (see \S\ref{sss:the ring stacks over F_p} below), so the conjecture outlined in \S\ref{sss:the format} ultimately represents $\BT_{n,\BF_p}^{G,\mu}$ as a quotient of the scheme $G(W_{n,\BF_p})$ by a certain  
groupoid.

(ii) The situation in mixed characteristic is almost the same. The stack $\SR_n$ can still be represented\footnote{See \cite[\S 8]{future}. The presentation of $\SR_n$ as a quotient of $W_n$ given in \cite[\S 8]{future} depends on the choice of an integer $m\ge\delta_p$, where $\delta_p=0$ for $p>2$ and $\delta_2=1$; it is natural to set $m=\delta_p$. } rather naturally as a quotient of $W_n$, but the ring scheme $W_n$ is not a scheme of $(\BZ/p^n\BZ)$-algebras. However, if $G$ is lifted to a scheme over $\BZ_p$ then $G(W_n)$ is defined, and the conjecture ultimately represents $\BT_n^{G,\mu}$ as a quotient of the scheme $G(W_n)$ by a certain groupoid.

(iii) The idea of representing $\BT_{n,\BF_p}^{G,\mu}$ as a quotient of the scheme $G(W_{n,\BF_p})$ is natural in view of Conjecture~\ref{conjecture}, which says that $\BT_{n,\BF_p}^{G,\mu}$  is a gerbe over $\Disp_n^{G,\mu}$ banded by a group scheme killed by $\Fr^n$.
The pullback of such a gerbe via the map $\Fr^n:\Disp_n^{G,\mu}\to\Disp_n^{G,\mu}$ is trivial, so we get a morphism $\Disp_n^{G,\mu}\to\BT_{n,\BF_p}^{G,\mu}$ such that the composition $\Disp_n^{G,\mu}\to\BT_{n,\BF_p}^{G,\mu}\to\Disp_n^{G,\mu}$ equals $\Fr^n$. The composition $G(W_{n,\BF_p})\to\Disp_n^{G,\mu}\to\BT_{n,\BF_p}$ is faithfully flat (but not smooth), so it represents $\BT_{n,\BF_p}^{G,\mu}$ as a quotient of the scheme $G(W_{n,\BF_p})$ by some groupoid.

\subsection{Conventions}
A ring $R$ is said to be $p$-nilpotent if the element $p\in R$ is nilpotent.

The word ``stack'' will mean a fpqc-stack on the category opposite to that of $p$-nilpotent rings. The final object in the category of such stacks is denoted by $\Spf\BZ_p$; this is the functor that takes each $p$-nilpotent ring to a one-element set.

Ind-schemes and schemes over $\Spf\BZ_p$ are particular classes of stacks. The words ``scheme over $\Spf\BZ_p$'' are understood in the \emph{relative} sense (e.g., $\Spf\BZ_p$ itself is a scheme over $\Spf\BZ_p$).

$W$ will denote the functor $R\mapsto W(R)$, where $R$ is $p$-nilpotent. So $W$ is a ring scheme over $\Spf\BZ_p$. Same for $W_n$.

\subsection{Plan}
To formulate the conjecture, we need the ring stacks $\SR_n,\SR_n^\oplus$. In \S\ref{ss:sR_n and G(sR_n)}-\ref{ss:sR_n oplus} we discuss the easier ring stacks $\sR_n,\sR_n^\oplus$.
In \S\ref{sss:Idea of the definition} we explain the idea of the definition of $\SR_n$ and $\SR_n^\oplus$ (the details are explained in \cite{future}).
In \S\ref{sss:the ring stacks over F_p} we give a simple description of their reductions modulo $p$; this description can be used as a definition.
In \S\ref{ss:conjectural description of BT_n} we formulate the conjecture.

\subsection{The stacks $\sR_n$ and $G(\sR_n)$}  \label{ss:sR_n and G(sR_n)}
It is convenient to define various ring stacks as \emph{cones of quasi-ideals,} see \cite[\S 1.3]{Prismatization}. We will be using this approach.

\subsubsection{The case $n=1$}   \label{sss:sR_1 and G(sR_1)}
Let $n=1$ (so $G$ is a group scheme over $\BF_p$). Then $\sR_1=(\bA^1_{\BF_p})^\prism$ and $G(\sR_1)=G^\prism$, where the superscript $\prism$ stands for \emph{prismatization} in the sense of \cite{Bh, Prismatization}.
Thus $\sR_1=\Cone (W\overset{p}\longrightarrow W)$ (this is a stack of $\BF_p$-algebras), and $G(\sR_1)$ is defined via the procedure called \emph{transmutation} in \cite{Bh}. Namely, $G(\sR_1)$ is the functor that associates to a $p$-nilpotent ring $A$ the groupoid $G (\sR_1 (A))$;
here $\sR_1 (A)$ is an animated $\BF_p$-algebra, so the expression $G (\sR_1 (A))$ makes sense.

\subsubsection{The stacks $\sR_n$}   \label{sss:sR_n and G(sR_n)}
Now let $n$ be any positive integer. Similarly to \S\ref{sss:sR_1 and G(sR_1)}, $\sR_n$ is the stack of $\BZ/p^n\BZ$-algebras defined by
\begin{equation}  \label{e:sR_n}
\sR_n:=\Cone (W\overset{p^n}\longrightarrow W)
\end{equation}
and $G(\sR_n)$ is the functor that associates to a $p$-nilpotent ring $A$ the groupoid $G (\sR_n (A))$.
Since $\sR_n (A)$ is an animated $\BZ/p^n\BZ$-algebra and $G$ is a scheme over $\BZ/p^n\BZ$, the expression $G (\sR_n (A))$ makes sense.
Since $G$ is a group, $G(\sR_n)$ is a group stack. 

\subsubsection{Digression on $n$-prismatization}
Similarly to the case $n=1$ considered in \cite{Bh, Prismatization}, one can deform the ring stack \eqref{e:sR_n} by replacing $p^n$ with $\xi$, where $\xi$ is a primitive Witt vector of degree $n$, which matters only up to multiplication by $W^\times$.
Then one can define the functor of $n$-prismatization denoted by $X\mapsto X^{\prism_n}$ (where $X$ is a $p$-adic formal scheme) so that if $X$ is a scheme over $\BZ/p^n\BZ$ then $X^{\prism_n}=X(\sR_n)$.
We will not follow this path. The following property of $n$-prismatization is somewhat strange: if $X$ is a scheme over $\BZ/p^{n-1}\BZ$ then $X^{\prism_n}=\emptyset$. 

\subsubsection{An economic presentation of $\sR_n$}
The ring stack \eqref{e:sR_n} has the following ``economic'' description:
\begin{equation}  \label{e:sR_n economically}
\sR_n=\Cone (W^{(F^n)}\to W_n),
\end{equation} 
where $W^{(F^n)}:=\Ker (F^n:W\to W)$ and the map $W^{(F^n)}\to W_n$ is the tautological one. In the case $n=1$ this is \cite[Prop.~3.5.1]{Prismatization} or \cite[Cor.~2.6.8]{Bh}.
The argument for any $n$ is similar:
\[
\Cone (W^{(F^n)}\to W/V^nW)=\Cone (V^nW\to W/W^{(F^n)})=\Cone (V^nW\overset{F^n}\longrightarrow W),
\]
and $\Cone (V^nW\overset{F^n}\longrightarrow W)=\Cone (W\overset{F^nV^n}\longrightarrow W)=\Cone (W\overset{p^n}\longrightarrow W)$.

\subsection{The stacks $\sR_n^\oplus$}     \label{ss:sR_n oplus}
\subsubsection{The graded ring scheme $W^\oplus$}
Let $W^\oplus$ be the functor
\[
\{p\mbox{-nilpotent rings}\}\to\{\mbox{graded rings}\}, \quad R\mapsto W(R)^\oplus,
\]
where $W(R)^\oplus$ is the Witt frame (see  \S\ref{sss:Witt frame}). The functor $W^\oplus$ is a $\BZ$-graded ring ind-scheme over $\Spf\BZ_p$.
Each graded component of $W^\oplus$ is an affine scheme over $\Spf\BZ_p$, so by abuse of language we usually call $W^\oplus$ a ring scheme (rather than a ring ind-scheme).

By definition,  $W^\oplus$ is a graded ring subscheme of $W[u,u^{-1}]\times W[t,t^{-1}]$; this subscheme is defined by equations \eqref{e:a' via a}-\eqref{e:a via a'}.

\subsubsection{Remark}  \label{sss:not preserved by F,V}
Let $F,V:W[u,u^{-1}]\to W[u,u^{-1}]$ be the $\BZ_p[u,u^{-1}]$-linear maps extending $F,V:W\to W$. Similarly, one has $F,V:W[t,t^{-1}]\to W[t,t^{-1}]$ and $$F,V:W[u,u^{-1}]\times W[t,t^{-1}] \to W[u,u^{-1}]\times W[t,t^{-1}].$$
\emph{Warning:} the subscheme $W^\oplus\subset W[u,u^{-1}]\times W[t,t^{-1}]$ is \emph{not preserved} by $F$ or by $V$. To see this, look at equations \eqref{e:a' via a}-\eqref{e:a via a'} and recall that the maps $F,V:W\to W$ do not commute.

\subsubsection{The graded ring stack $\sR_n^\oplus$}
Let $\sR_n^\oplus$ be the stack of $\BZ$-graded $(\BZ/p^n\BZ)$-algebras defined by
\begin{equation}  \label{e:sR_n oplus}
\sR_n^\oplus:=\Cone (W^\oplus\overset{p^n}\longrightarrow W^\oplus).
\end{equation}

The embedding $W^\oplus\mono W[u,u^{-1}]\times W[t,t^{-1}]$ induces a homomorphism of $\BZ$-graded $\BZ/p^n\BZ$-algebra stacks
\begin{equation} \label{e:sR_n oplus to Laurent polynomials}
\sR_n^\oplus\to \sR_n[u,u^{-1}]\times \sR_n[t,t^{-1}].
\end{equation}

\subsubsection{An economic presentation of $\sR_{n,\BF_p}^\oplus$}
Let $\sR_{n,\BF_p}^\oplus:=\sR_n^\oplus\times\Spec\BF_p$ (i.e., $\sR_{n,\BF_p}^\oplus$ is the base change of $\sR_n^\oplus$ to $\BF_p$). Similarly, let
$W_{\BF_p}:=W\times\Spec\BF_p$, $W_{n,\BF_p}:=W_n\times\Spec\BF_p$. Good news:
\emph{the subscheme $W_{\BF_p}^\oplus\subset W_{\BF_p}[u,u^{-1}]\times W_{\BF_p}[t,t^{-1}]$ is preserved by $F$ and $V$} (the warning from \S\ref{sss:not preserved by F,V} does not apply because in characteristic $p$ one has $FV=VF=p$).
Thus we have the maps $F,V:W_{\BF_p}^\oplus\to W_{\BF_p}^\oplus$.

Let $W_{n,\BF_p}^\oplus$ be the $n$-truncated Witt frame, i.e., the functor
\[
\{\BF_p\mbox{-algebras}\}\to\{\mbox{graded }\BZ/p^n\BZ \mbox{-algebras}\}, \quad R\mapsto W_n(R)^\oplus,
\]
where $W_n(R)^\oplus$ is as in \S\ref{sss:truncated Witt frame} or \S\ref{sss:2truncated Witt frame}. We claim that
\begin{equation}  \label{e:reduction of sR_n oplus economically}
\sR_{n,\BF_p}^\oplus=\Cone ((W_{\BF_p}^\oplus)^{(F^n)}\to W_{n,\BF_p}^\oplus).
\end{equation}
The proof of \eqref{e:reduction of sR_n oplus economically} is similar to that of \eqref{e:sR_n economically}; it uses the maps $F,V:W_{\BF_p}^\oplus\to W_{\BF_p}^\oplus$.

\subsection{The $\BZ/p^n\BZ$-algebra stacks $\SR_n,\SR_n^\oplus$ and the related group stacks}  \label{ss:the ring stacks over F_p}

\subsubsection{Idea of the definition of $\SR_n,\SR_n^\oplus$}  \label{sss:Idea of the definition}
One can define $\SR_n,\SR_n^\oplus$ similarly to the definitions of $\sR_n,\sR_n^\oplus$ given in \S\ref{ss:sR_n and G(sR_n)}-\ref{ss:sR_n oplus} but replacing $W,W^\oplus$ by certain ring ``spaces'' $\SW,\SW^\oplus$ over $\Spf\BZ_p$. ($\SW$ is called the \emph{space of sheared Witt vectors,} whence the superscript $s$ in the notation for it.) 
One has 
\[
\SW:=\leftlimit{n} W/\hat W^{(F^n)}
\]
(the transition maps are equal to $F$). The quotient $W/\hat W^{(F^n)}$ is understood as an fpqc sheaf (this sheaf is not a scheme).
There is an obvious homomorphism $F:\SW\to\SW$ and a less obvious operator $\hV:\SW\to\SW$. These data have the properties from \S\ref{sss:cC economic}, and $\SW^\oplus$ is obtained from
$(\SW ,F,\hV)$ by applying the functor $\fL$ from Proposition~\ref{p:Lau equivalence}. Finally, let
\[
\SR_n:=\Cone (\SW\overset{p^n}\longrightarrow \SW), \quad \SR_n^\oplus:=\Cone (\SW^\oplus\overset{p^n}\longrightarrow \SW^\oplus);
\]
then $\SR_n$ is an $\BZ/p^n\BZ$-algebra stack, and $\SR_n^\oplus$ is a stack of $\BZ$-graded $\BZ/p^n\BZ$-algebras. Similarly to \eqref{e:sR_n oplus to Laurent polynomials}, we have a homomorphism of $\BZ$-graded $\BZ/p^n\BZ$-algebra stacks
\begin{equation} \label{e:hat sR_n oplus to Laurent polynomials}
\SR_n^\oplus\to\SR_n[u,u^{-1}]\times \SR_n[t,t^{-1}].
\end{equation}

The details are explained in \cite{future}.

\subsubsection{The ring stacks $\SR_{n,\BF_p},\SR_{n,\BF_p}^\oplus$}    \label{sss:the ring stacks over F_p}
The reductions of $\SR_n,\SR_n^\oplus$ modulo $p$ have the following simple descriptions,
which can be used as definitions:
\begin{equation}  \label{e:reduction of sheared sR_n economically}
\SR_{n,\BF_p}=\Cone (\hat W^{(F^n)}_{\BF_p}\to W_{n,\BF_p}),
\end{equation} 
\begin{equation}  \label{e:reduction of sheared sR_n oplus economically}
\SR_{n,\BF_p}^\oplus=\Cone ((\hat W_{\BF_p}^\oplus)^{(F^n)}\to W_{n,\BF_p}^\oplus).
\end{equation} 
Formulas \eqref{e:reduction of sheared sR_n economically}-\eqref{e:reduction of sheared sR_n oplus economically} are parallel to \eqref{e:sR_n economically} and \eqref{e:reduction of sR_n oplus economically}.
Let us note that $(\hat W_{\BF_p}^\oplus)^{(F^n)}$ already appeared in \S\ref{sss:canonical object of PREDISP_n}: for any $\BF_p$-algebra $R$, the group of $R$-points of 
$(\hat W_{\BF_p}^\oplus)^{(F^n)}$ was denoted there by $\hat W^{(F^n)}(R)^\oplus$.

\subsubsection{The group stack $G(\SR_n)$}
As before, let $G$ be a smooth affine group scheme over $\BZ/p^n\BZ$.
Similarly to \S\ref{sss:sR_n and G(sR_n)}, we define $G(\SR_n)$ to be the functor that associates to a $p$-nilpotent ring~$A$ the groupoid $G (\SR_n (A))$.
Since $\SR_n (A)$ is an animated $\BZ/p^n\BZ$-algebra and $G$ is a scheme over $\BZ/p^n\BZ$, the expression $G (\SR_n (A))$ makes sense.
Since $G$ is a group, $G(\SR_n)$ is a group stack. Note that
\begin{equation}  \label{e:G(hat sR_n (A))}
G (\SR_n (A))=\Hom (H^0(G,\cO_G), \SR_n (A)),
\end{equation}
where $\Hom$ is understood in the sense of animated $\BZ/p^n\BZ$-algebras.

\subsubsection{The group stack $G(\SR_n^\oplus)^{\BG_m}$}    \label{sss:group stack corresponding to the graded stack}
Similarly to \eqref{e:G(hat sR_n (A))}, we define the group stack $G(\SR_n^\oplus)^{\BG_m}$ to be the functor that takes a $p$-nilpotent ring $A$ to
$\Hom^{\BG_m} (H^0(G,\cO_G), \SR_n (A))$, where $\Hom^{\BG_m}$ stands for homomorphisms of $\BZ$-graded (animated) $\BZ/p^n\BZ$-algebras.
The grading on $H^0(G,\cO_G)$ is defined using $\mu:\BG_m\to G$ just as in \S\ref{sss:2Definition of BP_n}.

The map \eqref{e:hat sR_n oplus to Laurent polynomials} induces a group homomorphism
\begin{equation}  \label{e:group homomorphism to the product}
G(\SR_n^\oplus )^{\BG_m}\to G(\SR_n[u,u^{-1}])^{\BG_m}\times G(\SR_n[t,t^{-1}])^{\BG_m}=G(\SR_n)\times G(\SR_n).
\end{equation}

\subsection{The conjecture}  \label{ss:conjectural description of BT_n} 
\subsubsection{The quotient of a groupoid by an action of a 2-group}  \label{sss:quotient of a groupoid by a 2-group}
A \emph{$2$-group} is a monoidal category in which all objects and morphisms are invertible. Equivalently,  a  2-group is a 2-groupoid with a single object.

Recall that if a 2-group $\sG$ acts on a groupoid $\sX$ then one defines the  \emph{quotient 2-groupoid} $\tilde\sX=\sX/\sG$ as follows: 

(i) $\Ob\tilde\sX:=\Ob\sX$;

(ii) for $x_1,x_2\in\sX$ let $\MMor_{\tilde\sX} (x_1,x_2)$ be the following groupoid: its objects are pairs 
$$(g,f), \mbox{ where }g\in\sG,\quad f\in\Isom (x_2,gx_1),$$
and a morphism $(g,f)\to (g',f')$ is a morphism $g\to g'$ such that the corresponding morphism $gx_1\to g'x_1$ equals $f'f^{-1}$; 

(iii) the composition functor $\MMor_{\tilde\sX} (x_1,x_2)\times \MMor_{\tilde\sX} (x_2,x_3)\to \MMor_{\tilde\sX} (x_1,x_3)$ comes from the product in $\sG$.

\subsubsection{The quotient of a stack by an action of a group stack}  \label{quotient of a stack by a group stack}
Apply the construction from \S\ref{sss:quotient of a groupoid by a 2-group} at the level of presheaves, then sheafify the result.

\subsubsection{The 2-stack $\BT_n^{G,\mu,?}$}  \label{sss:the conjectural stack over F_p}
The group stack $G(\SR_n)\times G(\SR_n)$ acts on the stack $G(\SR_n)$ by two-sided translations; our convention is that the first copy of 
$G(\SR_n)$ acts by right translations and the second one by left translations, just as in formula~\eqref{e:2-sided translations}.
So the group stack $G(\SR_n^\oplus )^{\BG_m}$ (which depends on $\mu$, see \S\ref{sss:group stack corresponding to the graded stack}) acts on the stack $G(\SR_n)$ via the homomorphism 
\eqref{e:group homomorphism to the product}. Let $\BT_n^{G,\mu,?}$ be the quotient 2-stack, see \S\ref{quotient of a stack by a group stack}. 

The following conjecture is motivated by Conjecture~\ref{conjecture}.

\begin{conj} \label{conjecture in characteristic p}
Suppose that $\mu :\BG_m\to G$ is 1-bounded in the sense of \S\ref{sss:again 1-boundedness}. Then there is a canonical isomorphism $\BT_n^{G,\mu}\iso\BT_n^{G,\mu,?}$, where $\BT_n^{G,\mu}$ is the 1-stack defined in~\cite{GMM}.
\end{conj} 

\subsubsection{Remarks}
(i) Conjecture~\ref{conjecture in characteristic p} implies that if $\mu$ is 1-bounded then $\BT_n^{G,\mu ,?}$ is a 1-stack.

(ii) By \cite[Thm.~D]{GMM}, if $\mu$ is 1-bounded then for every $m\in\BN$ the restriction of $\BT_n^{G,\mu}$ to the category of $(\BZ/p^m\BZ)$-algebras is a smooth algebraic stack over $\BZ/p^m\BZ$. So Conjecture~\ref{conjecture in characteristic p} implies
a similar statement for $\BT_n^{G,\mu ,?}$. This statement is somewhat surprising since the definition of $\BT_n^{G,\mu ,?}$ involves the ind-scheme $\hat W$. It becomes less surprising once you think about formula~\eqref{e:Zink functor} or about the simpler formula
$\Coker (\hat W_{\BF_p}^{(F)}\overset{V}\longrightarrow \hat W_{\BF_p}^{(F)})=\alpha_p$.

(iii) The above formulation of Conjecture~\ref{conjecture in characteristic p} appeared as a result of my conversations with D.~Arinkin and N.~Rozenblyum. My original formulation was more ``elementary'' (in the spirit of \S\ref{sss:original description})
but not elegant enough.

(iv) Two variants of Conjecture~\ref{conjecture in characteristic p} for $n=\infty$ are formulated on slides 12 and 13 of the talk \cite{Bonn-2025}.

\subsubsection{How explicit is $\BT_n^{G,\mu,?}$?}   \label{sss:original description}
Short answer: rather explicit (but not quite explicit). Here are some details.

(i) $\BT_{n,\BF_p}^{G,\mu ,?}$ can be represented  as a quotient of the scheme $G(W_{n,\BF_p})$ by a certain 2-groupoid\footnote{Conjecture~\ref{conjecture in characteristic p} would imply that if $\mu$ is 1-bounded then $\BT_n^{G,\mu ,?}$ is a 1-stack. If so then $\Gamma$ is a 1-groupoid, as said in \S\ref{sss:what this really means}(i).} $\Gamma$. More precisely, one has an explicit action of an explicit group ind-scheme on the scheme $G(W_{n,\BF_p})$ such that $\Gamma$ is a certain quotient of the groupoid corresponding to this action. 

The interested reader can reconstruct the details using \cite{nonsense}. The point is that as explained in the Appendix of \cite{nonsense}, the group stacks $G(\SR_{n,\BF_p}), G(\SR_{n,\BF_p}^\oplus )^{\BG_m}$ naturally come from certain crossed modules in the category of ind-schemes,  and the map \eqref{e:group homomorphism to the product} naturally comes from a certain homomorphism of crossed modules. One can apply \S 2 of~\cite{nonsense} to this homomorphism.

(ii) Assuming\footnote{The reason for this assumption was explained in \S\ref{sss:what this really means}(ii).} that $G$ is lifted to a group scheme over $\BZ_p$, one has a similar situation in mixed characteristic, i.e., a presentation of $\BT_n^{G,\mu,?}$ as a quotient of the scheme $G(W_n)$ by a rather explicit groupoid $\Gamma$. To get such a presentation, represent $\SR_n$ as a quotient of the ring scheme $C=W_n$ (see \cite[\S 8]{future}) and define $\tilde C$ as in \cite[\S 9.2.1]{future}. Then the group ind-scheme $G(\tilde C)^{\BG_m}$ acts on $G(W_n)$ by 2-sided translations, and $\Gamma$ is a certain quotient of the groupoid corresponding to this action.

\bibliographystyle{alpha}

\begin{thebibliography}{BFM}




\bibitem[AM]{AM}
D.~Arinkin and J.~Mundinger, \emph{Cartier duality via Mittag-Leffler modules}, arXiv.2512.13856.


\bibitem[AN]{AN}
B.~Antieau and T.~Nikolaus, \emph{Cartier modules and cyclotomic spectra}, J. Amer. Math. Soc. {\bf 34} (2021), no. 1, 1--78. 



\bibitem[Bh]{Bh} 
B.~Bhatt, \emph{Prismatic F-gauges}. Available at \url{https://www.math.ias.edu/~bhatt/teaching.html}.

\bibitem[BMV]{BMVZ}
B.~Bhatt, A.~Mathew, and V.~Vologodsky, \emph{Sheared Witt vectors}, arXiv:2607.01178, version~1.


\bibitem[BKMVZ]{Sheared}
B.~Bhatt, A.~Kanaev, A.~Mathew, V.~Vologodsky, and M.~Zhang, \emph{Sheared prismatization}, work in progress.



\bibitem[BBE]{BBE}
A.~Beilinson, S.~Bloch, and H.~Esnault, 
\emph{$\varepsilon$-factors for Gauss-Manin determinants,} Mosc. Math. Jour\-nal~{\bf 2} (2002), no.3, 477--532.

\bibitem[BP]{BP}
O.~B\"ultel and G.~Pappas, \emph{$(G,\mu)$-displays and Rapoport-Zink spaces}, J. Inst. Math. Jussieu {\bf 19} (2020), no. 4, 1211--1257. 



\bibitem[Ca]{Ca}
P.~Cartier, \emph{Groupes formels associ\'es aux anneaux de Witt g\'en\'eralis\'es},
C. R. Acad. Sci. Paris S\'er.~A-B,  {\bf 265} (1967), A49--A52.



\bibitem[Co]{Co}
C.~Contou-Carr\`ere, \emph{
Jacobienne locale, groupe de bivecteurs de Witt universel, et symbole mod\'er\'e},
C.~R. Acad. Sci. Paris S\'er. I Math. {\bf 318} (1994), no.8, 743--746.

\bibitem[Da]{Daniels}
P.~Daniels, \emph{A Tannakian framework for $G$-displays and Rapoport-Zink spaces}, Int. Math. Res. Not. 2021, no. 22, 16963--17024.


\bibitem[De]{De}
P.~Deligne, \emph{Le symbole mod\'er\'e}, Publ. Math. IHES {\bf 73} (1991), 147--181.


\bibitem[Dem]{Dem}
M.~Demazure, \emph{Lectures on p-divisible groups}, Lecture Notes in Math., {\bf 302}, Springer-Verlag, Berlin, 1972. 

\bibitem[DG]{DG}
M.~Demazure and P.~Gabriel, \emph{Groupes alg\'ebriques. Tome I: G\'eom\'etrie alg\'ebrique, g\'en\'eralit\'es, groupes commutatifs}, North-Holland Publishing Co., Amsterdam, 1970. 


\bibitem[D20]{Prismatization} 
V.~Drinfeld, \emph{Prismatization},  Selecta Mathematica New Series {\bf 30} (2024), article no.~49, https://doi.org/10.1007/s00029-024-00937-3.



\bibitem[D23]{Dr}
V.~Drinfeld, \emph{On Shimurian generalizations of the stack ${\rm BT}_1\otimes{\bf F}_p$}, Journal of Mathematical Physics, Analysis, Geometry, {\bf 21} (2025), No. 3, 276--301.
Also available as arXiv:2304.11709.

\bibitem[D25a]{nonsense} 
V.~Drinfeld, \emph{On the quotient of a groupoid by an action of a 2-group}, arXiv:2504.21764.

\bibitem[D25b]{future} 
V.~Drinfeld, \emph{Ring stacks conjecturally related to the stacks $\BT_n^{G,\mu}$}, arXiv:2510.04958.

\bibitem[D25c]{Bonn-2025} 
\emph{The ring space $\SW$ and Barsotti-Tate groups}, a talk whose recording and slides are available at \url{https://archive.mpim-bonn.mpg.de/id/eprint/5179}



\bibitem[GM]{GMM}
Z.~Gardner and K.~Madapusi, \emph{An algebraicity conjecture of Drinfeld and the moduli of $p$-divisible groups}, arXiv:2201.06124, version 3. 



\bibitem[Gr]{Gr}
A.~Grothendieck, \emph{Groupes de Barsotti-Tate et cristaux de Dieudonn\'e}, 
S\'eminaire de Math\'ematiques Sup\'erieures, No. 45 (1970), Les Presses de l'Universit\'e de Montr\'eal, Montreal, Quebec, 1974. 
Available at \url{www.grothendieckcircle.org}.

\bibitem[H]{H}
M.~Hazewinkel, \emph{Formal groups and applications}, Pure Appl. Math. {\bf 78}, Academic Press, New York-London, 1978. 


\bibitem[KM]{KM}
C.~Kothari and J.~Mundinger, \emph{Dieudonn\'e theory for $n$-smooth group schemes}, arXiv:2408.15333, version~1.

\bibitem[LZ07]{Langer-Zink}
A.~Langer and Th.~Zink, \emph{De Rham-Witt cohomology and displays}.
Documenta Math. {\bf 12} (2007), 147--191.


\bibitem[L13]{Lau13} E.~Lau, \emph{Smoothness of the truncated display functor},
J. Amer. Math. Soc. {\bf 26} (2013), no. 1, 129--165.


\bibitem[L21]{Lau21} E.~Lau, \emph{Higher frames and $G$-displays}, Algebra Number Theory {\bf 15} (2021), no. 9, 2315--2355. 

\bibitem[L25]{Lau25} E.~Lau, \emph{The Shimurian BT stack is a gerbe over truncated displays}, arXiv:2510.23207, version 1.


\bibitem[LZ]{LZ} 
 E.~Lau and Th.~Zink, \emph{Truncated Barsotti-Tate groups and displays},
J. Inst. Math. Jussieu {\bf 17} ( 2018), no.3, 541--581.


\bibitem[Il85]{Il} L.~Illusie, D\'eformations de groupes de Barsotti-Tate (d'apr\`es A. Grothendieck), In: Seminar on arithmetic bundles: the Mordell conjecture (Paris, 1983/84), 
Ast\'erisque {\bf 127} (1985), 151--198. 

\bibitem[dJ]{dJ} A.~J.~de Jong, \emph{Finite locally free group schemes in characteristic p and Dieudonn\'e modules},
Invent. Math. {\bf 114} (1993), no. 1, 89--137. 


\bibitem[Me72]{Me72}
W.~Messing, \emph{The crystals associated to Barsotti-Tate groups: with applications to abelian schemes},
Lecture Notes in Mathematics, Vol. {\bf 264}, Springer-Verlag, Berlin-New York, 1972. 


\bibitem[N]{N}
P.~Norman, \emph{An algorithm for computing local moduli of abelian varieties,} Ann. of Math. (2), {\bf 101} (1975), 499--509.



\bibitem[SGA3]{SGA3} \emph{Sch\'emas en groupes, I: Propri\'et\'es g\'en\'erales des sch\'emas en groupes},
S\'eminaire de G\'eom\'etrie Alg\'ebrique du Bois Marie 1962/64 (SGA 3), Dirig\'e par M.~Demazure et A.~Grothendieck, Lecture Notes in Mathematics, Vol. {\bf 151}, Springer-Verlag, Berlin-New York 1970. Reedited by P.~Gille and P.~Polo, Documents Math\'ematiques (Paris), {\bf 7}, Soci\'et\'e Math\'ematique de France, Paris, 2011. 

\bibitem[W]{Wedhorn}
T.~Wedhorn, \emph{The dimension of Oort strata of Shimura varieties of PEL-type.} In: Moduli of Abelian Varieties, Progress in Mathematics {\bf 195}, 441--471, Birkh\"{a}user, Basel, 2001.

\bibitem[Zi01]{Zink-short}
Th.~Zink, \emph{A Dieudonn\'e theory for p-divisible groups.} In: Class field theory -- its centenary and prospect (Tokyo, 1998), Adv. Stud. Pure Math. {\bf 30}, 139--160, Mathematical Society of Japan, Tokyo, 2001.


\bibitem[Zi02]{Zink}
Th.~Zink, \emph{The display of a formal $p$-divisible group.} In: Cohomologies $p$-adiques et applications arithm\'etiques, I,  Ast\'{e}risque {\bf 278} (2002), 127--248.


\end{thebibliography}

\end{document}